%% file: paper_am_5.tex
\documentclass[11pt]{article}
\usepackage{latexsym,amsmath,url,caption2,epsfig}
\usepackage{amsfonts,euscript}
\usepackage{graphicx}
\usepackage{amsmath}
\usepackage{amssymb}
\usepackage{amsfonts}
\usepackage{amsthm}
\usepackage{mathrsfs}
\usepackage{fullpage}
\usepackage[all]{xy}
\usepackage{rf}
\usepackage{epstopdf}
\usepackage{subfigure}
\usepackage{rotating}
\usepackage{appendix}
\usepackage{url}
\usepackage{color}
\usepackage{algorithm}
\usepackage{algorithmicx}
\usepackage{algpseudocode}
\usepackage{setspace}
\usepackage{listings}
\usepackage{color}
\usepackage{tikz}
\usepackage{graphics}
\usepackage[english]{babel}
\usepackage{datetime}
\usepackage{booktabs}
\usepackage{dsfont}
\usepackage{rotfloat}
\usepackage{framed}
\usepackage[utf8]{inputenc}
\usepackage[numbers]{natbib}
\usepackage[breaklinks, colorlinks, citecolor=blue, urlcolor=blue, pdfborder={0 0 1},
citebordercolor={1 1 1}]{hyperref}

\def\hf{\hat{f}}
\def\hp{\hat{p}}
\def\hh{h^{\prime}}
\def\hInfo{\hat{{\rm I}}}

\usetikzlibrary{arrows,shapes}
\settimeformat{ampmtime} \numberwithin{equation}{section}
\newcommand{\be}{\begin{equation}}
\newcommand{\ee}{\end{equation}}
\newcommand{\ba}{\begin{aligned}}
\newcommand{\ea}{\end{aligned}}
\renewcommand{\thesubsection}{\arabic{section}.\arabic{subsection}}
\renewcommand{\theequation}{\arabic{section}.\arabic{equation}}

\newcommand{\one}{\textbf{1}}
\newcommand{\zero}{\textbf{0}}
\newcommand{\opt}{^*}
\newcommand{\difI}{\delta_{W,\eps}}

\makeatletter
\newcommand*{\rom}[1]{\expandafter\@slowromancap\romannumeral #1@}
\makeatother

\newcommand{\dezero}{\varepsilon}
\newcommand{\event}{\Lambda}
\newcommand*\rmd{\mathop{}\!\mathrm{d}}
\newcommand{\asto}{\xrightarrow{a.s.}}

\def\hf{\hat{f}}
\def\hp{\hat{p}}
\def\hh{h^{\prime}}
\def\hInfo{\hat{{\rm I}}}

\def\TV{\mbox{\tiny\rm TV}}

\input{andrea_def.tex}

\begin{document}

\title{Adapting to Unknown Noise Distribution in Matrix Denoising}

\author{Andrea Montanari\thanks{Department of Electrical Engineering and
  Department of Statistics,    Stanford University} \and Feng
Ruan\thanks{Department of Statistics,    Stanford University} 
\and Jun Yan\thanks{Department of Statistics,    Stanford University}}

\date{October 5, 2018}
\maketitle

\begin{abstract}
We consider the problem of estimating an unknown matrix $\bX\in \reals^{m\times n}$, from observations $\bY = \bX+\bW$
where $\bW$ is a noise matrix with independent and identically distributed entries, as to minimize estimation error 
measured in operator norm.  Assuming that the underlying signal $\bX$ is low-rank and incoherent with respect to the canonical basis, 
we prove that minimax risk is equivalent to $(\sqrt{m}\vee\sqrt{n})/\sqrt{\Info_W}$ in the high-dimensional limit $m,n\to\infty$,
where $\Info_W$ is the Fisher information of the noise.
Crucially, we develop an efficient procedure that achieves this risk, adaptively over 
the noise distribution (under certain regularity assumptions).

Letting  $\bX = \bU\bSigma\bV^{\sT}$ --where $\bU\in \reals^{m\times r}$, $\bV\in\reals^{n\times r}$ are orthogonal, and $r$ is
kept fixed as $m,n\to\infty$-- we use our method to estimate $\bU$, $\bV$.  Standard spectral methods provide  non-trivial estimates of
the factors $\bU,\bV$ (weak recovery)  only if the singular values of $\bX$ are larger than $(mn)^{1/4}\Var(W_{11})^{1/2}$. 
We prove that the new approach achieves 
weak recovery down to the the information-theoretically optimal threshold $(mn)^{1/4}\Info_W^{1/2}$.
\end{abstract}

\section{Introduction and main result}

Let $\bX\in\reals^{m\times n}$ be an unknown signal, and assume we are given observations of its entries corrupted by 
additive noise:
\begin{align}
\bY = \bX +\bW \, ,\label{def:model}
\end{align}
where $\bW = (W_{ij})_{i\le m, j\le n}$ has i.i.d. entries $W_{ij} \sim p_W$. We would like to estimate $\bX$ in operator norm,
namely construct an estimator $\hbX:\reals^{m\times n}\to \reals^{m\times n}$ as to minimize $\E\{\|\hbX(\bY)-\bX\|_{\op}\}$.
Random matrix theory \cite{Guionnet,bai2010spectral} characterizes the
unbiased estimator $\hbX_{\sub}(\bY) = \bY$ as $m,n\to\infty$ (under suitable tail conditions on the law of the noise):
\begin{align}
\E\Big\{\big\|\hbX_{\sub}(\bY)-\bX\big\|_{\op}\Big\} = (\sqrt{m}+\sqrt{n})\sqrt{\Var(W_{11})} +O(1)\, . \label{eq:Unbiased}
\end{align}

Anderson's lemma \cite{anderson1955integral} implies that the unbiased estimator is minimax optimal when the noise is Gaussian.
However, for more general noise distributions, the result (\ref{eq:Unbiased}) leads to a natural statistical question: 
\begin{quote}
\begin{center}
\emph{Can we improve over the error of the unbiased estimator?}
\end{center}
\end{quote}

This question was answered positively in several cases, under the assumption that $\bX$ is low rank and the noise distribution $p_W$ is known
\cite{deshpande2015finding,lesieur2015mmse,perry2016optimality}. It is quite easy to understand a mathematical mechanism leading to this answer \cite{lesieur2015mmse,perry2016optimality} (a somewhat different mechanism is studied in \cite{deshpande2015finding}).
 In many cases of interest (e.g. when
$\bX = \bu\bv^{\sT}$ with $\bu\in\reals^m$, $\bv\in\reals^n$ incoherent with respect to the canonical basis), the entries of $\bX$
are much smaller than the entries of the noise matrix $\bW$. Imagine to apply a non-linear denoiser $f:\reals\to\reals$ component-wise
to $\bX$ to obtain $\hbX(\bY) = f(\bY)$. By Taylor expansion we have
\begin{align}
\hbX(\bY) \approx  f(\bW) +f'(\bW) \odot \bX  \approx \E\{f'(W)\} \bX +f(\bW) \, ,\label{eq:TayolorExpansion}
\end{align}
where $\odot$ denotes Hadamard (entry-wise) product, and $W\sim p_W$ is a scalar with the same distribution as the entries of $\bW$.
Notice that $f(\bW)$ is a matrix with i.i.d. entries. Therefore, for any nonlinearity such that $\E\{f'(W)\} = 1$,  $\E\{f(W)\} = 0$, 
we expect $\E\big\{\big\|\hbX(\bY)-\bX\big\|_{\op}\big\} = (\sqrt{m}+\sqrt{n}) \sqrt{\E\{f(W)^2\}} +O(1)$. This suggests to choose
the nonlinearity $f$ by solving the optimization problem
\begin{align}
\mbox{minimize}&\;\;\; \E\{f(W)^2\}\, ,\\
\mbox{subject to}&\;\;\; \E\{f(W)\} = 0\, ,\;\;\; \E\{f'(W)\} = 1\, .
\end{align}
Assuming $p_W$ to have a differentiable density (also denoted by $p_W$), a simple application of Cauchy-Schwartz inequality implies that the
optimal $f(x)$ is given by
\begin{equation}
\label{eqn:def-score-information}
 f_W(x)=-\frac{1}{\Info_W}\, \frac{p_W'(x)}{p_W(x)} \, ,\;\;\;\;\; 
\Info_W=\int_{\mathbb{R}}\frac{(p^\prime_W(x))^2}{p_W(x)}\de x\, .
\end{equation}
Notice that $\Info_W$ is the Fisher information for the location family $\{p_w(W+\theta)\}_{\theta\in\reals}$.
This denoiser would achieve $\E\big\{\big\|\hbX(\bY)-\bX\big\|_{\op}\big\} = (\sqrt{m}+\sqrt{n}) \Info_W^{-1/2} +O(1)$.

Unfortunately, the above argument does not provide a concrete statistical procedure, since the noise distribution $p_W$ is
typically unknown to the data analyst. Our main result is that the ideal estimation error can be 
achieved without prior knowledge of $p_W$. In order to state this result, we define a class of matrices $\cF_{m,n}(r,M,\eta)$ that formalizes  
the assumption that $\bX$ has small entries and is low-rank:
\begin{align}
\cF_{m,n}(r, M, \eta) \equiv \big\{&\bX\in\reals^{m\times n}: \;\; \rank(\bX) \leq r,~
											\opnorm{\bX} \le M(m \vee n)^{1/2},~
											\opnormbig{\bX \odot \bX} \le M(m \vee n)^{1/2-\eta}, 
											\nonumber \\
	&~~~~~~~~\norm{\bX}_{\ell_2 \to \ell_\infty} \le Mn^{1/2 - \eta}~~\text{and}~~
		\normbig{\bX^{\sT}}_{\ell_2 \to \ell_\infty} \le Mm^{1/2 - \eta}
		 \big\}\, .
\end{align}
(Recall that $\norm{\bX}_{\ell_2 \to \ell_\infty}$ is the maximum $\ell_2$ norm of any row of $\bX$.)

We also denote by $\cF_{m,n}(M, \eta)$ the same class, where the rank constraint is removed, i.e. $\cF_{m,n}(M,\eta) = \cF_{m,n}(m\wedge n,M,\eta)$.
Our first result concerns estimation in the class $\cF_{m,n}(M,\eta)$.
\begin{theorem}\label{thm:Main0}
For each $m,n\in\naturals$, there exists an estimator $\hbX^{(*)}:\reals^{m\times n}\to\reals^{m\times n}$ such that the following holds.
Assume that the noise distribution is absolutely continuous with
respect to Lebesgue measure, with density $p_W$ satisfying the
following assumptions:
\begin{itemize}
\item[{\sf A0.}] The Fisher information $\Info_{W}$ is finite.
\item[{\sf A1.}] $\int |w|^2 p_W(w) \, \de w\le M_1$ for some constant $M_1$.
\item[{\sf A2.}] $p_W\in C^3(\reals)$, with derivatives $p^{(\ell)}_W$ satisfying
\begin{align}
\norm{p_W^{(\ell)} (\,\cdot\,)}_\infty \leq M_2 ~~\text{for all $\ell\in \{0,1,2\}$.}
\end{align}
\end{itemize}
Then, letting $\bY = \bX+\bW$, with $(W_{ij})_{i\le m,j\le n}\sim_{iid}p_W$, assuming $m \asymp n$, we have :
\begin{align}
\sup_{\bX\in\cF_{m,n}(M,\eta)}\E\big\{\|\hbX^{(*)}(\bY)-\bX\|_{\op}\big\}  \le    \frac{\sqrt{m}+ \sqrt{n}}{\sqrt{\Info_W}}  +o((m\vee n)^{1/2})\, . \label{eq:AdaptiveUB_0}
\end{align}
\end{theorem}

We next consider the rank-constrained class $\cF_{m,n}(r,M,\eta)$. 
\begin{theorem}\label{thm:Main}
For each $m,n\in\naturals$, there exists an estimator $\hbX:\reals^{m\times n}\to\reals^{m\times n}$ such that the following holds.
Under the same assumptions of Theorem \ref{thm:Main0}, letting $\bY = \bX+\bW$, with $(W_{ij})_{i\le m,j\le n}\sim_{iid}p_W$, assuming $m \asymp n$, we have :
\begin{align}
\sup_{\bX\in\cF_{m,n}(r,M,\eta)}\E\big\{\|\hbX(\bY)-\bX\|_{\op}\big\}  \le    \frac{\sqrt{m}\vee \sqrt{n}}{\sqrt{\Info_W}}  +o((m\vee n)^{1/2})\, .  \label{eq:AdaptiveUB}
\end{align}
\end{theorem}
\begin{remark}
The upper bound of Eq.~(\ref{eq:AdaptiveUB_0}) improves over the risk of the unbiased estimator in Eq.~(\ref{eq:Unbiased}) 
in that the factor $\sqrt{\Var(W)}$ is replaced by $1/\sqrt{\Info_W}$, which is strictly smaller in any case except for Gaussian noise (in which case the two bounds are equal).
This is achieved by entry-by-entry denoising of the matrix $\bY$, following the ideas described above.

The bound  of Eq.~(\ref{eq:AdaptiveUB}) further improves the factor $(\sqrt{m}+\sqrt{m})$ to $\sqrt{m}\vee \sqrt{n}$ by exploiting the low-rank 
structure in $\bX$. Entrywise denoising is followed by suitable
shrinkage of  the singular values of $\bY$. Singular values shrinkage is a well studied topic, see for instance \cite{shabalin2013reconstruction,gavish2014optimal,chatterjee2015matrix,donoho2014minimax,donoho2018optimal}. However, by itself, it does not achieve the minimax error for non-Gaussian noise.
\end{remark}

Our second main result establishes that no estimator can do substantially better than what guaranteed by , in the rank constrained case, even with knowledge of $p_W$.
\begin{theorem}\label{thm:LowerBound}
Assume the density $p_W$ to be weakly differentiable and $\Info_W$ to be finite, and assume $m=m(n)$ be such that
$\lim_{n\to\infty}m(n)/n =\gamma$. Then
\begin{equation}
\lim_{n\to\infty}\inf_{\hbX(\,\cdot\,)}\sup_{\bX\in\cF_{m,n}(r,M,\eta)}
	\frac{1}{(mn)^{1/4}}\E\big\{\|\hbX(\bY)-\bX\|_{\op}\big\}  \ge  
		\frac{\gamma^{1/4}\vee \gamma^{-1/4}}{\sqrt{\Info_W}} -o_M(1). 
\end{equation}
(Here the infimum is taken over all measurable functions $\hbX:\reals^{m\times n}\to\reals^{m\times n}$, and $o_M(1)$ is a quantity vanishing 
as $M\to\infty$.)
\end{theorem}

\begin{remark}
Roughly speaking, Theorem \ref{thm:LowerBound} states that, in the rank constrained class, any estimator
has worst case error lower bounded as $\sup_{\bX\in\cF_{m,n}(r,M,\eta)}\E\big\{\|\hbX(\bY)-\bX\|_{\op}\big\}  \gtrsim (\sqrt{m}\vee \sqrt{n})/\sqrt{\Info_W}$.
This lower bound obviously applies to the larger class $\cF_{m,n}(M,\eta)$, but it falls short of the upper bound of Theorem \ref{thm:Main0}
by a factor $(\sqrt{m}+\sqrt{n})/ (\sqrt{m}\vee \sqrt{n})\le 2$ (for large $m,n$).

It is an open problem to close this gap between upper and lower bounds in the rank-unconstrained class.  
\end{remark}

The rest of this paper is organized as follows. Section \ref{sec:Related} overviews related work. 
Section \ref{sec:Construction} describes the construction of the
adaptive estimator, and states a more detailed version of Theorem \ref{thm:Main}
Section \ref{sec:LowRank} applies the adaptive denoiser to the problem of estimating the top singular subspaces of a low-rank 
matrix $\bX$. Section \ref{sec:Numerical} illustrates the performance of our method through a simulation study. 
Finally, Sections \ref{sec:Techniques_LB} and \ref{sec:Techniques_UB} outline the proofs of our main theorems, with details deferred to the appendices.

\section{Related work}
\label{sec:Related}

The problem of estimating a noisy matrix in operator norm is central in 
principal component analysis and covariance estimation
\cite{vershynin2012close}, matrix completion \cite{candes2010matrix,keshavan2010matrix}, 
graph localization \cite{singer2008remark,javanmard2013localization}, group synchronization 
\cite{singer2011angular,wang2013exact,javanmard2016phase},
community detection \cite{moore2017computer,abbe2018community}, gene expression analysis \cite{ringner2008principal}, and a number of other applications. 

A large part of the theoretical literature has focused on the Gaussian spiked model, or on sparse covariance structures.
In the last setting, several estimators are known that achieve minimax error rates to varying degree of generality 
\cite{el2008operator,bickel2008covariance,cai2010optimal,cai2013sparse}.
These approaches are often based on entry-wise thresholding of the empirical covariance.
While sparsity is a useful assumptions for some datasets, it is not warranted for other applications (including matrix completion, group synchronization, localization,
and community detection).
Accordingly, we do not assume $\bX$ to be sparse, but rather incoherent with respect to the canonical basis. Our focus is 
on achieving adaptivity with respect to the unknown noise distribution. In order to investigate the effect of the noise distribution, we cannot limit
ourselves to determine error rates, but instead we need to pinpoint the asymptotic estimation error up to sub-leading error terms.

Our work is closely related to results on the spiked covariance model. Within the spiked model, we observe $m$-dimensional vectors $(\by_i)_{i\le n}\sim \normal(0,\bSigma)$
with covariance $\bSigma = \bU\bU^{\sT}+\id_m$, with $\bU\in\reals^{m\times r}$ a matrix describing the $r\ll m$ spikes.
If $\by_1, \dots,\by_m$ are stacked as columns of a matrix $\bY$, we can equivalently write $\bY = \bX+\bW$, where $\bX = \bU\bV^{\sT}$,
$\bV\in\reals^{n\times r}$ has entries $(V_{ij})_{i\le n,j\le r}\sim_{iid}\normal(0,1/m)$, and $(W_{ij})_{i\le m, j\le n}\sim \normal(0,1)$. 
The low-rank component $\bU\bU^{\sT}$ of the covariance is extracted using principal component analysis, which is known be optimal, 
but is not asymptotically consistent in the high-dimensional regime $n\asymp m$ unless $\sigma_{\rm min}(\bU)\to \infty$ \cite{johnstone2009consistency}. 
Random matrix theory has been used to determine asymptotic detection (or weak recovery) thresholds, asymptotic estimation errors, as
well optimal hypothesis testing procedures \cite{BBAP05,BS06,paul2007asymptotics,onatski2013asymptotic}.

A substantial literature studies matrix denoising under a Gaussian
noise model \cite{shabalin2013reconstruction,gavish2014optimal,chatterjee2015matrix,donoho2014minimax,donoho2018optimal}.
By rotational invariance, in this case optimal denoising is achieved by a form of singular value shrinkage, 
i.e. taking the singular value decomposition of
the observed matrix $\bY$ and and applying a nonlinear function to its singular values.
Singular value shrinkage comes with error guarantees also under non-Gaussian noise models
(among others, \cite{chatterjee2015matrix} studies a broad class of models). However, 
it is suboptimal in this more general case. 

Singular value shrinkage is also suboptimal when the signal $\bX=\bU\bV^{\sT}$ is of rank 
$r\ll m,n$, and the factors $\bU,\bV$ have additional structure. Sparse principal 
component analysis is a prominent example of this phenomenon \cite{johnstone2009consistency,zou2006sparse,d2005direct}.
 More recently, the case in which the rows of $\bU$, $\bV$ are i.i.d. draws from probability distributions $p_U$, $p_V$ on $\reals^r$
has attracted attention, see e.g.  
\cite{lesieur2017constrained,deshpande2014information,deshpande2016asymptotic,krzakala2016mutual,barbier2016mutual,lelarge2016fundamental,miolane2017fundamental,montanari2017estimation}.

As suggested by the simple argument in the introduction, singular value shrinkage (and other simple spectral methods) 
also becomes suboptimal when 
the noise is non-Gaussian. Namely, by applying a suitable entry-wise nonlinearity to the data matrix $\bY$, we can reduce the operator norm
of the noise, without affecting the signal. This phenomenon was investigated in \cite{deshpande2015finding} in the context of the hidden 
sub-matrix problem, and in \cite{krzakala2016mutual} in the context of rank-one matrix estimation for symmetric matrices.
In particular, \cite{krzakala2016mutual} assumed the rank-one signal $\bX = \sqrt{n}\bu\bu^{\sT}$ to be incoherent with the standard basis and identified 
a weak recovery phase transition when the Fisher information crosses $\Info_W=1$ (assuming the normalization $\|\bu\|_2=1$).
Subsequently, Perry and co-authors \cite{perry2016optimality} proved that the same threshold correspond to
a phase transition in hypothesis testing. If $\Info_W<1$ it is impossible to distinguish with  vanishing error probability  
between data from the spiked model $\bY  = \sqrt{n}\bu\bu^{\sT}+ \bW$, and pure noise $\bY=\bW$. If $\Info_W>1$,
instead the two hypotheses can be distinguished with vanishing error probability as $n\to\infty$.

\section{Construction of the adaptive estimators}
\label{sec:Construction}

In this section we describe the estimators $\hbX^{(*)}(\bY)$ and $\hbX(\bY)$ of Theorems 
\ref{thm:Main0} and \ref{thm:Main}.

Equation (\ref{eqn:def-score-information}) suggest to estimate the density of the noise $p_W(\,\cdot\,)$, 
in order to construct an approximation of the denoiser $f_W(\,\cdot\,)$. The challenge is of course that we do not have samples from 
$p_W$ but only matrix entries $(Y_{ij})_{i\le m, j\le n}$. However, it turns out that we can use these as surrogates for the noise values
$(W_{ij})_{i\le m, j\le n}$, provided we introduce a suitable noise correction.

Let $K:\reals\to\reals$ be a first order non-negative kernel, i.e, 
\begin{equation}
\label{eqn:def-1st-order-kernel}
K(z)\geq 0,~~\int_{\R}K(z)\, \de z=1~~\text{and}~~\int_{\R}zK(z)\,\de z=0\, .
\end{equation}
Throughout, we assume  $K(\, \cdot \, )$ to satisfy the following smoothness
condition. For some constant $M>0$ (independent of the model parameters $\bX$, $p_W$, $m,n$) and  for $\ell \in \{0,\dots,3\}$, we have
\begin{align}
\norm{K^{(\ell)}}_{\infty }\le M\, ,\;\;\; \int_{\R}\left(
K^{(l)}(z)\right) ^{2}\de z \leq M\, ,\;\;\;\;\int_{\R}z^{2}K(z)\, \de z\leq M\, .
\end{align}
Many standard kernels satisfy these conditions, a simple example being the Gaussian kernel 
$K(z) = (2\pi )^{-1/2}\exp (-z^{2}/2)$.

We further define the function 
\begin{align}
H(\sigma) =
\begin{cases}
\sqrt{ (\sigma + \gamma^{-1/2} \sigma^{-1})(\sigma+ \gamma^{1/2} \sigma^{-1})}
 & \text{if } \sigma \ge 1\text{,} \\ 
\gamma^{1/4} + \gamma^{-1/4} & \text{otherwise.}%
\end{cases}\label{eq:Hsigma}
\end{align}
Note that $\sigma\mapsto H(\sigma)$ is strictly increasing on $(1,\infty)$, and hence
we can define its inverse  $H^{-1}(y)$ on $(H(1),\infty)$ with $H(1) = \gamma^{1/4} + \gamma^{-1/4}$.

Our adaptive estimators depends on parameters $h_n,\hh_n,\eps, \delta>0$ and  proceeds as follows:
\begin{enumerate}
\item Compute the average entry $\oY \equiv \sum_{i\le m, j\le n} Y_{ij}/(mn)$.
\item Compute the corrected kernel estimates
\begin{align}
&\hat{p}_{Y}(x) = \frac{1}{mnh_{n}}\sum_{i\le m,j\le n}K\left( \frac{Y_{ij}-\oY-x}{h_{n}}\right),~ 
&\hat{p}_{Y}^\prime(x) = \frac{1}{mn (h_{n}^\prime)^{2}}
\sum_{i\le m, j\le n}K^{\prime }\left( \frac{Y_{k,l}-\oY-x}{h_{n}^{\prime }}\right).
\end{align}
\item Define the denoiser $\hf_{Y,\eps}(\,\cdot\, )$, and the estimated Fisher information
\begin{align}
\hf_{Y,\eps}(x) = - 
	\frac{\hp_Y'(x)}{\hp_Y(x)+\eps}\, ,\;\;\;\;\;\; \hInfo_{W,\eps} \equiv 
		\frac{1}{mn}\sum_{i\le m,j\le n} \left(\frac{\hp_Y'(Y_{ij} - \bar{Y})}{\hp_Y(Y_{ij} - \bar{Y})+\eps}\right)^2 + \eps\, .
\label{eq:hf_def}
\end{align}
(We will omit the dependence on $\eps$ whenever clear from the context.)
\item Construct the matrix obtained by applying this denoiser entry-wise to $\bY$: 
$\hbX^{(0)}(\bY;h_n,\hh_n,\eps) = \hf_{Y}(\bY)$, and
compute its singular value decomposition
\begin{align}
\label{eqn:svd-of-hbX0}
\hbX^{(0)} = (mn)^{1/4}\hbU\hbSigma^{(0)}\hbV^{\sT}\, .
\end{align}
\item Construct the diagonal matrix $\hbSigma$ with entries
\begin{equation}
\label{eqn:def-hbSigma}
\hSigma_{i,i} = \begin{cases}
\hat{\Info}_{W, \eps}^{-1/2}
	H^{-1}(\hat{\Info}_{W, \eps}^{-1/2}\hSigma^{(0)}_{i,i}) & 
		\mbox{ if $\hSigma^{(0)}_{i,i}\ge (1+\delta)H(1) \hat{\Info}_{W, \eps}^{1/2}$,}\\
0& \mbox{ otherwise.}
\end{cases}
\end{equation}
\item Return
\begin{align}
\hbX^{(*)} (\bY;h_n,\hh_n,\eps)  &= \hat{\Info}_{W, \eps}^{-1} \,  \hbX^{(0)}(\bY;h_n,\hh_n,\eps)  \, ,\\
\hbX(\bY;h_n,\hh_n,\eps,\delta)  &= (mn)^{1/4}\hbU\hbSigma\hbV^{\sT}\, .
\end{align}
\end{enumerate}
(In what following we will often omit the  we will often omit the dependence the on the parameters
$h_n,\hh_n,\eps,\delta$.)

\begin{remark}
Note that $\hat{p}_{Y}(x)$, $\hat{p}_{Y}^{\prime}(x)$ are estimators for the noise density $p_W$ and its first derivative. They differ from standard
kernel estimators because of the shift $\oY$. This is designed to correct the error made because $Y_{ij}\neq W_{ij}$. If
all the entries of $\bX$ were equal to $\oX$, we would have $W_{ij} = Y_{ij}-\oX$. The correction we use simply 
replaces $\oX$ by $\oY$. It turns
out that this simple correction is sufficient in the large $m,n$ limit.
\end{remark}

The next statement provides a more detailed version  of Theorems \ref{thm:Main0}.
%
\begin{theorem}
\label{theorem:upper-bound}
Assume that conditions {\sf A0}, {\sf A1}, {\sf A2} of Theorem \ref{thm:Main0} hold. Further, 
let $h_n = n^{-\eta_1}$ and $h_n^\prime = n^{-\eta_2}$ for some 
$\eta_1 \in \left(1/4, 1\right)$ and $\eta_2 \in \left(1/4, 1/3\right)$.
Then the estimators  $\hbX^{(*)}(\,\cdot\,)$ and  $\hbX(\,\cdot\,)$  defined above satisfies
\begin{align}
	&\limsup_{n \to \infty, m/n \to \gamma}\sup_{\bX\in\cF_{m,n}(M,\eta)}
		\frac{1}{(mn)^{1/4}}\E \norm{\hbX^{(*)}(\bY;h_n,h'_n,\eps) - \bX}_{\op}
		\le \big(\gamma^{1/4}+ \gamma^{-1/4}\big)\, \Info_W^{-1/2}+o_{\eps,\delta}(1)\, ,\label{eqn:upper-bound-0}\\
	 	&\limsup_{n \to \infty, m/n \to \gamma}\sup_{\bX\in\cF_{m,n}(r,M,\eta)}
		\frac{1}{(mn)^{1/4}}\E \norm{\hbX(\bY;h_n,h'_n,\eps,\delta) - \bX}_{\op}
		\le \max\{\gamma^{1/4}, \gamma^{-1/4}\} \Info_W^{-1/2}+o_{\eps,\delta}(1)\, ,\label{eqn:upper-bound}
\end{align}
where $\lim_{\delta\to 0}\lim_{\eps\to 0}o_{\eps,\delta}(1) =0$.
\end{theorem}
\section{Singular space recovery and additional results}
\label{sec:LowRank}

An important application of the denoising method introduced above is to estimate latent factors in 
high-dimensional data. We denote the singular value decomposition of the signal $\bX$ by
\begin{align}
\bX =(mn)^{1/4} \bU\, \bSigma\, \bV^{\sT}\, ,\label{eq:SVD_X}
\end{align}
where $\bU\in\reals^{m\times r}$, $\bV\in\reals^{n\times r}$ are orthogonal matrices (i.e. satisfying 
$\bU^{\sT}\bU = \bV^{\sT}\bV = \id_r$), and $\bSigma\in\reals^{r\times r}$ is a diagonal matrix containing the
scaled singular values of $\bX$ (denoted by $\sigma_1\ge \sigma_2\ge \dots\ge \sigma_r$). The scaling factor $(mn)^{1/4}$ is introduced so that
the operator norm of $\bX$ is of the same order as the one of $\bW$ when the $\sigma_i$'s are of order one.

In the following, we denote by $\bU_k\in \reals^{m\times k}$ the matrix containing the first $k$ columns of $\bU$,
by $\bV_k\in \reals^{n\times k}$ the matrix containing the first $k$ columns of $\bV$, and so on. We also denote by $\bu_i$
the $i$-th column of $\bU$, and by $\bv_i$ the $i$-th column of $\bV$.

Our next theorem bounds the principal angle between the eigenspaces $\bU_k, \bV_k$, and their estimates $\hbU_k, \hbV_k$.
Define the function 
\begin{equation*}
G(\sigma; t )\equiv
\left(\frac{1- t^{-2} \sigma^{-4}}{1 + (\gamma^{1/2}\wedge \gamma^{-1/2}) t^{-1} \sigma^{-2}}\right)^{1/2}\, .
\end{equation*}
\begin{theorem}
\label{thm:subspace-estimation} 
Assume that conditions {\sf A0}, {\sf A1}, {\sf A2} of Theorem \ref{thm:Main0} hold. Further, let
$h_n = n^{-\eta_1}$ and $h_n^\prime = n^{-\eta_2}$ for some 
$\eta_1 \in \left(1/4, 1\right)$ and $\eta_2 \in \left(1/4, 1/3\right)$.
Consider a sequence of matrices $\bX= \bX_n \in \cF_{m, n}(r, M, \eta)$ as in 
Eq.~(\ref{eq:SVD_X}) with scaled singular values $\bSigma$
independent of $n$. 
For any $l \in [r]$ such that $\sigma_{l} \neq \sigma_{l+1}$ and $\sigma_l > \Info_W^{-1/2}$, the following holds
\begin{align}
\lim_{n\to \infty, m/n \to \gamma} 
	\sigma _{\min }\left( \widehat{\bU}_{l}^{\sT}\bU_{l}\right)
 	= G(\sigma_{l}; \Info_W) +o_{\eps,\delta}(1)\, 
\end{align}
where $\lim_{\delta\to 0}\lim_{\eps\to 0}o_{\eps,\delta}(1) =0$.
\end{theorem}

Note that this theorem implies the following phase transition phenomenon.
\begin{corollary}
Under the assumptions of Theorem \ref{thm:subspace-estimation} 
\begin{align}
\sigma_{\min}^2\left(\hbU_k^{\sT}\bU_k\right)& = o_P(1)\, , ;\;\;\;\;\;\;\;\;\;\;\;\;\mbox{ if $\sigma_k<1/\sqrt{\Info_W}$,}\\
\sigma_{\min}^2\left(\hbU_k^{\sT} \bU_k\right)& \ge \eps_0+o_P(1)\, , \;\;\;\;\;\;\mbox{ if $\sigma_k>1/\sqrt{\Info_W}$, $\sigma_k>\sigma_{k+1}$,}
\end{align} 
where $\eps_0>0$ is a constant depending on $\sigma_k$, $\sigma_{k+1}$, $p_W$.
\end{corollary}
In words, we will say the estimator $\hbU_k$ achieves weak recovery of $\bU_k$
for $\sigma_k>1/\sqrt{\Info_W}$. Obviously, the same threshold implies weak recovery of $\bV_k$. 
The arguments of \cite{krzakala2016mutual,perry2016optimality} imply that no estimator (even using knowledge of $p_W$)
can achieve weak recovery below the same threshold. 
We established that the information-theoretically  optimal weak recovery threshold can be achieved \emph{adaptively} with respect to the
noise distribution $p_W$.

\section{Numerical illustration}
\label{sec:Numerical}

 In this section we conduct experiments to illustrate the performances of our algorithm.
\subsection{Model setting}
We generate  observations $\bY \in \reals^{m\times n}$ according to the model (\ref{def:model}), i.e. $\bY = \bX+\bW$. 
For the signal matrix $\bX$, we let $\bX =(mn)^{1/4} \bU\, \bSigma\, \bV^{\sT}$,
where $\bSigma = \diag(\sigma_1, \ldots, \sigma_r)$ is diagonal and $\bU \in \reals^{m\times r}$ and $\bV \in \reals^{n \times r}$ are independent uniformly random 
orthogonal matrices.

For the noise matrix $\bW$, we let $\bW = (W_{ij})_{i\le m, j\le n}$
with i.i.d. entries $W_{ij} \sim p_W$. We choose $p_W$ to be the a mixture of the  Gaussian distributions $\normal(-\mu, 1)$ and $\normal(\mu, 1)$, whose density is
\begin{align}
\label{eqn:experiment-distribution}
p_W(x) = \frac{1}{2} \left( \frac{1}{\sqrt{2 \pi}} e^{-\frac{(x - \mu)^2}{2}} + \frac{1}{\sqrt{2 \pi}} e^{-\frac{(x + \mu)^2}{2}} \right). 
\end{align}

\paragraph{Parameter choice}
For the data generating process, we choose
\begin{equation*}
r \in \{1,3\}, \,\, m=n \in \{200, 400, 800\},  \,\, \mu=2.
\end{equation*}
Under this choice of $\mu$,
\begin{equation*}
\Var(W_{11}) = 5, \,\,\,\, \Info_W \approx 0.7256.
\end{equation*}
For our algorithm, we let 
\begin{equation*}
h_n = 1.2\,(mn)^{-1/5}, \,\, h_n^{\prime} = (mn)^{-1/7}, \,\,\varepsilon = 0.001,\,\, \delta = 0.01 .
\end{equation*}
In the case $r = 1$, we let $$\sigma_1 \in \{0.2, 0.4, 0.6, \ldots, 4 \}$$
In the case $r = 3$, we let $$\sigma_1 \in \{0.2, 0.4, 0.6, \ldots, 4 \}, \,\, \sigma_2=0.8 \sigma_1, \,\,
\sigma_3=0.6 \sigma_1.$$
For each choice of $r,n$ and $\bSigma$, we generate $50$ instances of the matrices $\bX$ and $\bY$, and apply our estimation algorithm to
compute the empirical average of corresponding accuracy metrics.


\subsection{Singular space recovery}

We first consider the accuracy in recovering the left singular subspace (by symmetry, there is no loss of generality in considering the left, rather than the
right singular subspace). We will compare our approach to classical principal component analysis, which does not denoise the matrix $\bY$, and
instead computes directly its singular value decomposition. 
Let the singular value decomposition of $\bY$ be
\begin{equation*}
\bY = \overline{\bU} \, \overline{\bSigma}^{(0)} \, (\overline{\bV})^{\sT}.
\end{equation*}
As usual, we denote by $\overline{\bU}_k\in\reals^{m\times k}$ the sub-matrix containing the top $k$ left singular vectors of $\bY$,
and by $\obV_k\in\reals^{n\times k}$ the sub-matrix containing the top $k$ right singular vectors.

We compare the two approaches by computing $\sigma_{\min} ( \hbU_{i}^{\sT}\bU_{i} )$ and 
$\sigma_{\min} ( \obU_{i}^{\sT}\bU_{i}) $ for $1 \leq i \leq r$. Note that Theorem \ref{thm:subspace-estimation}
predicts the limit of $\sigma_{\min} ( \hbU_{i}^{\sT}\bU_{i} )$ to be given by  $G(\sigma _{i}; \Info_W)$ (by taking $n \to \infty$, then $\eps \to 0$, then $\delta \to 0$).
Analogously,  Theorem \ref{theorem:general-deform} predicts the limit of  $\sigma_{\min} ( \obU_{i}^{\sT}\bU_{i})$ to be given by
$G(\sigma _{i}; \Var(W_{11})^{-1})$.
\begin{figure}
	\centering
	\begin{subfigure}{} 
		\includegraphics[width=0.53\textwidth]{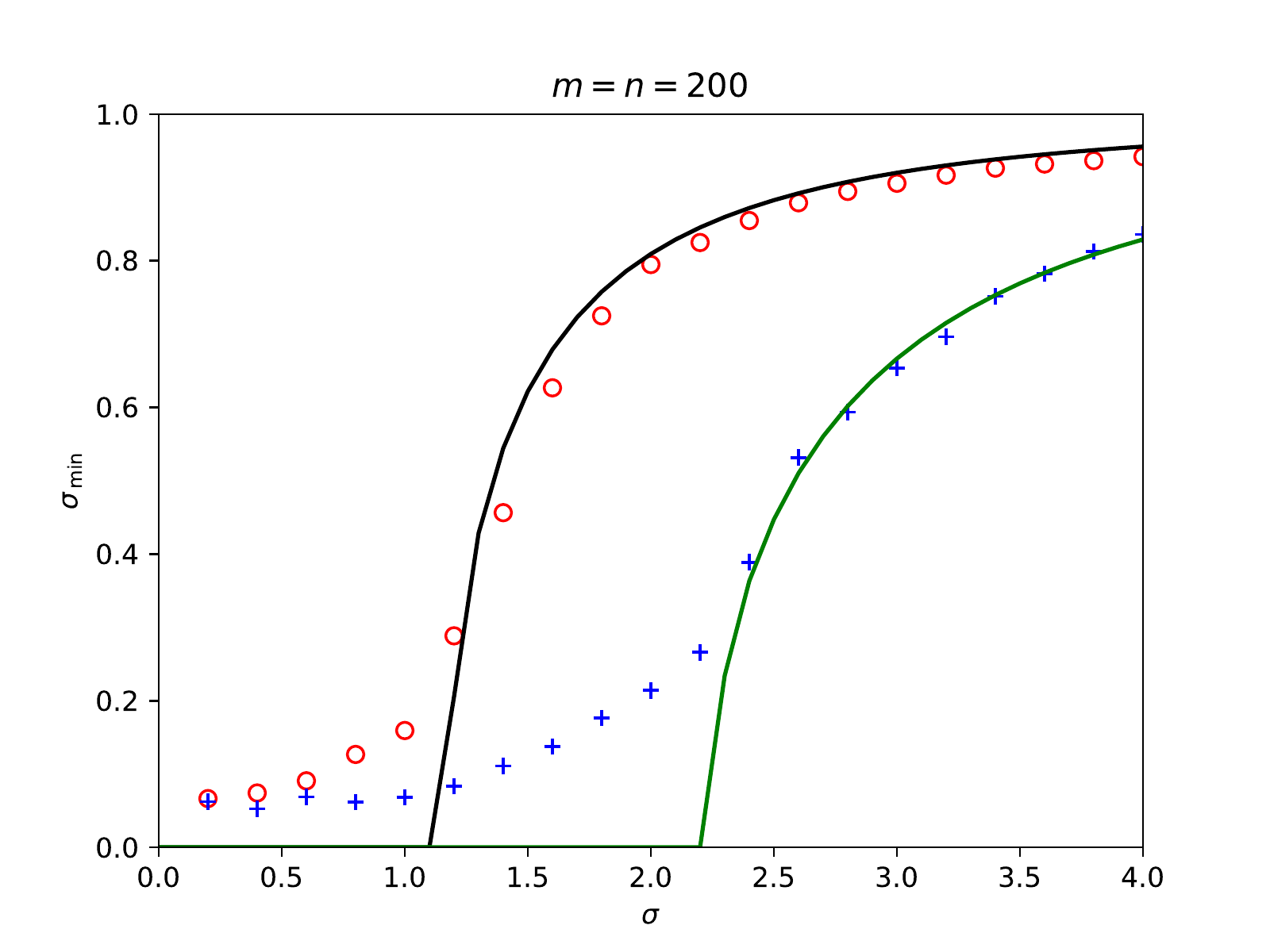}
	\end{subfigure}
	\begin{subfigure}{} 
		\includegraphics[width=0.53\textwidth]{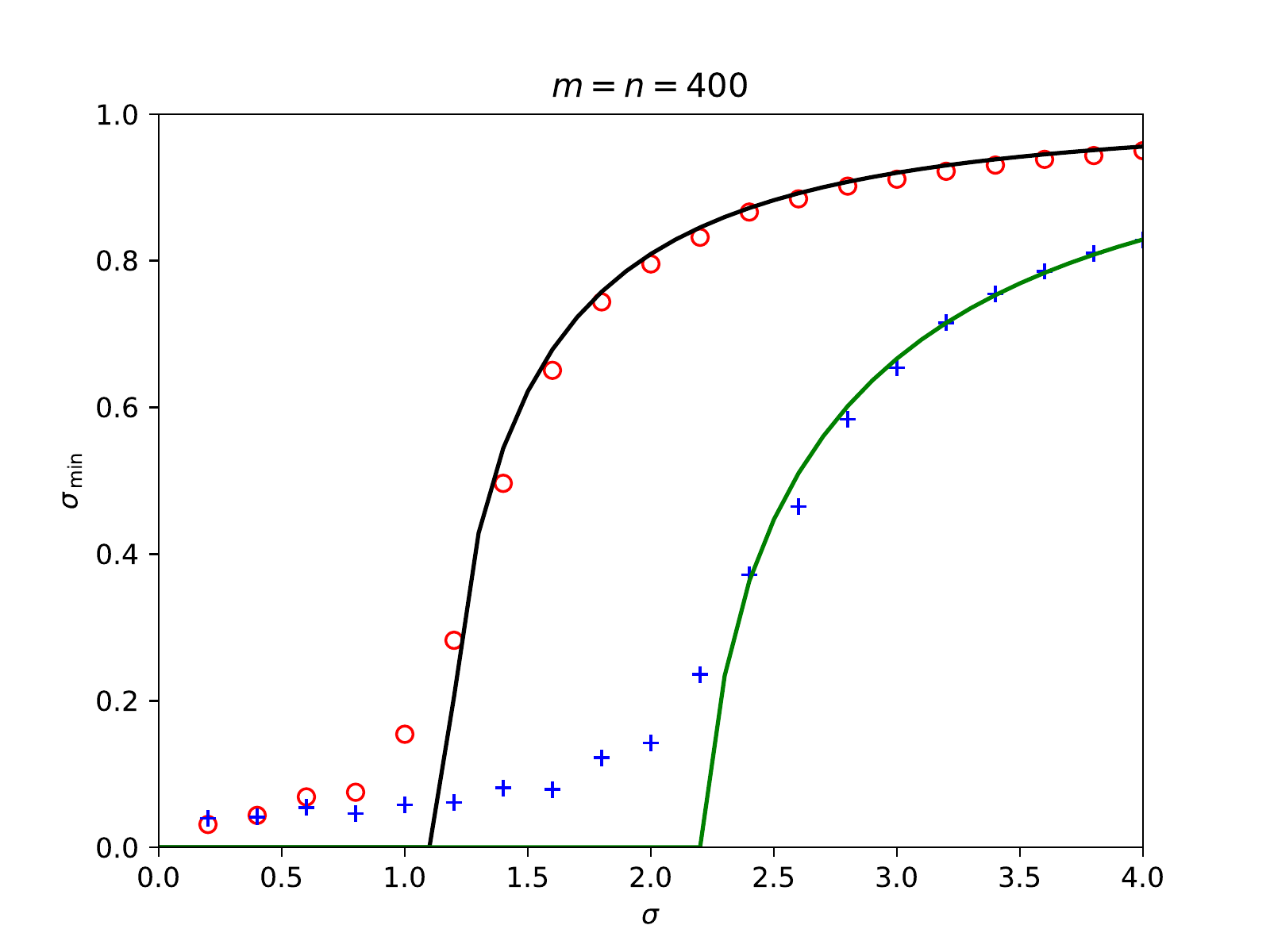}
	\end{subfigure}
	\begin{subfigure}{} 
		\includegraphics[width=0.53\textwidth]{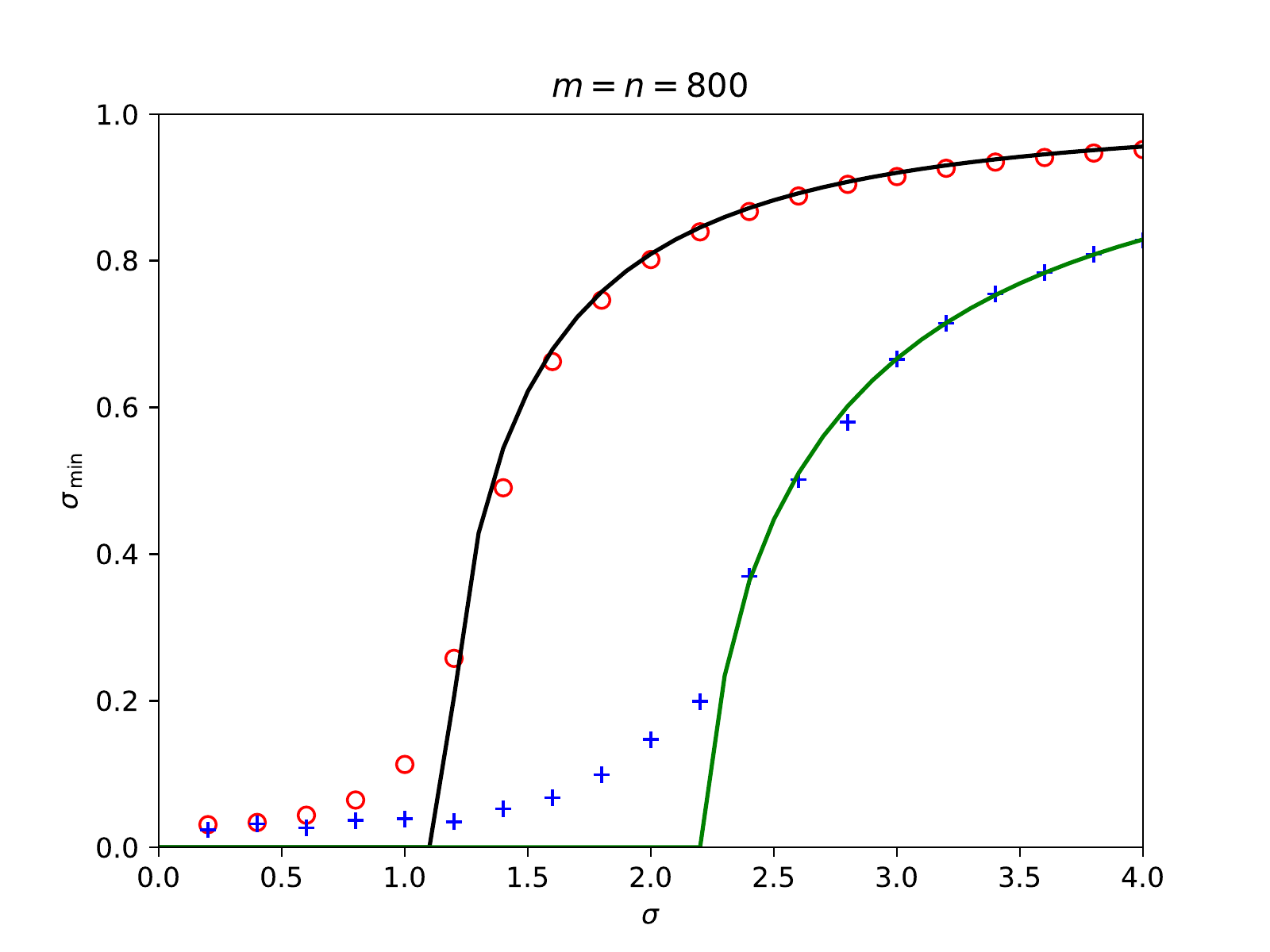}
	\end{subfigure}
	\caption{Singular space recovery in the case $r=1$. The three plots refer to $m=n\in\{200,400,800\} $ respectively. Red circles are the average of $\sigma_{\min} ( \hbU_{1}^{\sT}\bU_{1} )$ over $50$ realizations, and black curve is its asymptotic prediction, given by $G(\sigma_1; \Info_W)$. Blue `$+$' are the average of $\sigma_{\min} ( \obU_{1}^{\sT}\bU_{1})$ over $50$ realizations, and green curve is its asymptotic prediction, given by $G(\sigma_1; \Var(W_{11})^{-1})$.}\label{fig:Ucorr1}
\end{figure}

\begin{figure}
	\centering
	\begin{subfigure}{} 
		\includegraphics[width=0.53\textwidth]{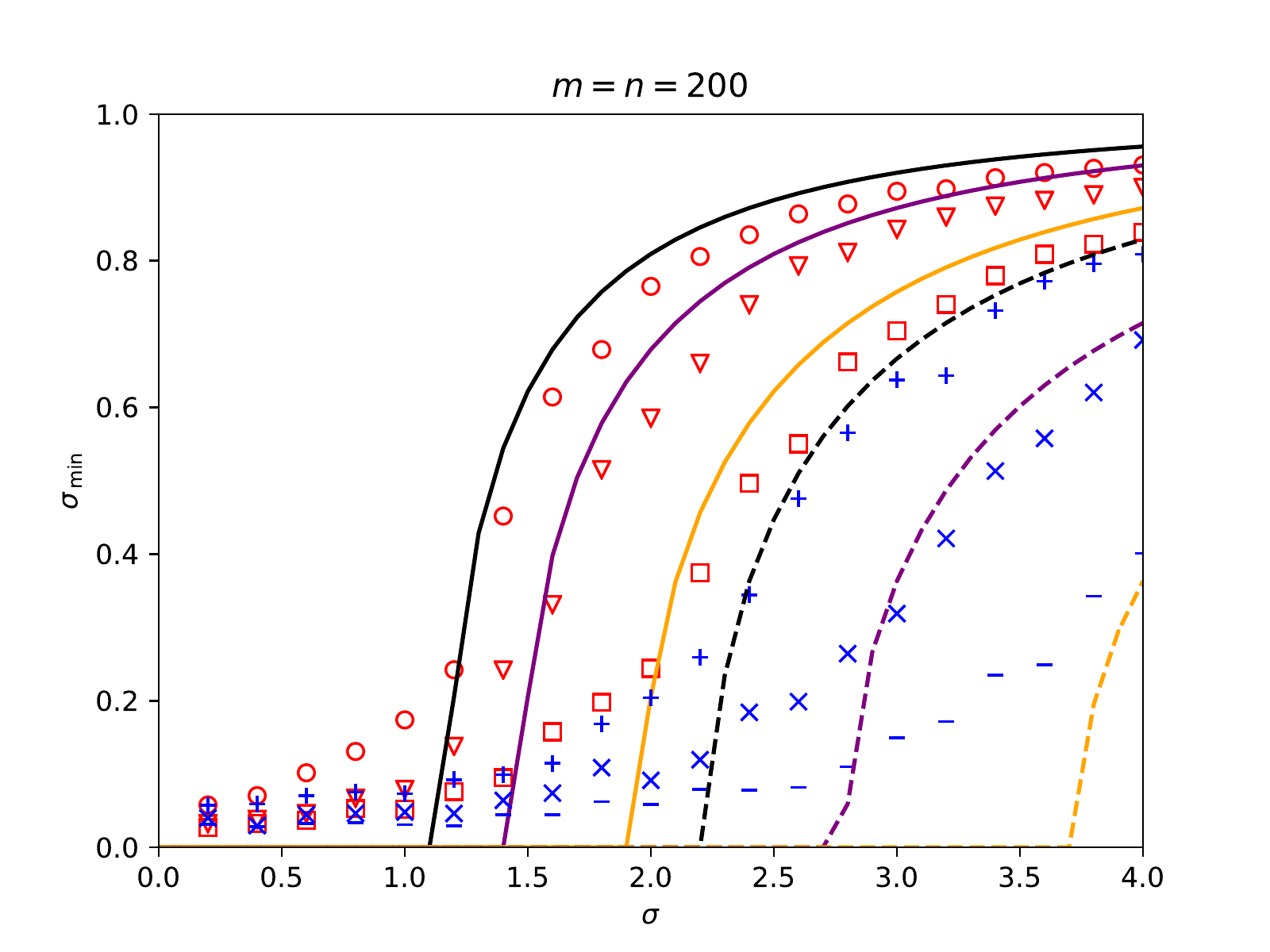}
	\end{subfigure}
	\begin{subfigure}{} 
		\includegraphics[width=0.53\textwidth]{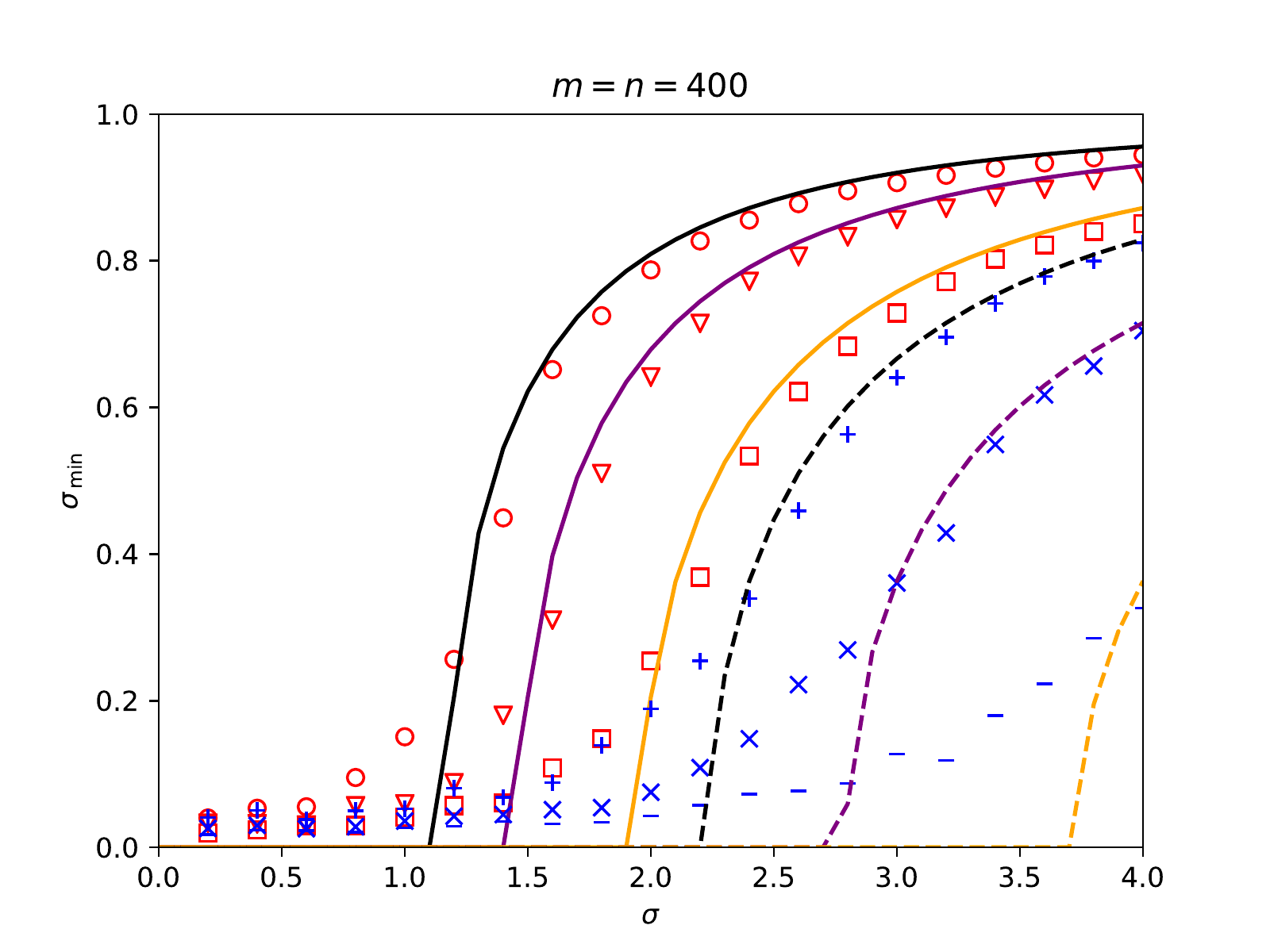}
	\end{subfigure}
	\begin{subfigure}{} 
		\includegraphics[width=0.53\textwidth]{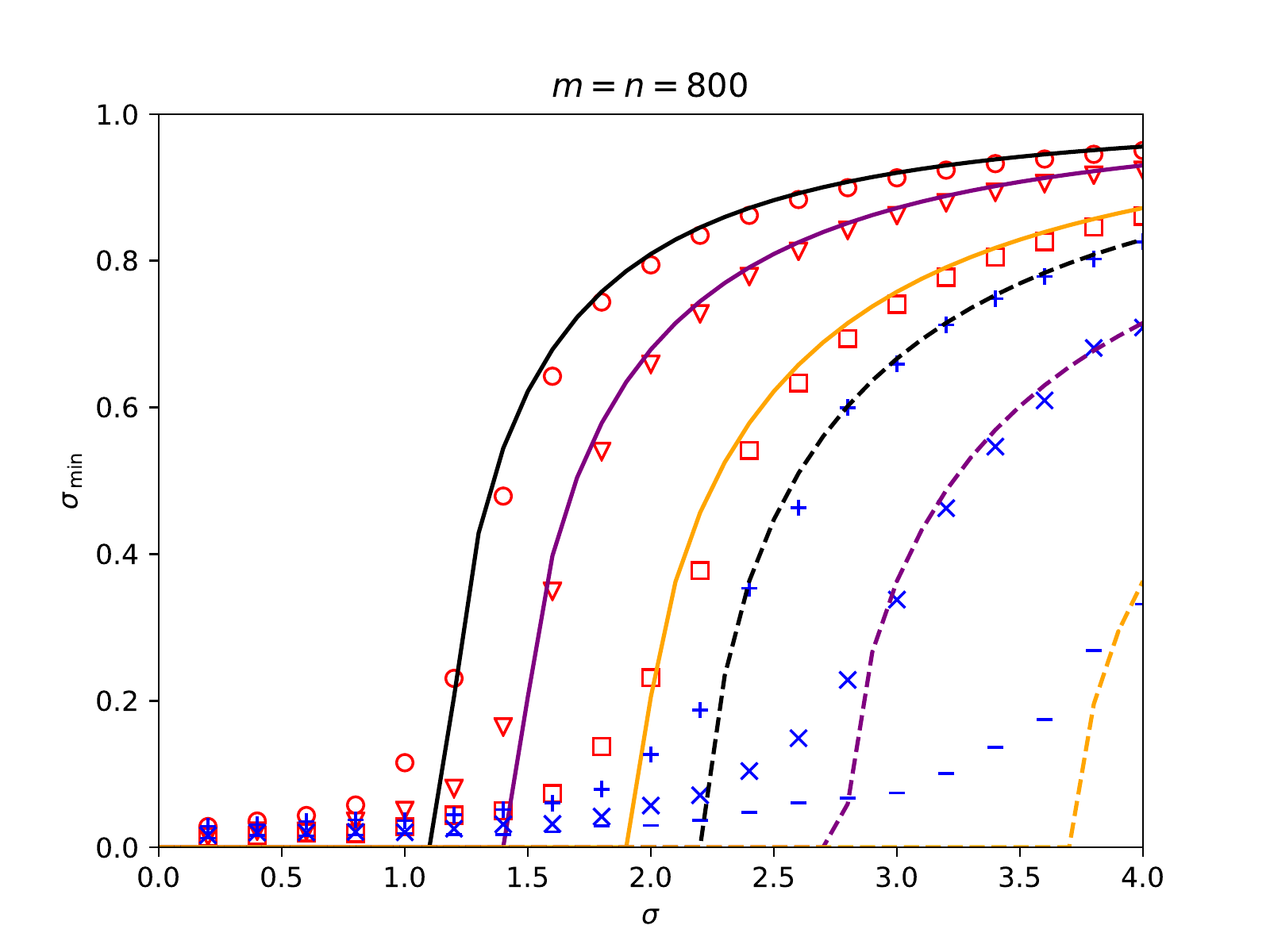}
	\end{subfigure}
	\caption{Singular space recovery in the case $r=3$. The three plots refer to $m=n\in\{200,400,800\}$ respectively. Red circles, triangles and squares are the averages of $\sigma_{\min} ( \hbU_{i}^{\sT}\bU_{i} )$ over $50$ realizations for $i=1,2,3$ respectively; black, purple and orange solid curves are their asymptotic predictions, given by $G(\sigma_i; \Info_W)$. 
Blue `$+$', `$x$' and `$-$' are the average of $\sigma_{\min} ( \obU_{i}^{\sT}\bU_{i})$ over $50$ realizations for $i=1,2,3$ respectively;  black, purple and orange dotted curves are their asymptotic predictions, given by $G(\sigma_i; \Var(W_{11})^{-1})$.}\label{fig:Ucorr2}
\end{figure}

The results of this computation are reported in Figure \ref{fig:Ucorr1} (for $r=1$) and Figure \ref{fig:Ucorr2} (for $r=3$). The present approach outperforms substantially standard PCA, and
its behavior is well captured by the asymptotic theory.

\subsection{Signal estimation}

We next consider the problem of estimating the matrix $\bX$. As before, we compare
our estimator $\hbX(\bY)$ to  a simpler one, denoted by $\obX(\bY; \delta)$, which does not make use of the denoising step,
and is defined below (this approach was originally proposed in \cite{shabalin2013reconstruction}):
\begin{enumerate}
\item Construct the diagonal matrix $\overline{\bSigma}$ with entries
\begin{equation}
\overline{\Sigma}_{i,i} = \begin{cases}
\sigma
	H^{-1}({\sigma}^{-1}\overline{\Sigma}_{i,i}) & 
		\mbox{ if $\overline{\Sigma}_{i,i}\ge (1+\delta)H(1)\sigma$,}\\
0& \mbox{ otherwise.}
\end{cases}
\end{equation}
\item Return
\begin{align}
\overline{\bX}(\bY;\delta)  = (mn)^{1/4}\overline{\bU}\, \overline{\bSigma}\, \overline{\bV}^{\sT}\, .
\end{align}
\end{enumerate}

\begin{figure}
	\centering
	\begin{subfigure}{} 
		\includegraphics[width=0.53\textwidth]{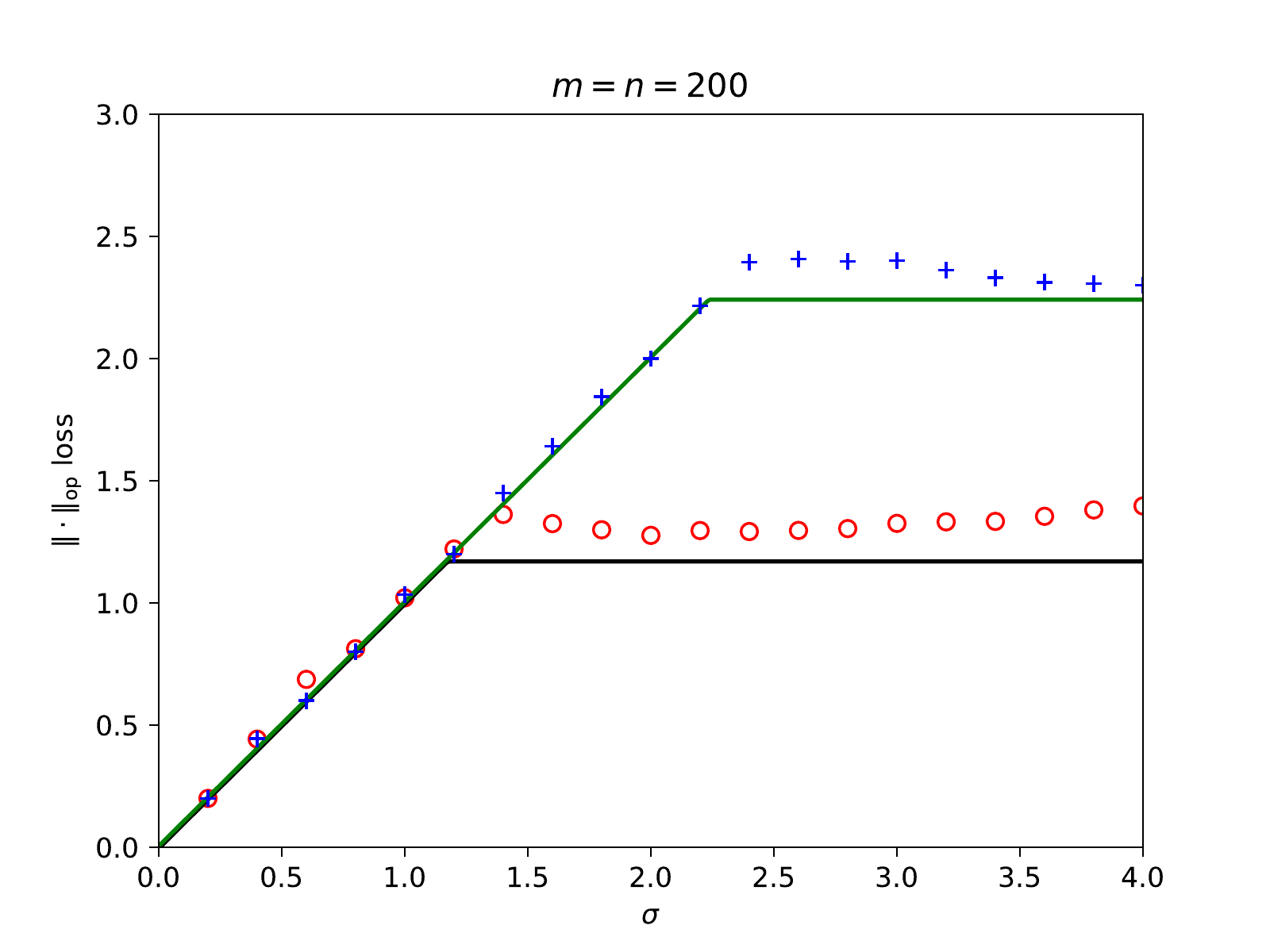}
	\end{subfigure}
	\begin{subfigure}{} 
		\includegraphics[width=0.53\textwidth]{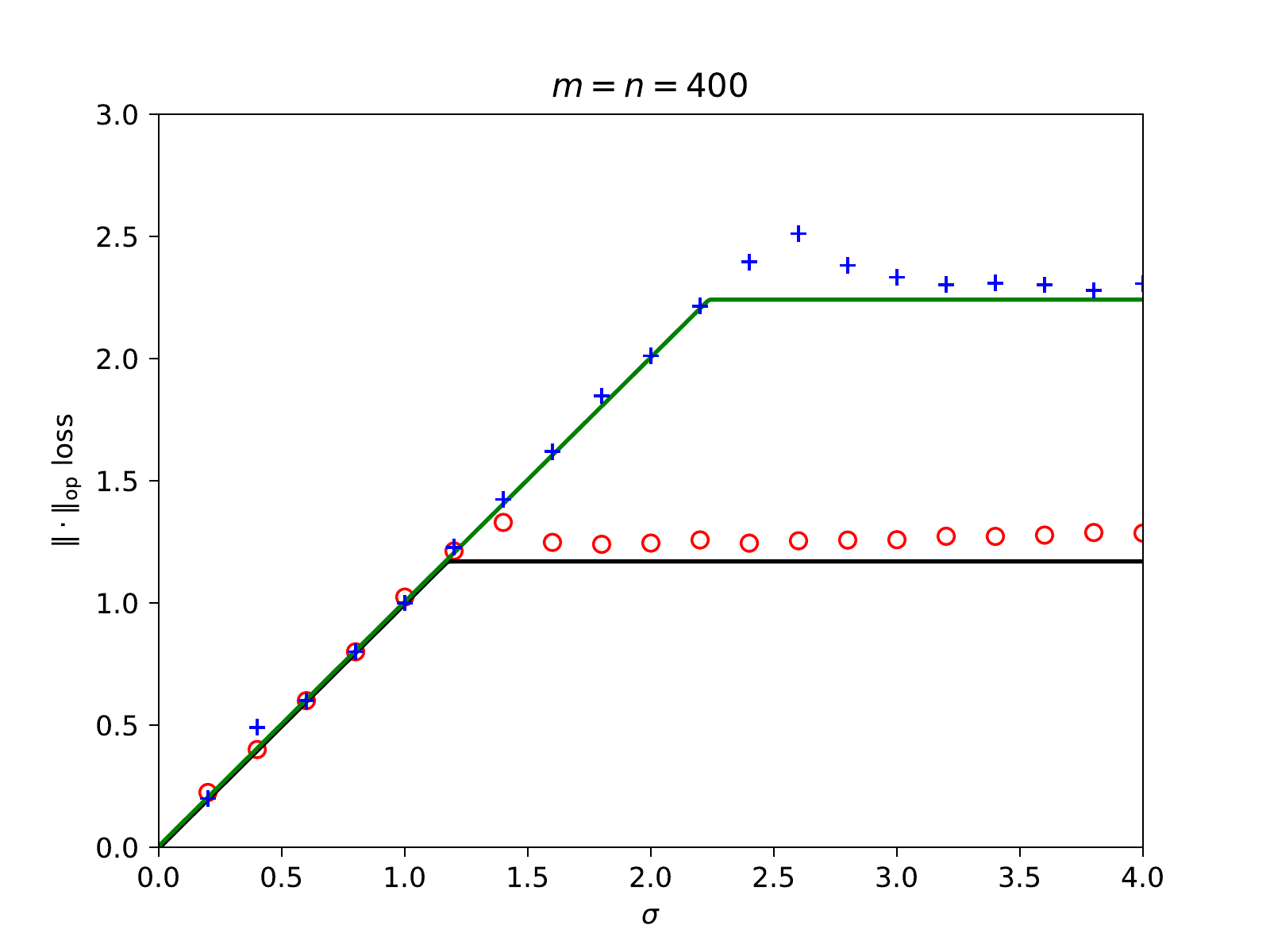}
	\end{subfigure}
	\begin{subfigure}{} 
		\includegraphics[width=0.53\textwidth]{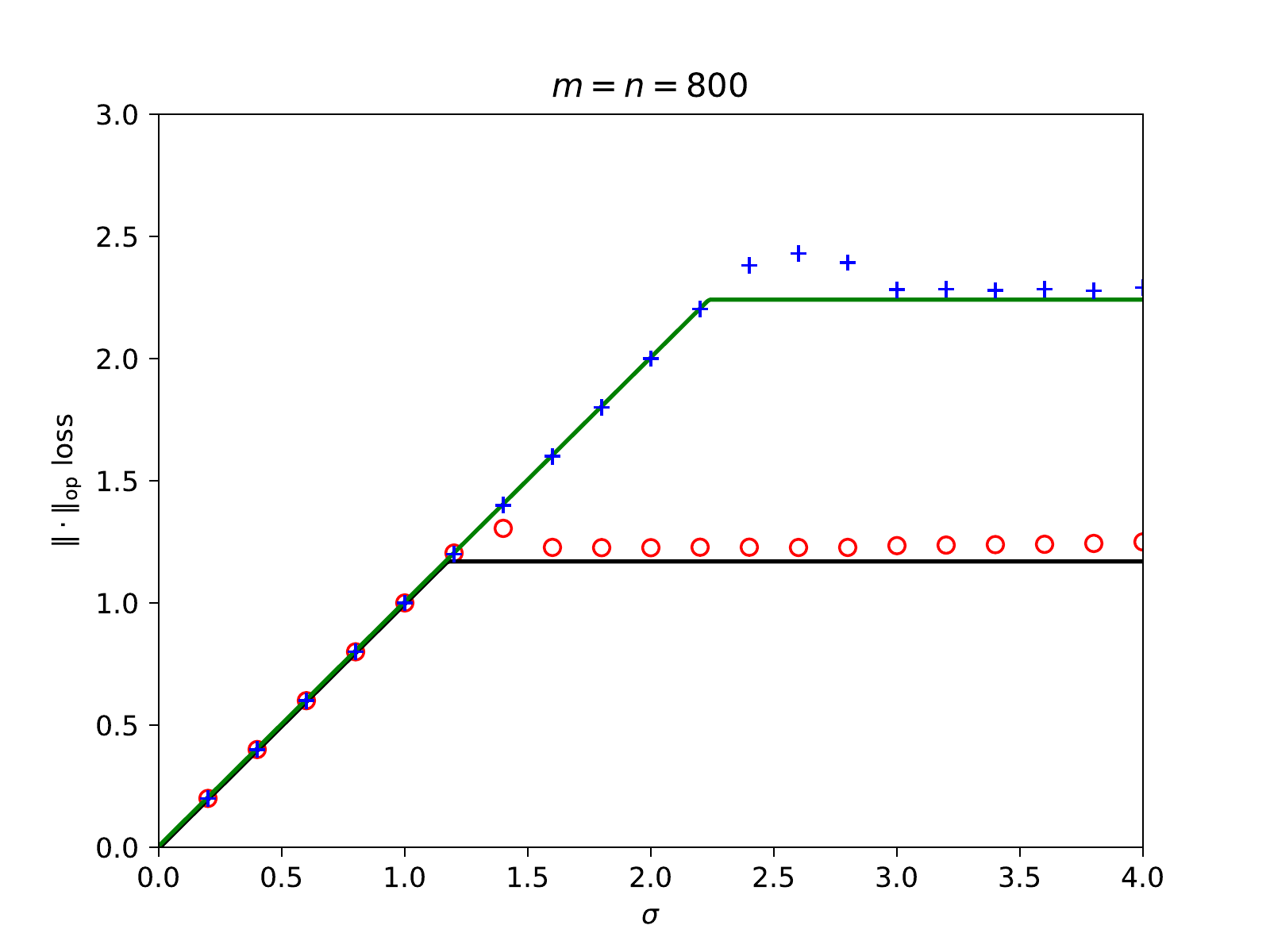}
	\end{subfigure}
	\caption{Matrix denoising under the operator norm loss in the case $r=1$. The three plots refer to $m=n\in\{200,400,800\}$ respectively. Red circles are  averages of $(nm)^{-1/4} \|\widehat{\bX}(\bY)-\bX\|_{\op}$ over 50 realizations, and blue curve is its asymptotic prediction, given by $\sigma \wedge \Info_W^{-1/2}$. Blue `$+$' are  averages of $(nm)^{-1/4} \|\overline{\bX}(\bY;\delta)-\bX\|_{\op}$ over 50 realizations, and green curve is its asymptotic prediction, given by $\sigma \wedge \Var (W_{11})^{-1/2}$.}
\end{figure}
\begin{figure}
	\centering
	\begin{subfigure}{} 
		\includegraphics[width=0.53\textwidth]{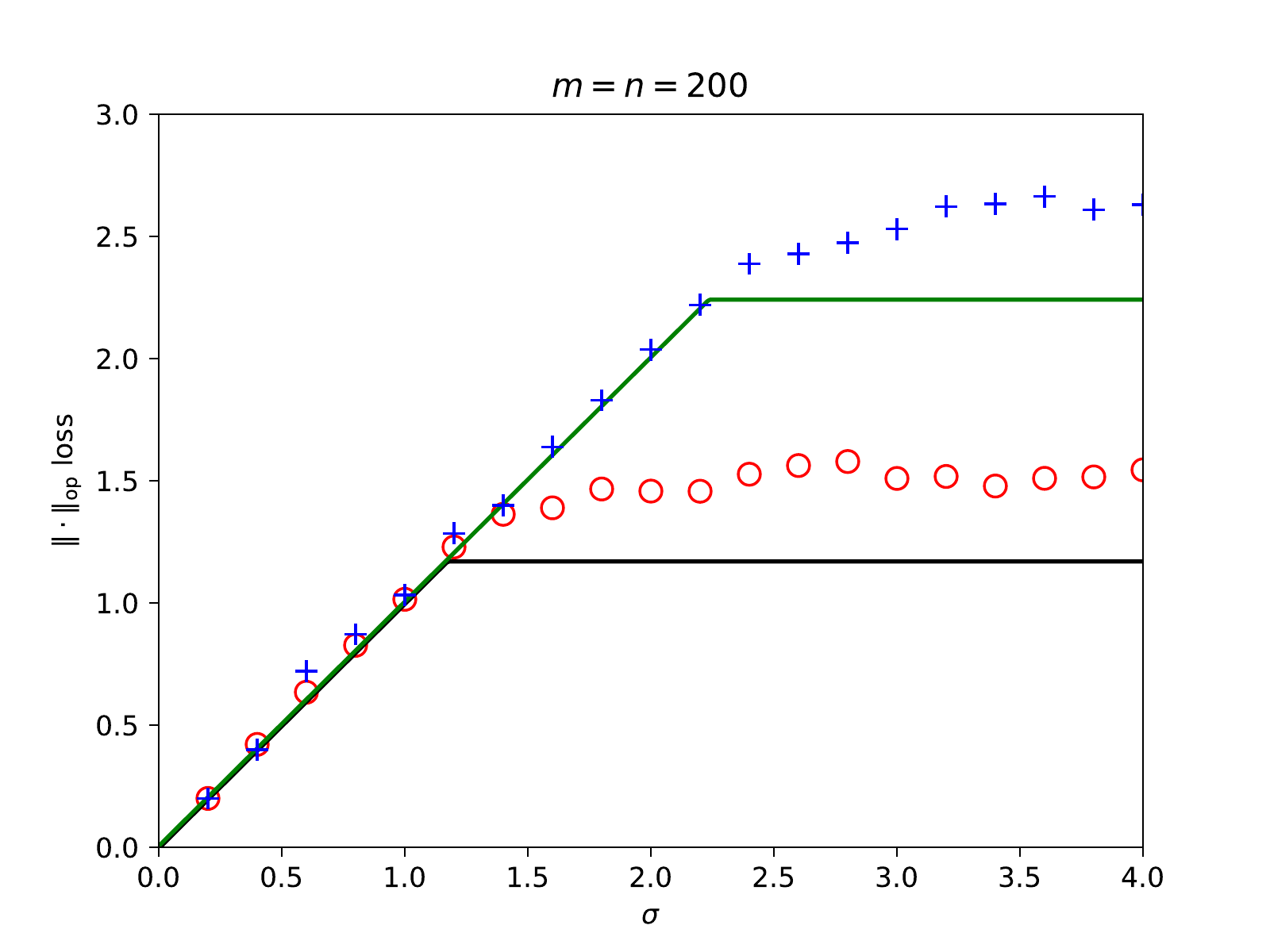}
	\end{subfigure}
	\begin{subfigure}{} 
		\includegraphics[width=0.53\textwidth]{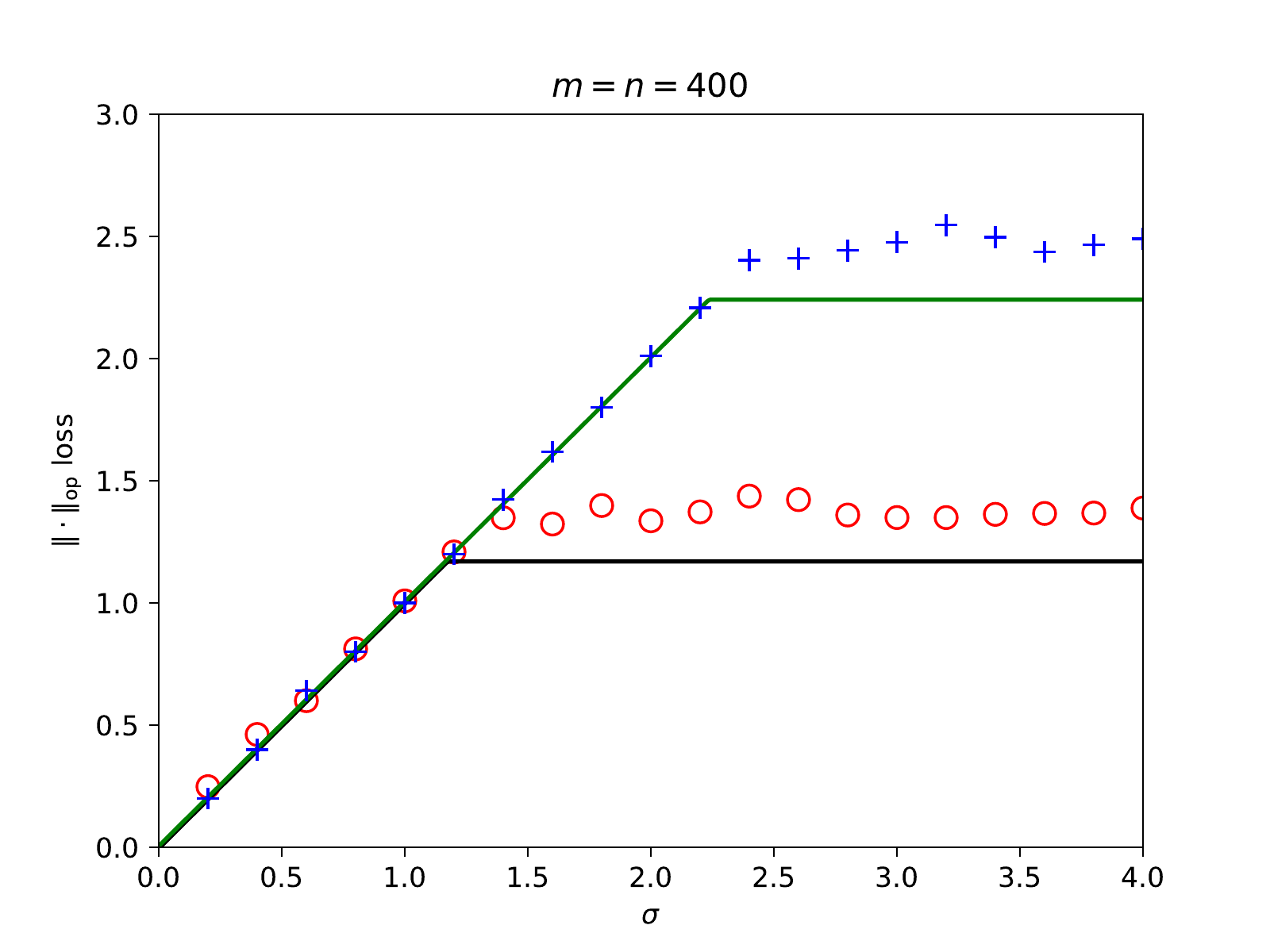}
	\end{subfigure}
	\begin{subfigure}{} 
		\includegraphics[width=0.53\textwidth]{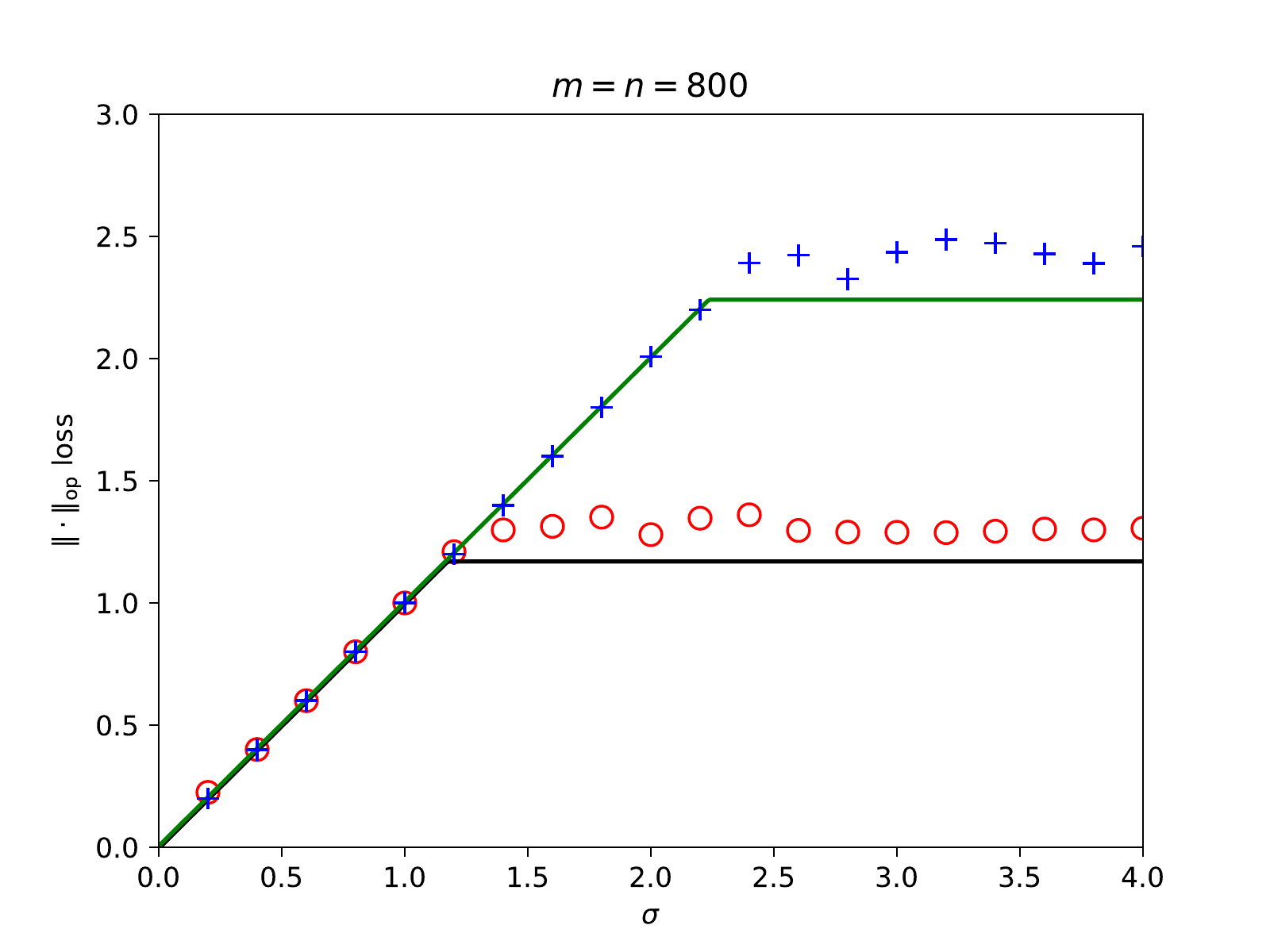}
	\end{subfigure}
	\caption{Matrix denoising under the operator norm loss in the case $r=3$. The three plots refer to $m=n\in\{200,400,800\}$ respectively. Red circles are the average of $(nm)^{-1/4} \|\widehat{\bX}(\bY)-\bX\|_{\op}$ over 50 realizations, and blue curve is its asymptotic prediction, given by $\sigma \wedge \Info_W^{-1/2}$. Blue `$+$' are the average of $(nm)^{-1/4} \|\overline{\bX}(\bY;\delta)-\bX\|_{\op}$ over 50 realizations, and green curve is its asymptotic prediction, given by $\sigma \wedge \Var (W_{11})^{-1/2}$.}
\end{figure}

The asymptotic theory predicts $(nm)^{-1/4} \|\widehat{\bX}(\bY)-\bX\|_{\op}$ to converge to $\sigma_1\wedge \Info_W^{-1/2}$,
and  $(nm)^{-1/4} \|\obX(\bY)-\bX\|_{\op}$ to converge to $\sigma_1 \wedge \Var(W_{11})^{-1/2}$. Again, 
theory captures well the behavior of our experiments already at moderate sizes.
Above the  information-theoretic threshold $\sigma_1 = \Info_W^{-1/2}$,
the new approach outperforms standard PCA.

\section{Proof of Theorem \ref{thm:LowerBound} (lower bound)}
\label{sec:Techniques_LB}

Throughout the proof, we let
\begin{equation}
\label{eqn:def-Q-gamma}
Q(\gamma; r, \eta) \defeq 
	\liminf_{M \to \infty} \liminf_{n \to \infty, m/n \to \gamma} \inf_{\hat{\bX}}
		\sup_{\bX \in \cF_{m, n}(r, M, \eta)} \frac{1}{(mn)^{1/4}}\E \norm{\hat{\bX}(\bY) - \bX}_{\op}. 
\end{equation}
It suffices to show the following lower bound on the quantity $Q(\gamma; r, \eta)$: 
\begin{equation}
\label{eqn:lower-bound-result}
Q(\gamma; r, \eta) \ge \frac{\gamma^{1/4} \vee \gamma^{-1/4}}{\sqrt{\Info_W}}. 
\end{equation}
By symmetry, we can assume $\gamma \ge 1$. Denote $\bv \in \R^n$ to be the 
vector such that $v_i = 1/\sqrt{n}$ for all $i \in [n]$.  For any $c > 0$, define the set 
$\set(c)$ as follows:
\begin{equation}
\set(c) \defeq \{\bx \bv^{\sT} \mid \bx \in \R^m, \norm{\bx}_\infty \le c\}. 
\end{equation}
Note that $\set(c) \subseteq \cF_{m, n}(r, c, \eta)$ for any constant $c > 0$.
Let $\pi^{\Gamma, c}$ denote the one dimensional Gaussian distribution with mean 
$0$ and variance $\Gamma > 0$ truncated on the interval $[-c, c]$. Now, draw the vector 
$\bx \in \R^m$ such that each coordinate of $\bx$ is independent following the distribution 
$\pi^{\Gamma, c}$. This induces a distribution on $\set(c)$, for which we denote by 
$\prob_{\pi}^{\Gamma, c}$. 

To show Eq~\eqref{eqn:lower-bound-result}, we start by introducing the 
following Bayesian setting. Let $\bX$ be sampled from $\pi^{\Gamma, c}$ 
and $\bW$ be an independent matrix with i.i.d entries with  distribution $p_W$. 
Set $\bY = \bX+ \bW$. With a slight abuse of notations, in the rest of the proof, we denote 
the joint distribution of $(\bX, \bY)$ under this sampling scheme by $\prob_{\pi}^{\Gamma, c}$. 
We use $\E_{\pi}^{\Gamma, c}$ to denote the expectation under $\prob_{\pi}^{\Gamma, c}$. 
Now, for any $\gamma, r, M, \eta, n$, define the quantity 
\begin{equation}
R(\gamma; r, M, \eta, n) \defeq 
	\inf_{\hat{\bX}} \sup_{\bX \in \cF_{m, n}(r, M, \eta)} \left\{ \frac{1}{(mn)^{1/4}}\E \norm{\hat{\bX}(\bY) - \bX}_{\op} \right\}.
\end{equation} 
By the standard reduction argument of lower bounding minimax risk by Bayesian risk, 
we get that, 
\begin{align}
R(\gamma; r, M, \eta, n) 
	&\ge \inf_{\hat{\bX}} \left\{\E^{\Gamma, c}_{\pi} \left[\frac{1}{(mn)^{1/4}} \norm{\hat{\bX}(\bY) - \bX}_{\op}\right] \right\}  
		\nonumber \\
&= \inf_{\hat{\bX}} \left\{\E^{\Gamma, c}_{\pi}\left[ \E^{\Gamma, c}_{\pi} 
	\left[\left.\frac{1}{(mn)^{1/4}} \norm{\hat{\bX}(\bY) - \bX}_{\op} \right| \bY \right]\right]\right\} \nonumber \\
&\ge  \E^{\Gamma, c}_{\pi}\left[ 
	\inf_{\hat{\bX}}\left\{ \E^{\Gamma, c}_{\pi} 
		\left[\left.\frac{1}{(mn)^{1/4}} \norm{\hat{\bX}(\bY) - \bX}_{\op} \right| \bY\right]\right\} \right].
		\label{eqn:minimax-bayesian-method}
\end{align}
Now that, $\ltwo{\bv} = 1$. Thus, when $\bX$ takes the form $\bX= \bx \bv^{\sT}$, any estimator $\hat{\bX}$ satisfies, 
\begin{equation}
\norm{\hat{\bX}(\bY) - \bX}_{\op} \ge \ltwo{(\hat{\bX}(\bY) - \bX)\bv} = \ltwo{\hat{\bX} \bv - \bx}.
\end{equation}
Substituting the above estimate into Eq~\eqref{eqn:minimax-bayesian-method}, we get the lower bound, 
\begin{align}
R(\gamma; r, M, \eta, n) &\ge 
	\E^{\Gamma, c}_{\pi}\left[ 
	\inf_{\hat{\bX} \in \R^{m \times n}}\left\{ \E^{\Gamma, c}_{\pi} 
		\left[\left.\frac{1}{(mn)^{1/4}} \ltwo{\hat{\bX}(\bY)\bv - \bx} \right| \bY\right]\right\} \right]  \nonumber \\
	&\ge \E^{\Gamma, c}_{\pi}
		\left[ \inf_{\hat{\bx} \in \R^m}
			\left\{ \E^{\Gamma, c}_{\pi} \left[\left.\frac{1}{(mn)^{1/4}} \ltwo{\hat{\bx}(\bY) - \bx} \right| \bY\right]
				\right\} \right].
\label{eqn:reduction-to-vector}
\end{align}
To further lower bound the right-hand of Eq~\eqref{eqn:reduction-to-vector}, we use the next lemma.
\begin{lemma}
\label{lemma:norm-lower-bound}
Let $\bZ = (Z_1, Z_2, \ldots, Z_n) \in \R^n$ be a random vector with independent coordinates. Then for any 
constant $K > 0$, any (non-random) vector $\ba = (a_1, a_2, \ldots, a_n) \in \R^n$ and any $t > 0$, we have
\begin{equation}
\E \ltwo{\bZ - \ba} \ge  \left(1- K^2n^{-1}t^{-2} \right)_+ \left(\E \loss^2(\bZ, \ba; K)- nt\right)_+^{1/2}. 
\end{equation}
where the loss $\loss(\cdot, \cdot; K): \R^n \times \R^n \to R$ is defined by 
\begin{equation}
\label{eqn:def-loss}
\loss(\bZ, \ba; K) = \Big(\sum_{i=1}^n [(Z_i - a_i)^2 \wedge K]\Big)^{1/2}. 
\end{equation}
\end{lemma}
\begin{proof}
Denote by $T_i = (Z_i - a_i)^2 \wedge K$ for $i \in [m]$. Note first that, by definition, we have
\begin{equation}
\label{eqn:determistic-lower-bound-norm}
\ltwo{\bZ - \ba} \ge \loss(\bZ, \ba; K) = \Big(\sum_{i=1}^n T_i \Big)^{1/2} 
	\ge \Big(\sum_{i=1}^n \E [T_i]  - nt\Big)_+^{1/2} \times \indic{\sum_{i=1}^n T_i \ge 
		\sum_{i=1}^n\E[T_i]  -nt}. 
\end{equation}
Using the definition of Eq~\eqref{eqn:def-loss} and taking expectation over both sides of 
Eq~\eqref{eqn:determistic-lower-bound-norm}, we get
\begin{equation}
\label{eqn:expectation-lower-bound-norm}
\E \ltwo{\bZ - \ba} \ge \Big(\E \loss^2(\bZ, \ba; K)  - nt\Big)_+^{1/2} \times  
	\prob\left(\sum_{i=1}^n T_i \ge \sum_{i=1}^n\E[T_i] -nt\right). 
\end{equation}
By assumption, $\{T_i\}_{i=1}^n$'s are independent random variables with 
$\big|T_i\big| \le K$ almost surely. Thus, we have, 
$\Var\left(\sum_{i=1}^n T_i\right) \le \sum_{i=1}^n \Var(T_i) \le nK^2$. Now
Markov's inequality implies, 
\begin{equation}
\label{eqn:Markov-inequality-norm}
\prob\left(\sum_{i=1}^n T_i \ge \sum_{i=1}^n\E[T_i] -nt\right) \ge 
1- \prob \left(\left|\frac{1}{n} \sum_{i=1}^n (T_i - \E [T_i])\right| \ge t\right) \ge 1- K^2n^{-1}t^{-2}.
\end{equation}
Now, the desired result follows by Eq~\eqref{eqn:expectation-lower-bound-norm} and 
Eq~\eqref{eqn:Markov-inequality-norm}.
\end{proof}
Now we turn to lower bound the RHS of Eq~\eqref{eqn:reduction-to-vector}. 
For each $i \in [n]$, denote $\bY_i = (Y_{i, 1}, Y_{i, 2}, \ldots, Y_{i, n})^{\sT}$
and $\bW_i = (W_{i, 1}, W_{i, 2}, \ldots, W_{i, n})^{\sT}$ to be the $i$th row 
of the matrix $\bY$ and $\bW$ respectively. Then, since $\bY = \bX + \bW$, 
this implies 
\begin{equation}
\label{eqn:observation-new-model-vector}
\bY_i = x_i \bv + \bW_i~~\text{for all $i \in [m]$}.
\end{equation}
Thus, by the joint independence of $\{(x_i, \bW_i, \bY_i)\}_{i=1}^m$, we 
know that the posterior distribution of $\bx$ given $\bY$ has independent 
coordinates. Moreover, we have, 
\begin{equation}
\law(\bx_i \mid \bY) = \law (\bx_i \mid \bY_i)~~\text{for all $i \in [m]$}.
\end{equation}
Now, we can use Lemma~\ref{lemma:norm-lower-bound} to show that, for 
any estimator $\hat{\bx} \in \R^m$ and $t, K > 0$:
\begin{equation}
\E^{\Gamma, c}_{\pi} \left[ \ltwo{\hat{\bx}(\bY) - \bx} \mid \bY\right]
	\ge \left(1- K^2m^{-1} t^{-2}\right)_+
\sum_{i=1}^m \E^{\Gamma, c}_{\pi} 
	\Big([(\hx_i(\bY) - x_i)^2 \wedge K \,|\, \bY_i] - m t\Big)_+^{1/2}
\end{equation}
Hence, substituting the above estimate into Eq~\eqref{eqn:reduction-to-vector},
we get that for $M \ge c$, 
\begin{equation}
\label{eqn:reduction-to-conditional-expectation-sum}
R(\gamma; r, M, \eta, n) \ge 
	  \left(1- K^2m^{-1} t^{-2} \right)_+\E^{\Gamma, c}_{\pi} \left[ \inf_{\hbx \in \R^n}
	  	\Big(\frac{1}{(mn)^{1/2}}
			\sum_{i=1}^m \E^{\Gamma, c}_{\pi} [(\hx_i - x_i)^2 \wedge K \,|\, \bY_i] 
				- \gamma^{1/2} t\Big)_+^{1/2}\right].
\end{equation}
Now, to further lower bound the RHS of Eq~\eqref{eqn:reduction-to-conditional-expectation-sum}, 
we use~\cite[Proposition 4, Section 6.4]{LeYa12}. Note that a simple transformation of that 
result gives the result below (see~e.g.~\cite[Theorem 2]{Duchi18}).

\begin{theorem}
\label{theorem:posterior-is-almost-gaussian}
Let $\bX = (X_1, X_2, \ldots, X_n)^{\sT} \in \R^n$ whose coordinates $X_i$ are independent 
random variables drawn from the same distribution $P$. Denote by $P_{\theta}$ the 
distribution of a shift of $P$ by $\theta$ for all $\theta \in \R$. 
Suppose that the families $\{P_{\theta}^{\otimes n}\}_{\theta\in \R}$ is a quadratic mean differentiable family, so that 
for some quantity $\Info$ and random variable $\Delta_n(\bX)$, 
\begin{equation}
\log \frac{\de P^{\otimes n}_{h/\sqrt{n}}(\bX)}{\de P^{\otimes n}(\bX)} = h \Delta_n(\bX) - \half h^2 \Info+ o_P(1)\label{eq:QMDLemma}
\end{equation}
holds for all $h \in \R$. Now, denote $Y_i = \frac{1}{\sqrt{n}} h + X_i$ for $i \in [n]$ 
and $\bY = (Y_1, Y_2, \ldots, Y_n)^{\sT} \in \R^n$. Fix $\Gamma > 0$. Denote 
$\pi^{\Gamma, c}$ to be the one dimensional Gaussian distribution with mean $0$
and variance $\Gamma$, truncated to $[-c,c]$. Put $\pi^{\Gamma, c}$ as the prior distribution on $h$ 
and denote $\pi^{\Gamma, c}(\cdot \mid \bY)$ to be the posterior of $h$ given $\bY$.
Then for any $\eps > 0$, there exist $C = C(\eps)$ and $N = N(\eps)$ such that for 
$c \ge C$ and $n \ge N$,  
\begin{equation}
\int \tvnorm{G^{\Gamma}(\, \cdot\, \mid \Info^{-1} \bY) - \pi^{\Gamma, c}(\, \cdot\, \mid \bY)} \de \wbar{P}_n(\bY)
	\le \eps, 
\end{equation}
where in above we denote $G^{\Gamma}(\cdot \mid \Info^{-1} \bY)$ to be the Gaussian distribution
\begin{equation}
\label{eqn:def-posterior-normal}
G^{\Gamma}(\, \cdot\, \mid \Info^{-1} \bY) = \normal \left((\Info + \Gamma^{-1})^{-1}\Delta_n(\bY), (\Info + \Gamma^{-1})^{-1}\right).
\end{equation} 
and $\wbar{P}_n(\bY)$ to be the marginal distribution of $\bY$, 
\begin{equation*}
\wbar{P}_n = \int P_{h/\sqrt{n}, n} \de\pi^{\Gamma, c}(h).
\end{equation*}
\end{theorem} \noindent\noindent

Now, fix $ \Gamma > 0$ and $\eps > 0$. 
By assumption, the coordinates of $\bY=\bY_i$ (cf. Eq.~\eqref{eqn:observation-new-model-vector})  are i.i.d. with common distribution 
$P_{\theta}$ which is a shift of the density $p_W$ with finite Fisher information. 
Hence, Eq.~(\ref{eq:QMDLemma}) holds with $I = \Info_W$, and $\Delta_n(\bY) = n^{-1/2}\sum_{i\le n}p_W'(Y_i)/p_W(Y_i)$.
Denote by $\pi^{\Gamma, c}(\, \cdot\, \mid \bY_i)$ the conditional 
distribution of $x_i$ given $\bY_i$ and $G^{\Gamma}(\, \cdot \, | \bY_i)$ to be 
the normal distribution as in Eq~\eqref{eqn:def-posterior-normal}
\begin{equation}
G^{\Gamma}(\,\cdot\, \mid  \bY) = \normal \left((\Info_W + \Gamma^{-1})^{-1}\Delta_n(\bY), (\Info_W + \Gamma^{-1})^{-1}\right).
\end{equation}
Recall our observation model is equivalent to the model in 
Eq~\eqref{eqn:observation-new-model-vector}. By Theorem 
\ref{theorem:posterior-is-almost-gaussian}, there exist some
$c = c(\eps) > 0$ and $n_0 = n_0(\eps) \in \N$, such that the following holds 
for all $n \geq n_0(\varepsilon)$ and $i \in [m]$: 
\begin{equation}
\E^{\Gamma, c}_{\pi} \tvnorm{G^{\Gamma}(\,\cdot \, | \bY_i) - \pi^{\Gamma,c}(\,\cdot\,|\,\bY_i)} \le \eps,
\end{equation}
Now, for each $i \in [m]$, define the event 
\begin{equation}
\event_i := \{\norm{G^{\Gamma}(\,\cdot \, | \bY_i) - \pi^{\Gamma,c}(\,\cdot\,|\,\bY_i)}_{\TV} \le \eps^{1/2} \}. 
\end{equation} 
By Markov's inequality, we know that 
\begin{equation}
\label{eqn:lower-bound-tv-event}
\prob_{\pi}^{\Gamma, c}(\event_i) \ge 
	1-\eps^{-1/2}\E^{\Gamma, c}_{\pi} \tvnorm{G^{\Gamma}(\cdot \, | \bY_i) - \pi^{\Gamma,c}(\cdot\,|\,\bY_i)}
	\ge 1- \eps^{1/2}.
\end{equation}
Now by definition of $\tvnorm{\cdot}$, we know on event $\event_i$, the following holds for all $x_i \in \R$
\begin{equation}
\label{eqn:conditional-lower-bound-on-good-event}
\E^{\Gamma, c}_{\pi} [(\hx_i - x_i)^2 \wedge K \,|\, \bY_i] \ge 
	\left(\E_G^{\Gamma,c}[(\hx_i - x_i)^2 \wedge K \,|\, \bY_i] - K \varepsilon^{1/2}\right)_+. 
\end{equation}
Meanwhile, if we define the quantity 
\begin{equation*}
J_W^{\Gamma, K} = \E_T [T^2 \wedge K]
	~\,\,\text{for}\,\,~T \sim \normal(0, (\Info_W + \Gamma^{-1})^{-1}), 
\end{equation*}
then Anderson's lemma~\cite{Anderson55} shows that, for all $i \in [m]$, 
\begin{equation}
\label{eqn:conditional-lower-bound-centering}
\inf_{x \in \R} \E_G^{\Gamma}[(\hx_i - x)^2 \wedge K \,|\, \bY_i] 
	\ge \E_G^{\Gamma}[(\hx_i - \E_G^{\Gamma}[\hx_i \mid \bY_i])^2 \wedge K \,|\, \bY_i] 
		= J_W^{\Gamma, K},
\end{equation}
Thus Eq~\eqref{eqn:conditional-lower-bound-on-good-event} and 
Eq~\eqref{eqn:conditional-lower-bound-centering} together imply 
that for all $i \in [m]$ and $x_i \in \R$, 
\begin{equation}
\label{eqn:conditional-lower-bound-almost-final}
\E^{\Gamma, c}_{\pi} [(\hx_i - x_i)^2 \wedge K \,|\, \bY_i] \indic{\event_i} 
	\ge \left(J_W^{\Gamma, K} - K \varepsilon^{1/2}\right)_+ \indic{\event_i}.
\end{equation}
Thus, summing over $i \in [m]$ of Eq~\eqref{eqn:conditional-lower-bound-almost-final},
we get that, 
\begin{equation}
\inf_{x \in \R^m}
 \sum_{i=1}^m \E^{\Gamma, c}_{\pi} [(\hx_i - x_i)^2 \wedge K \,|\, \bY_i]\
	\ge \left(J_W^{\Gamma, K} - K \varepsilon^{1/2}\right)_+ 
		\sum_{i=1}^m  \indic{\event_i}.
\end{equation}
Substituting the above bound into Eq~\eqref{eqn:reduction-to-conditional-expectation-sum},
we get 
\begin{equation}
\label{eqn:conditional-lower-bound-final-except-binary-event}
R(\gamma; r, M, \eta, n) \ge \gamma^{1/4} \cdot (1- K^2m^{-1}t^{-2})_+ \cdot
	\E_{\pi}^{\Gamma, c}\left[\left(\Big(J_W^{\Gamma, K} - K \varepsilon^{1/2}\Big) \cdot 
		\frac{1}{m}\sum_{i=1}^m  \indic{\event_i} - t\right)_+^{1/2}\right].
\end{equation}
Now that $\event_i \in \sigma(\bY_i)$ by definition of $\{\event_i\}_{i\in [m]}$, we see that
the events $\{\event_i\}_{i\in [m]}$ are mutually independent. Using Eq~\eqref{eqn:lower-bound-tv-event} 
and Hoeffding's inequality, we immediately get for any $\eps > 0$, 
\begin{equation}
\label{eqn:Hoeffding-high-prob-lower-bound}
\prob_{\pi}^{\Gamma, c} \left(\frac{1}{m}\sum_{i=1}^m \indic{\event_i} \ge 1- 2\eps^{1/2} \right) \ge 1- \exp(-4m\eps).
\end{equation}
Thus, Markov's inequality implies that, 
\begin{align*}
&\E_{\pi}^{\Gamma, c}\left[\Bigg(\Big(J_W^{\Gamma, K} - K \varepsilon^{1/2}\Big) \cdot 
		\frac{1}{m}\sum_{i=1}^m  \indic{\event_i} - t\Bigg)_+^{1/2}\right] 
\ge (1- \exp(-4m\eps)) \cdot \left((J_W^{\Gamma, K} - K \varepsilon^{1/2}) (1- 2\eps^{1/2})_+ - t\right)_+^{1/2}.
\end{align*}
Substituting the above estimate into Eq 
\eqref{eqn:conditional-lower-bound-final-except-binary-event}, we get that,
\begin{equation*}
R(\gamma; r, M, \eta, n) \ge \gamma^{1/4}\cdot (1- K^2m^{-1}t^{-2})_+(1 - \exp(-4m\eps))
	\left((J_W^{\Gamma, K} - K \varepsilon^{1/2}) (1- 2\eps^{1/2})_+ - t\right)_+^{1/2}.
\end{equation*}
Now, we set $t = m^{-1/3}$. Thus, we get for any $\Gamma, K, \eps > 0$, there exist 
$c = c(\eps), n_0 = n_0(\eps) > 0$ such that for $n \ge n_0$ and $M \ge c$, 
\begin{equation}
\label{eqn:conditional-lower-bound-final}
R(\gamma; r, M, \eta, n) \ge \gamma^{1/4}(1- K^2m^{-1/3})_+ (1 - \exp(-4m\eps))^{1/2} 
	\left((J_W^{\Gamma, K} - K \varepsilon^{1/2}) (1- 2\eps^{1/2})_+ - m^{-1/3}\right)_+^{1/2}.
\end{equation}
Now, take $n \to \infty$. We get for any $\Gamma, K, \eps > 0$, there exists $c = c(\eps) > 0$ 
such that for all $M \ge c$, 
\begin{equation}
\liminf_{n \to \infty, m/n\to \gamma}R(\gamma; r, M, \eta, n)\ge \gamma^{1/4}
	\left(J_W^{\Gamma, K} - K \varepsilon^{1/2}\right)_+^{1/2}(1- 2\eps^{1/2})_+^{1/2}.
\end{equation}
Then, take $M \to \infty$. This shows for any $\Gamma, K, \eps > 0$, 
\begin{equation}
\label{eqn:final-step-lower-bound}
\lim_{M \to \infty}\liminf_{n \to \infty, m/n\to \gamma}R(\gamma; r, M, \eta, n)\ge \gamma^{1/4}
	\left(J_W^{\Gamma, K} - K \varepsilon^{1/2}\right)_+^{1/2}
\end{equation}
Lastly, we take $\Gamma \to \infty, K \to \infty, K\eps^{1/2} \to 0$ on both sides of 
Eq~\eqref{eqn:final-step-lower-bound}. Since the LHS of the above is independent of 
$\eps, K, \Gamma > 0$, and moreover we have, 
\begin{equation}
\lim_{\Gamma \to \infty, K \to \infty, K\eps^{1/2} \to 0} 
	\left(J_W^{\Gamma, K} - K \varepsilon^{1/2}\right)_+^{1/2} = 
		\Info_W^{-1/2},
\end{equation}
and the theorem thus follows. 

\section{Proof outline of Theorem \ref{theorem:upper-bound} (upper bound)}
\label{sec:Techniques_UB}

\subsection{Notation}

For a family of random variables $\{X_{n, \alpha}\}_{n \in \N, \alpha \in A}$ and 
scalars $\alpha_0, c$, we write
\begin{equation}
\lim_{\alpha \to \alpha_0} \lim_{n \to \infty} X_{n, \alpha} = c
\end{equation}
if there exist a (non-random) function $f: A \to \R$ 
and a neighborhood of $\alpha_0$, denoted by $\neb(\alpha_0)$ such that, 
for any fix $\alpha \in \neb(\alpha_0)$, we have, almost surely, 
$\lim_{n \to \infty} X_{n, \alpha} = f(\alpha)$  and further $\lim_{\alpha \to \alpha_0} f(\alpha) = c$. 

We write
\begin{equation}
\limsup_{\alpha \to \alpha_0} \limsup_{n \to \infty} X_{n, \alpha} \le c
\end{equation}
if there exist a (non-random) function $f: A \to \R$ 
and a neighborhood $\neb(\alpha_0)$ of $\alpha_0$, such that, 
for any fix $\alpha \in \neb(\alpha_0)$, we have, almost surely, 
$\limsup_{n \to \infty} X_{n, \alpha} \le f(\alpha)$, and further $\limsup_{\alpha \to \alpha_0} f(\alpha) = c$. 

We write
\begin{equation}
\lim_{\alpha \to \alpha_0} \limsup_{n \to \infty} X_{n, \alpha} = c
\end{equation}
if there exist two (non-random) function $f_1: A \to \R$
$f_2: A \to \R$ and a neighborhood $\neb(\alpha_0)$ of $\alpha_0$,
such that, for any fix $\alpha \in \neb(\alpha_0)$, we have, almost surely, 
$f_1(\alpha) \le \limsup_{n \to \infty} X_{n, \alpha} \le f_2(\alpha)$ 
and $\lim_{\alpha \to \alpha_0} f_1(\alpha) = \lim_{\alpha \to \alpha_0} f_2(\alpha) = c$.

\subsection{Proof outline and key lemmas}

The proof consists of three steps. In the first step, we characterize the behavior of the denoised 
matrix $\widehat{f}_{Y}(\bY)$. Note that this theorem immediately implies the error estimate for $\hbX^{(*)}$ in the rank-unconstrained problem,
namely Eq.~(\ref{eqn:upper-bound-0}).
\begin{theorem}\label{proposition:upper-bound-proposition}
Let $\hf_{Y,\eps}(\,\cdot\,)$, $\hInfo_{W,\eps}$ be defined in Eq.~\eqref{eq:hf_def}.
Then under the assumptions of Theorem \ref{thm:Main0}, the following decomposition holds
\begin{equation}
\label{eqn:decomposition-of-key-estimator}
\widehat{f}_{Y,\eps}(\bY) = \Info_W \bX + \sqrt{\Info_W}\,\bZ+ \bDelta, 
\end{equation}
where the matrices $\bZ\in \R^{m\times n}$ and $\bDelta \in \R^{m \times n}$ 
satisfy the following properties: 
\begin{enumerate}
\item The matrix $\bZ$ is a random matrix whose entries are 
i.i.d mean $0$ and variance $1$. Moreover, there exist constants
$\eps_0, C > 0$ independent of $m$ and $n$, such that almost surely 
for all $\eps \le \eps_0$, 
\begin{equation}
\label{eqn:Z-matrix-bound}
\norm{\bZ}_{\max} \le C\eps^{-1}.
\end{equation}
\item The matrix $\bDelta$ satisfies 
\begin{equation}
\label{eqn:funny-bound-error-decomposition-main-as}
\lim_{\eps \to 0^+} \limsup_{n \to \infty, m/n\to \gamma} 
	\frac{1}{(mn)^{1/4}}\opnorm{\bDelta} = 0. 
\end{equation}
Moreover, for some $\nu \in (0, 1]$, we have 
\begin{equation}
\label{eqn:funny-bound-error-decomposition-main}
\lim_{\eps \to 0^+} \lim_{n \to \infty, m/n\to \gamma} 
	\frac{1}{(mn)^{(1+\nu)/4}}\E \opnorm{\bDelta}^{1+\nu} = 0. 
\end{equation}
\end{enumerate}
Finally, $\hat{\Info}_{W, \eps}$ satisfies 
\begin{equation}
\lim_{\eps\to 0} \limsup_{n \to \infty, m/n \to \gamma}
	\big|\hat{\Info}_{W, \eps} - \Info_W\big| = 0. 
\end{equation}
\end{theorem} 
An outline of the proof of this statement is  presented in Section \ref{sec:ProofDecomposition}, with most  technical work 
deferred to the appendices.

In the second step, we prove an ``almost sure convergence version'' of 
of Theorem~\ref{theorem:upper-bound}. This is formulated 
in a more precise way in Lemma~\ref{lemma:upper-bound-as} below. 
Its proof is deferred to Appendix \ref{sec:proof-lemma-upper-bound-as}.
\begin{lemma}
\label{lemma:upper-bound-as}
Let $(\bX_{n})_{n\ge 1}$ be a sequence of  matrices with $\bX_{n}\in\cF_{m,n}(r,M,\eta)$
where $m=m(n)$ is such that $\lim_{n\to\infty}m(n)/n = \gamma\in(0,\infty)$. Denote
by $\bY_n$ a noisy observation of $\bX_n$.
Under the assumptions of Theorem \ref{thm:Main}, the estimator $\hbX(\,\cdot\,)$ satisfies
\begin{align} 
\label{eqn:upperbound-lowrank}
\lim_{\delta\to 0}\lim_{\eps \to 0}\limsup_{n\to \infty}
	\frac{1}{(mn)^{1/4}}\|\hbX(\bY_n)-\bX_n\|_{\op}
		\le  \big(\gamma^{1/4}\vee \gamma^{-1/4}\big)\, \Info_W^{-1/2}.
\end{align}
\end{lemma}  
The proof of this lemma requires an analysis of the eigen-structure of
the random matrix  $\widetilde{\bY}= \Info_W \bX + \sqrt{\Info_W}\,\bZ$ which is (by the previous theorem) a good
approximation of $\widehat{f}_{Y,\eps}(\bY)$. Note that  $\widetilde{\bY}$  is a low-rank deformation of a matrix with 
independent random entries, 
and a vast amount of work has been devoted to the study of such matrices over the last ten years, see e.g. 
\cite{BBAP05,BS06,paul2007asymptotics,onatski2013asymptotic, BGN12, BloemendalKnYa16, Ding17}. 
Despite the wealth of information available in the literature, only the recent work of Xiucai Ding~\cite{Ding17} considers
our setting, whereby the signal matrix $\bX$ is non-random.  
For the reader's convenience provide an independent proof of the relevant random matrix theory estimates, 
which are stated in Section  \ref{sec:RMT}.

In the third step, we prove that the errors
$\|\hbX(\bY_n)-\bX_n\|_{\op}/(mn)^{1/4}
$ are uniformly integrable, thus allowing to 
translate the almost sure result to a result on the expected error. 
\begin{lemma}
\label{lemma:upper-bound-ui}
Within the same setting of Lemma \ref{lemma:upper-bound-as}, 
there exist some $c > 0, \nu \in (0, 1]$ such that for any $\delta, \eps \le c$, 
\begin{align} 
\limsup_{n\to \infty, m/n\to \gamma}
	\frac{1}{(mn)^{(1+\nu)/4}}\E \left[ \|\hbX(\bY_n)-\bX_n\|_{\op}^{(1+\nu)}\right]
		< \infty.
\end{align}
\end{lemma}
The proof of the lemma is deferred to 
Appendix~\ref{sec:proof-lemma-upper-bound-ui}.

Now, we are ready to prove Theorem~\ref{theorem:upper-bound}. 
Note that, by Lemma~\ref{lemma:upper-bound-as} and Lemma 
\ref{lemma:upper-bound-ui}, we know that, for any $\zeta > 0$, there exists 
some $\delta_0 > 0$, such that, for all $\delta < \delta_0$, there exists some 
$\eps_0 = \eps_0(\delta)$ such that for all $\eps \le \eps_0$, we have almost 
surely,  
\begin{equation}
\label{eqn:main-upper-step-one}
\limsup_{n\to \infty}
	\frac{1}{(mn)^{1/4}}\|\hbX(\bY_n;\eps,\delta)-\bX_n\|_{\op}
		\le \max\{\gamma^{1/4}, \gamma^{-1/4}\}\Info_W^{-1/2}
                + \zeta\, .
\end{equation}
(For the sake of clarity, we made explicit the dependence upon $\eps,\delta$.)
Moreover, for some $\nu > 0$, 
\begin{equation}
\label{eqn:main-upper-step-two}
\limsup_{n\to \infty}
	\frac{1}{(mn)^{(1+\nu))/4}}\E\big\{\|\hbX(\bY_n;\eps,\delta)-\bX_n\|_{\op}^{(1+\nu)}\big\}
		< \infty. 
\end{equation}
Thus, by Eq~\eqref{eqn:main-upper-step-one} and Eq~\eqref{eqn:main-upper-step-two},
we know that, for any $\zeta > 0$, there exists some $\delta_0 > 0$, such that, for all 
$\delta < \delta_0$, there exists some $\eps_0 = \eps_0(\delta)$ such that for all 
$\eps \le \eps_0$, 
\begin{align}
\limsup_{n\to \infty}
	 \frac{1}{(mn)^{1/4}}\E \|\hbX(\bY_n;\eps,\delta)-\bX_n\|_{\op}
		\le \max\{\gamma^{1/4}, \gamma^{-1/4}\}\Info_W^{-1/2} + \zeta\, .
\end{align}
Now notice that $\bX\mapsto\E \|\hbX(\bX+\bW;\eps,\delta)-\bX\|_{\op}$ is a continuous function of $\bX$
in the compact domain $\cF_{m,n}(r,M,\eta)$. Hence, there exists $\bX^*_n$ such that
\begin{align}
\E \|\hbX(\bX^*_n+\bW;\eps,\delta)-\bX^*_n\|_{\op} = \sup_{\bX\in\cF_{m,n}(r,M,\eta)} \E \|\hbX(\bX+\bW;\eps,\delta)-\bX\|_{\op} \, .
\end{align}
Letting $\bY^*_n = \bX^*_n$, we conclude that for any $\zeta > 0$, there exists some $\delta_0 > 0$, 
such that, for all $\delta < \delta_0$, there exists some $\eps_0 = \eps_0(\delta)$ such that for all 
$\eps \le \eps_0$, 
\begin{align*}
\limsup_{n\to \infty}
	 \frac{1}{(mn)^{1/4}}\sup_{\bX\in \cF_{m,n}(r,M,\eta)}\E \|\hbX(\bY;\eps,\delta)-\bX\|_{\op}&=
\limsup_{n\to \infty}
	 \frac{1}{(mn)^{1/4}}\E \|\hbX(\bY_n^*;\eps,\delta)-\bX_n^*\|_{\op}\\
		&\le \max\{\gamma^{1/4}, \gamma^{-1/4}\}\Info_W^{-1/2} + \zeta\, .
\end{align*}
Hence, by taking $\eps_n,\delta_n\downarrow 0$ sufficiently slowly, we obtain
\begin{align*}
\limsup_{n\to \infty}
	 \frac{1}{(mn)^{1/4}}\sup_{\bX\in \cF_{m,n}(r,M,\eta)}\E \|\hbX(\bY;\eps_n,\delta_n)-\bX\|_{\op}&
		\le \max\{\gamma^{1/4}, \gamma^{-1/4}\}\Info_W^{-1/2} \,,
\end{align*}
which coincides with the claim of the theorem.

\subsection{Proof of Theorem  \ref{proposition:upper-bound-proposition}}
\label{sec:ProofDecomposition}

For any $\eps > 0$, define the auxiliary function 
$\targetpert(\, \cdot \,)$ by
\begin{equation}
\targetpert(x)\defeq -\frac{p_W'(x)}{p_W(x) + \eps}. 
\end{equation}
Notice that $\targetpert(x)$ is the `oracle' denoiser, which uses knowledge of the noise distribution $p_W$. 
The proof of Theorem~\ref{proposition:upper-bound-proposition}  proceed in two steps:
\begin{enumerate}
\item We analyze the matrix $f_{W, \eps}(\bY)$ obtained by applying the oracle denoiser, see 
Lemma~\ref{lemma:upper-bound-lemma-one} below.
\item We show that the denoiser $\hf_Y$ behaves similarly to the oracle denoiser  $f_{W, \eps}$, 
for our purposes, cf. Lemma \ref{lemma:upper-bound-lemma-two}.
\end{enumerate}
\begin{lemma}
\label{lemma:upper-bound-lemma-one}
Assume that the conditions of Theorem~\ref{theorem:upper-bound} hold.
Then, we have the following decomposition of the matrix $\targetpert(\bY)$, 
\begin{equation}
\label{eqn:decomposition-of-f-dezero}
\targetpert(\bY) = \Info_W\bX + \sqrt{\Info_W}\bZ + \bar{\bDelta}, 
\end{equation}
where the matrices $\bZ\in \R^{m\times n}$ and $\bar{\bDelta} \in \R^{m \times n}$ 
satisfy the following properties: 
\begin{enumerate}
\item The matrix $\bZ$ is a random matrix whose entries are 
i.i.d mean $0$ and variance $1$. Moreover, there exist constants
$\eps_0, C > 0$ independent of $m$ and $n$, such that almost surely 
for all $\eps \le \eps_0$, 
\begin{equation}
\label{eqn:Z-bound}
\norm{\bZ}_{\max} \le C\eps^{-1}.
\end{equation}
\item The matrix $\bar{\bDelta}$ satisfies
\begin{equation}
\label{eqn:funny-bound-error-decomposition-as}
\lim_{\eps \to 0} \lim_{n \to \infty, m/n\to \gamma} 
	\frac{1}{(mn)^{1/4}}\opnorm{\bar{\bDelta}} = 0. 
\end{equation}
Moreover, we have
\begin{equation}
\label{eqn:funny-bound-error-decomposition}
\lim_{\eps \to 0} \lim_{n \to \infty, m/n\to \gamma} 
	\frac{1}{(mn)^{1/2}}\E \opnorm{\bar{\bDelta}}^2 = 0. 
\end{equation}
\end{enumerate}
\end{lemma}  
The basic idea in this lemma is to control various error terms in the
Taylor expansion  (\ref{eq:TayolorExpansion}).
A complete  proof  is deferred to Appendix~\ref{sec:proof-upper-bound-lemma-one}.

\begin{lemma}
\label{lemma:upper-bound-lemma-two}
Assume that the conditions of Theorem~\ref{theorem:upper-bound} hold.
We have:
\begin{equation}
\label{eqn:upper-bound-lemma-two-one}
\lim_{\eps \to 0^+} \lim_{n \to \infty, m/n\to \gamma}
	\frac{1}{(mn)^{1/4}} \opnorm{\widehat{f}_{Y}(\bY) - \targetpert(\bY)} = 0,
\end{equation}
and for some positive $\nu \in (0, 1]$, 
\begin{equation}
\label{eqn:upper-bound-lemma-two-two}
\lim_{\eps \to 0} \lim_{n \to \infty, m/n\to \gamma}
	\frac{1}{(mn)^{(1+\nu)/4}} \E \opnorm{\widehat{f}_{Y}(\bY) - \targetpert(\bY)}^{(1+\nu)} = 0.
\end{equation} 
Moreover, we have: 
\begin{equation}
\label{eqn:upper-bound-lemma-two-three}
\lim_{\eps \to 0} \lim_{n \to \infty, m/n\to \gamma}
	\big|\hat{\Info}_{W, \eps} - \Info_W\big| = 0.
\end{equation}
\end{lemma}
The key step in proving this lemma is to control  $\max_{i\le m,j\le n}|\hf_Y(Y_{ij})-\targetpert(Y_{ij})|$,
which then allows to bound the operator norm difference $\|\widehat{f}_{Y}(\bY) - \targetpert(\bY)\|_{\op}$.
In order to bound the maximum entry-wise difference, we use uniform bounds on the discrepancy between 
$\hp_W$ and $p_W$, and $\hp'_W$ and $p'_W$ on suitable intervals. 
We refer to 
Appendix~\ref{sec:proof-upper-bound-lemma-two} for the full proof.

It is clear that Theorem~\ref{proposition:upper-bound-proposition} 
follows form the last two lemmas.

\subsection{A random matrix theory result}
\label{sec:RMT}
As mentioned above, a key step in proving Lemma \ref{lemma:upper-bound-as} is to analyze the low-rank plus noise model that is derived
 in Theorem 
\ref{proposition:upper-bound-proposition}. In this section we state a random matrix theory result in order to prove Lemma \ref{lemma:upper-bound-as},
but which is potentially useful in other applications as well. Notice that, in Theorem \ref{proposition:upper-bound-proposition},
the signal and noise components are scaled --respectively-- by $\Info_W$, and $\sqrt{\Info_W}$. However, these scalings can be absorbed into a scaling
of the singular values of $\bX_n$, and an overall scaling of the matrix $\bY_n$, which affects only the singular values on $\bY_n$ and not its singular vectors. For the 
notational simplicity, we will remove these scalings in the present section.

We consider therefore a sequence of matrices $\bX_n \in \R^{m \times n}$ with rank $r$, and singular value decomposition:
\begin{equation}
\bX_n = (mn)^{1/4}\sum_{i=1}^r\sigma_i\bu_i\bv_i^{\sT},\label{eq:XnDef}
\end{equation}
where $\bu_i = \bu_i(n)\in \R^{m}$ and 
$\bv_i=\bv_i(n)\in \R^{n}$ are singular vectors, and $\sigma_1\ge \sigma_2\ge \dots\ge \sigma_r$ are singular values of $\bX_n/(mn)^{1/4}$ in decreasing order .
 This sequence of matrices is indexed by the number of columns, and we will assume the number of rows to be 
$m=m(n)$, such that $\lim_{n\to\infty} m(n)/n=\gamma$. 

Let $\bW_n\in \R^{m \times n}$ be a random matrix 
with i.i.d entries $W_{i, j}$ satisfying $\E \bW_n= \bzero$ and define $\bY_n\in \R^{m \times n}$ to be
\begin{equation}
\bY_n = \bX_n+ \bW_n = (mn)^{1/4}\sum_{i=1}^r\sigma_i\bu_i\bv_i^{\sT}, + \bW_n. \label{eq:YnDef}
\end{equation}
Introduce the singular value decomposition
\begin{equation}
\bY_n = (mn)^{1/4}\sum_{i=1}^{m\wedge n}\hsigma_i\hbu_i\hbv_i^{\sT},\label{eq:YnSVD}
\end{equation}
where  $\hbu_i =\hbu_i(n)\in \R^{m}$,  $\hbv_i =\hbv_j(n)\in \R^{n}$ are left and right singular vectors
and $\hsigma_1\ge \hsigma_2\ge \dots$ the singular values of $\bY_n/(mn)^{1/4}$ in decreasing order.

Recall the definition of $H(\sigma)$ in Eq.~(\ref{eq:Hsigma}), and further define the functions
\begin{equation}
G^{(1)}(\sigma) = 
		\left(\frac{1-\sigma^{-4}}{1+\gamma^{1/2}\sigma^{-2}}\right)^{1/2},
~~~\text{and}~~~
G^{(2)}(\sigma) = 
		\left(\frac{1-\sigma^{-4}}{1+\gamma^{-1/2}\sigma^{-2}}\right)^{1/2}.
\end{equation}
The relevant random matrix theory estimates are stated below, and are essentially a restatement of a theorem in \cite{Ding17}.
\begin{theorem}
\label{theorem:general-deform}
Fix $r \in \N$ and singular values $\sigma_1,\dots,\sigma_r$. Let  $\{(\bu_1(n),\dots,\bu_r(n))\}_{n\in \N}$ and 
$\{(\bv_1(n),\dots,\bv_r(n))\}_{n\in \N}$ be two deterministic sequences of orthonormal 
sets of vectors, and define the matrices $\bX_n$ and $\bY_n$ as in Eqs.~(\ref{eq:XnDef}), (\ref{eq:YnDef}), 
with singular value decomposition (\ref{eq:YnSVD}).  Suppose that the entries of $\bW_n$ are i.i.d with
$\E W_{i, j} = 0$ and $\E W_{i, j}^2 = 1$ and moreover, $\sup_{n\ge 1} \E\{|W_{ij}|^p\}<\infty$ for each $p\ge 1$.
Assume $m=m(n)$ to be such that $\lim_{n\to\infty}m(n)/n=\gamma\in(0,\infty)$.

Then for each $i \in [r]$, 
\begin{equation}
\label{eqn:singular-value-convergence}
\hat{\sigma}_i \asto H(\sigma_i).
\end{equation}
Moreover, assume that for some $k \in [r]$, 
\begin{equation*}
\sigma_1 > \sigma_2 > \ldots > \sigma_k > 1 \ge \sigma_{k+1} 
	\ge \ldots \ge \sigma_r.
\end{equation*}
Then, we have for any $i \in [k]$, 
\begin{equation}
\label{eqn:singular-vector-sign-issue}
 \langle \hbu_i, \bu_i \rangle \langle \hbv_i, \bv_i \rangle \asto 
 	G^{(1)}(\sigma_i) G^{(2)}(\sigma_i).
\end{equation}
Moreover,  for any $i \in [k], j \in [r]$ we have
\begin{align}
\big| \langle \hbu_i, \bu_j \rangle\big| &\asto \big\{G^{(1)}(\sigma_i)\vee G^{(2)}(\sigma_i)\big\}  \indic{i = j}\label{eqn:singular-vector-convergence-A}\\
\big| \langle \hbv_i, \bv_j \rangle\big|&\asto \big\{G^{(1)}(\sigma_i)\wedge G^{(2)}(\sigma_i)\big\} \indic{i = j}. \label{eqn:singular-vector-convergence-B}
\end{align}
\end{theorem}

\begin{remark}
Notice that Theorem~\ref{theorem:general-deform} assumes the top $k$ singular values of the signal matrix $\bX_n/(mn)^{1/4}$ to be distinct.
This is necessary in order to obtain convergence results of the form (\ref{eqn:singular-vector-sign-issue}) that provide a one-to-one correspondence between 
singular vectors of the signal matrix, and singular vectors of the noisy matrix $\bY_n$.
However, no such assumption is made in our main result, Theorem \ref{thm:Main}, or in Lemma \ref{lemma:upper-bound-as}. Indeed, in
the proof of Lemma  \ref{lemma:upper-bound-as} we remove the non-degeneracy assumption via a perturbation argument.
\end{remark}

\begin{remark}
Results analogous to  Theorem~\ref{theorem:general-deform}  have been established under different 
settings~\cite{BGN12, BloemendalKnYa16}.
However, earlier work assumes that either $\bU_n=(\bu_1(n),\dots,\bu_r(n))$ or $\bV_n=(\bv_1(n),\dots,\bv_r(n))$ is random and independent from 
$\bW_n$. Such results cannot used for the present application in which 
both $\bU_n$ and $\bV_n$ are fixed. The only case in which they can be applied is when $\bW_n$ has 
bi-unitarily invariant distribution which, for i.i.d. entries,  implies that $\bW_n$ has Gaussian entries. 
 Using invariance, we can effectively replace $\bU_n$, $\bV_n$ by uniform random rotations of the same matrices.

 Ding \cite{Ding17} addressed the case of deterministic $\bU_n$, $\bV_n$,  establishing the estimates in Theorem \ref{theorem:general-deform},
by using earlier results of Knowles and Yin \cite{knowles2017anisotropic}.
We propose here an independent proof both for readability and to make our assumptions more transparent.
\end{remark}

Our proof  of  Theorem \ref{theorem:general-deform} (cf. Appendix \ref{app:RMT}) follows a well-established strategy, see e.g. \cite{BGN12}.
We express the singular vectors and singular values of $\bY_n$ in terms of the resolvent of the noise $\bW_n$, and the
low-rank perturbation. The technical core of this type of argument is to control the resolvent, which we do using the moments' method.

\subsection*{Acknowledgements}

This work was partially supported by grants NSF DMS-1613091, NSF CCF-1714305 and NSF IIS-1741162
and ONR N00014-18-1-2729.

\bibliographystyle{amsalpha}
\bibliography{bib}

\appendix
\renewcommand{\thesubsection}{\Alph{section}.\arabic{subsection}}

\numberwithin{equation}{section}
\numberwithin{theorem}{section}
\numberwithin{proposition}{section}


\newpage

\newcommand{\bS}{\boldsymbol{S}}
\newcommand{\bE}{\boldsymbol{E}}
\section{Proof of Lemma~\ref{lemma:upper-bound-as}}
\label{sec:proof-lemma-upper-bound-as}
\subsection{Notation}

We introduce the common notation used throughout 
the proof of Lemma~\ref{lemma:upper-bound-as}. By Theorem 
\ref{proposition:upper-bound-proposition}, we know that, 
$\what{\bX}\org(\bY) = \what{f}_{Y, \eps}(\bY)$ has the decomposition below: 
\begin{equation}
\label{eqn:decomposition-f-Y-eps-bY}
\what{\bX}\org(\bY) =  \Info_W \bX + {\sqrt{\Info_W}}\bZ + \bDelta, 
\end{equation}
where $\bDelta$ is a random matrix satisfying 
\begin{equation}
\label{eqn:as-zero-Delta-A}
\lim_{\eps \to 0} \limsup_{n \to \infty, m/n\to \gamma} \frac{1}{(mn)^{1/4}}
	\opnorm{\bDelta} = 0
\end{equation}
and $\bZ$ is some random matrix, whose entries are i.i.d bounded with mean $0$
variance $1$ and moreover, for some constants $\eps_0, C > 0$ independent of 
$m, n$, we have $\normmax{\bZ} \le C\eps^{-1}$ for $\eps \le \eps_0$ and $m, n\in \N$.

Denote the SVD decomposition of $\bX$ to be
\begin{equation}
\bX = (mn)^{1/4} \bU \bSigma \bV^{\sT}
\end{equation}
where $\bU \in \R^{m \times r}$, $\bV \in \R^{n \times r}$ and 
$\bSigma \in \R^{r \times r}$.
Let $k$ be the number of singular values of $\bX$ strictly larger than 
the threshold $\Info_W^{-1/2}$, i.e., 
\begin{equation}
k = \big|\{l: \Sigma_{l, l} > \Info_W^{-1/2}\}\big|, 
\end{equation}
and $\hat{k}$ be the rank of the estimator $\what{\bX}^{(0)}$, i.e., 
\begin{equation}
\label{eqn:def-k}
\hat{k} = \big|\{l: \hat{\Sigma}_{i, i} \neq 0\}\big|. 
\end{equation}
Recall the definition of $\what{\bU}, \what{\bV}, \what{\bSigma}, 
\hbSigma^{(0)}$ from Eq~\eqref{eqn:svd-of-hbX0} and Eq~\eqref{eqn:def-hbSigma}.
Define $\what{\bX}_k \in \R^{m \times n}$ by
\begin{equation}
\what{\bX}_k = (mn)^{1/4}\what{\bU}_k \what{\bSigma}_k \what{\bV}_k^{\sT}. 
\end{equation}
Define $\wtilde{\bX}^{(0)} \in \R^{m \times n}$ and its SVD decomposition as follows:
\begin{equation}
\wtilde{\bX}^{(0)} = \Info_W \bX + {\sqrt{\Info_W}}\bZ
	~~\text{and}~~
	\wtilde{\bX}^{(0)} = (mn)^{1/4}\wtilde{\bU} \wtilde{\bSigma}\org \wtilde{\bV}^{\sT},
\end{equation}
Define $\wtilde{\bSigma}_k \in \R^{k \times k}$ to be a diagonal matrix
with diagonal entries given by
\begin{equation}
\wtilde{\Sigma}_{i, i} = \begin{cases}
	\hat{\Info}_{W, \eps}^{-1/2} H^{-1}(\hat{\Info}_{W, \eps}^{-1/2}\wtilde{\Sigma}_{i, i}\org)
		&\text{if $\wtilde{\Sigma}_{i, i}\org \ge H(1) \hat{\Info}_{W, \eps}^{1/2}$} \\
	0 &\text{otherwise}
	\end{cases}
\end{equation}
Define the rank $k$ matrix $\wtilde{\bX}_k \in \R^{m \times n}$ by
\begin{equation}
\wtilde{\bX}_k = (mn)^{1/4}\wtilde{\bU}_k \wtilde{\bSigma}_k \wtilde{\bV}_k^{\sT}.
\end{equation}

\subsection{Proof Outline}
Our proof of Lemma~\ref{lemma:upper-bound-as} consists of three steps. 
In the first step, we argue that 
$\hat{k} = k$ with high probability, and $\what{\bX}(\bY) \approx
		\what{\bX}_k(\bY)$, as stated precisely in the next lemma. Its proof is 
given in Section~\ref{sec:proof-lemma-proof-upper-bound-first-step}.
\begin{lemma}
\label{lemma:proof-upper-bound-first-step}
We have 
\begin{equation}
\lim_{\delta\to 0}\lim_{\eps \to 0}\lim_{n\to \infty, m/n\to \gamma}
	\frac{1}{(mn)^{1/4}}\opnorm{\what{\bX}(\bY)-
		\what{\bX}_k(\bY)} = 0
\end{equation}
\end{lemma} \noindent\noindent
In the second step, we show that 
$\what{\bX}_k(\bY) \approx \wtilde{\bX}_{k}(\bY)$.
This approximation is formulated rigorously in the next lemma, whose proof is 
given in Section~\ref{sec:proof-lemma-proof-upper-bound-second-step}.
\begin{lemma}
\label{lemma:proof-upper-bound-second-step}
We have
\begin{equation}
\lim_{\eps \to 0}\limsup_{n\to \infty, m/n\to \gamma}
	\frac{1}{(mn)^{1/4}}\opnorm{\what{\bX}_k(\bY)-\wtilde{\bX}_{k}(\bY)} = 0
\end{equation}
\end{lemma}\noindent\noindent
In the last step, we bound the operator norm distance between $\wtilde{\bX}_k(\bY)$ and $\bX$, as stated in the next 
lemma, whose proof is given in Section~\ref{sec:proof-lemma-proof-upper-bound-third-step}.
\begin{lemma}
\label{lemma:proof-upper-bound-third-step}
We have 
\begin{equation}
\lim_{\eps \to 0}\limsup_{n\to \infty, m/n\to \gamma}
	\frac{1}{(mn)^{1/4}}\opnorm{\wtilde{\bX}_k(\bY)- \bX} 
		\le \Info_W^{-1/2}(\gamma^{1/4} + \gamma^{-1/4}).
\end{equation}
\end{lemma}\noindent\noindent
Now the desired claim of Lemma~\ref{lemma:upper-bound-as} follows 
easily by 
Lemma~\ref{lemma:proof-upper-bound-first-step}, 
Lemma~\ref{lemma:proof-upper-bound-second-step} and
Lemma~\ref{lemma:proof-upper-bound-third-step}. 

\subsubsection{Proof of Lemma~\ref{lemma:proof-upper-bound-first-step}} 
\label{sec:proof-lemma-proof-upper-bound-first-step}
We start by proving the following 
(recall the definition of $\hat{k}$ at Eq~\eqref{eqn:def-k})
\begin{equation}
\label{eqn:limit-of-hat-k}
\lim_{\delta\to 0}\lim_{\eps \to 0} \lim_{n\to \infty, m/n\to \gamma} \hat{k} = k.
\end{equation}
To do so, we first analyze the limiting singular values of $\wtilde{\bX}^{(0)}(\bY)$. The 
lemma below is useful. 
\begin{lemma}
\label{lemma:singular-value-of-bA}
There exists some $\eps_0 > 0$ such that, for all $\eps \le \eps_0$, we have, 
\begin{enumerate}
\item For $i \in [r]$, we have almost surely, 
\begin{equation}
\label{eqn:spectral-converge-r}
\lim_{n \to \infty, m/n\to \gamma} \frac{1}{(mn)^{1/4}}\sigma_i (\wtilde{\bX}^{(0)}(\bY)) 
	= \Info_W^{1/2} H(\Info_W^{1/2}\sigma_i),
\end{equation}
\item For $i = r+1$, we have almost surely, 
\begin{equation} 
\label{eqn:spectral-converge-r+1}
\limsup_{n \to \infty, m/n\to \gamma} 
	\frac{1}{(mn)^{1/4}}\sigma_{r+1}(\wtilde{\bX}^{(0)}(\bY)) 
	\le \Info_W^{1/2} H(1)\, . 
\end{equation}
\end{enumerate}
\end{lemma}
\begin{proof}
Both Eq~\eqref{eqn:spectral-converge-r} and Eq~\eqref{eqn:spectral-converge-r+1}
follow directly from Theorem~\ref{theorem:general-deform}.
\end{proof}
Note that $\sigma_k > \Info_W^{-1/2}$ by definition of $k$. Fix any 
$\delta_0^\prime$ such that, 
\begin{equation}
\sigma_k > (1+\delta_0^\prime) \Info_W^{-1/2}. 
\end{equation}
By Weyl's inequality, we have the bound
\begin{equation}
\label{eqn:bound-by-bDelta}
\max_{l \in \{k, k+1\}} \Big|\sigma_l(\what{\bX}^{(0)}(\bY)) - \sigma_l(\wtilde{\bX}^{(0)}(\bY)) \Big| 
		\le \opnorm{\bDelta}, 
\end{equation}
Since by definition of $\bDelta$, we have (see Eq~\eqref{eqn:as-zero-Delta-A})
\begin{equation}
\label{eqn:limit-of-Delta}
\lim_{\eps \to 0} \limsup_{n\to \infty, m/n\to \gamma}
	\frac{1}{(mn)^{1/4}}\opnorm{\bDelta} = 0, 
\end{equation}
Lemma~\ref{lemma:singular-value-of-bA}, Eq~\eqref{eqn:bound-by-bDelta}
and Eq~\eqref{eqn:limit-of-Delta} imply that, 
\begin{equation}
\label{eqn:limit-of-hat-f-Y-Y-inf}
\lim_{\eps \to 0}\liminf_{n\to \infty, m/n\to \gamma} 
	\frac{1}{(mn)^{1/4}}\sigma_k(\what{\bX}^{(0)}(\bY)) 
		\ge \Info_W^{1/2}H\left(1+\delta_0^\prime\right)
\end{equation}
and
\begin{equation}
\label{eqn:limit-of-hat-f-Y-Y-sup}
\lim_{\eps \to 0} \limsup_{n\to \infty, m/n\to \gamma} 
	\frac{1}{(mn)^{1/4}}\sigma_{k+1}(\what{\bX}^{(0)}(\bY))
		 \le\Info_W^{1/2} H(1). 
\end{equation}
Now Theorem~\ref{proposition:upper-bound-proposition} shows that 
\begin{equation}
\label{eqn:converge-hat-W-eps}
\lim_{\eps \to 0}\lim_{n\to \infty, m/n\to \gamma} 
	\left|\hat{\Info}_{W, \eps} - \Info_{W}\right| = 0. 
\end{equation}
As the function $\sigma \to H(\sigma)$ is strictly increasing on $[1, \infty)$, 
Eq~\eqref{eqn:limit-of-hat-f-Y-Y-inf}, Eq~\eqref{eqn:limit-of-hat-f-Y-Y-sup}
and Eq~\eqref{eqn:converge-hat-W-eps} together show that, for $\delta > 0$
sufficiently small satisfying $(1+\delta) H(1) < H(1+\delta_0^\prime)$,
\begin{equation}
\lim_{\eps \to 0} \lim_{n\to \infty, m/n\to \gamma} \hat{k} = k. 
\end{equation}
This proves the claim at Eq~\eqref{eqn:limit-of-hat-k}. In particular,
\begin{equation}
\lim_{\delta\to 0}\lim_{\eps \to 0}\lim_{n\to \infty, m/n\to \gamma}
	\frac{1}{(nm)^{1/4}}\|\what{\bX}(\bY)-
		\what{\bX}_k(\bY)\|_{\op} = 0,
\end{equation}
giving the desired claim of the Lemma. 

\renewcommand{\opt}{^*}
\renewcommand{\H}{\mathbb{H}}
\newcommand{\imax}{^{\rm max}}
\subsubsection{Proof of Lemma~\ref{lemma:proof-upper-bound-second-step}}
\label{sec:proof-lemma-proof-upper-bound-second-step}
The main idea of our proof is the following: we view $\what{\bX}\org(\bY)$ as a 
perturbation of $\wtilde{\bX}\org(\bY)$, and then we establish a general matrix 
perturbation result to bound the difference 
$\opnormbig{\what{\bX}(\bY) - \wtilde{\bX}(\bY)}$. It turns out that such bound 
on $\opnormbig{\what{\bX}(\bY) - \wtilde{\bX}(\bY)}$ gives the
desired claim of the lemma. 

We start by introducing our matrix perturbation result.  Let 
$\bA \in \R^{m \times n}$ be a matrix with singular value decomposition
\begin{equation}
\bA = \bP \bS \bQ^{\sT},
\end{equation}
where $\bP \in \R^{m \times m}$, $\bQ \in \R^{n \times n}$ and 
$\bS \in \R^{m \times n}$. Define $\bA_k \in \R^{m \times n}$
to be 
\begin{equation}
\bA_k = \bP_k \bS_k \bQ_k^{\sT},
\end{equation} 
the best rank $k$ approximation for $\bA$ under any 
unitarily invariant norm~\cite[Theorem 4.18]{StewartSun90}. 
Now, for a function $f: \R_+\to \R_+$, define the matrix
$f(\bA_k) \in \R^{m \times n}$ by
\begin{equation}
f(\bA_k) = \bP_k f(\bS_k) \bQ_k^{\sT},
\end{equation} 
where $f(\bS_k) \in \R^{k \times k}$ is the diagonal matrix we get by 
applying $f$ entry-wisely to the diagonal of $\bS_k$. 

Let $\wtilde{\bA} \in \R^{m \times n}$ denote a perturbation 
of matrix $\bA$: 
\begin{equation}
	\wtilde{\bA} = \bA+ \bE,
\end{equation} 
where $\bE\in \R^{m \times n}$ is the error matrix. 
Define the matrix $\wtilde{\bA}_k$ and $f(\wtilde{\bA}_k)$ in an 
analogous way. Intuitively, we expect that (under suitable
assumptions on  the matrix $\bA$ and function $f$)
$f(\wtilde{\bA}_k) \approx f(\bA_k)$
when the error matrix $\bE$ is ``small". 
Theorem~\ref{theorem:nonlinear-PCA-perturbation} provides a bound of this type.

\begin{theorem}
\label{theorem:nonlinear-PCA-perturbation}
Let $\wtilde{\bA}, \bA, \bE \in \R^{m \times n}$ be matrices such that
\begin{equation}
\wtilde{\bA} = \bA+ \bE.
\end{equation} 
Assume for some continuous function $f$ and some constants $L, \vartheta, \tau, \zeta > 0$
and $\alpha \in (0, 1]$, we have 
\begin{enumerate}
\item The function $f$ is $(L, \alpha)$ H\"{o}lder continuous on $[\tau, \zeta]$, i.e., for all $x_1, x_2 
	\in [\tau, \zeta]$, 
	\begin{equation}
	\label{eqn:f-Holder-continuous}
		\big| f(x_1) - f(x_2)\big| \le L \big| x_1 - x_2\big|^{\alpha}.
	\end{equation}
\item The following inequalities hold
	\begin{equation}
	\label{eqn:singular-gap}
	\zeta > \sigma_1(\bA),~
		\sigma_k(\bA) > \max\{\sigma_{k+1}(\bA), \tau\} + \vartheta
	~~\text{and}~~
		\vartheta > 2\opnorm{\bE}.
	\end{equation}
\end{enumerate} 
Then, the perturbation bound below holds for $f(\bA_k)$ and $ f(\wtilde{\bA}_k)$:
\begin{equation}
\label{eqn:crazy-PCA-pert-bound}
\opnormbig{f(\bA_k) - f(\wtilde{\bA}_k)} \le 
	4k L \opnorm{\bE}^{\alpha} + \frac{2}{\vartheta}f(\sigma_k(\bA))\opnorm{\bE}.
\end{equation}
\end{theorem}
We defer the proof of this theorem
to Appendix~\ref{sec:matrix-perturbation}.

Define for $\eps > 0$ the random variable $\tau_\eps$ and the (random) function
$f_{\eps}$ by 
\begin{equation}
\label{eqn:def-f-eps}
\tau_{\eps} = \hat{\Info}_{W, \eps}^{1/2}H(1)~~\text{and}~~
f_{\eps}(\sigma) = \begin{cases}
	\hat{\Info}_{W, \eps}^{-1/2} H^{-1}(\hat{\Info}_{W, \eps}^{-1/2} \sigma)~
		&\text{if $\sigma \ge \tau_{\eps}$} \\
	0~&\text{otherwise}
	\end{cases}.
\end{equation}
By definition, we have 
\begin{equation}
\what{\bX}(\bY) = f_{\eps}(\what{\bX}\org(\bY))
	~\text{and}~
\wtilde{\bX}(\bY) = f_{\eps}(\wtilde{\bX}\org(\bY))
\end{equation}
Now, viewing $\what{\bX}\org(\bY)$ as a perturbation of $\wtilde{\bX}\org(\bY)$,
we wish to use Theorem~\ref{theorem:nonlinear-PCA-perturbation} to give upper 
bounds on
\begin{equation*}
\opnorm{\what{\bX}(\bY) - \wtilde{\bX}(\bY)} = 
	\opnormbigg{f_{\eps}(\what{\bX}\org(\bY)) - f_{\eps}(\wtilde{\bX}\org(\bY))}.
\end{equation*} 
To apply Theorem~\ref{theorem:nonlinear-PCA-perturbation}, we need to check 
two conditions. As our first step, we show that, for any $\eps, \zeta > 0$, the function 
$f_{\eps}$ is H\"{o}lder continuous on $[\tau_{\eps}, \zeta]$. 
\begin{lemma}
\label{lemma:CBF-g-eps}
For any $\eps > 0, \zeta > 0$, $f_{\eps}$ is $(4\hat{\Info}_{W, \eps}^{-1}\zeta^{3/4}, 
	\frac{1}{4})$ H\"{o}lder continuous on 
$[\tau_{\eps}, \zeta]$, i.e., 
\begin{equation}
|f_{\eps}(x_1) - f_{\eps}(x_2)|\le 4 \hat{\Info}_{W, \eps}^{-1}\zeta^{3/4}|x_1 - x_2|^{1/4}
\end{equation}
\end{lemma}
\begin{proof}
For notational simplicity, we denote $h$ to be the constant
\begin{equation}
h = \gamma^{1/2} + \gamma^{-1/2}
\end{equation}
By elementary computations, we have for all $x \ge H(1)$, 
\begin{equation}
H^{-1}(x) = \frac{1}{\sqrt{2}} \left(x^2- h + 
	\left((x^2- h)^2-4\right)^{1/2}\right)^{1/2}.
\end{equation}
Noting the elementary fact that $|x^{1/2} - y^{1/2}| \le |x-y|^{1/2}$
for any $x, y\in \R_+$, we have for $x_1, x_2 \ge H(1)$
\begin{align}
\left|H^{-1}(x_1) - H^{-1}(x_2)\right|
	&\le \left(\left|x_1^2 - x_2^2\right| + \left|\left((x_1^2- h)^2-4\right)^{1/2}
		- \left((x_2^2- h)^2-4\right)^{1/2}\right|\right)^{1/2} \nonumber \\
	&\le \left|x_1^2 - x_2^2\right|^{1/2} + \left|(x_1^2- h)^2 -  (x_2^2- h)^2\right|^{1/4}
	\le 2|x_1 - x_2|^{1/4} |x_1 + x_2|^{3/4}. \nonumber
\end{align}
Hence, we have, for any $\eps > 0$ and $x_1, x_2 \in [\tau_{\eps}, \infty)$, 
\begin{equation}
\left|f_{\eps}(x_1) - f_{\eps}(x_2)\right|
= \hat{\Info}_{W, \eps}^{-1/2}
\left|H^{-1}(\hat{\Info}_{W, \eps}^{-1/2}x_1) - H^{-1}(\hat{\Info}_{W, \eps}^{-1/2} x_2)\right|
\le 2\hat{\Info}_{W, \eps}^{-1}|x_1 - x_2|^{1/4} |x_1 + x_2|^{3/4}.
\end{equation}
The estimate above gives the desired claim of Lemma~\ref{lemma:CBF-g-eps}. 
\end{proof}

%
%
\indent Next, we bound the largest singular value and singular gap of
	$\wtilde{\bX}\org$.
\begin{lemma}
\label{lemma:singular-gap-bA}
There exist constants $\eps_0, \vartheta_0, C_0 > 0$ independent of $m, n$ such that, 
for all $\eps \le \eps_0$, 
\begin{equation}
\label{eqn:top-singular-bound}
\limsup_{n \to \infty, m/n \to \gamma} 
	\frac{1}{(mn)^{1/4}}
		\sigma_1(\wtilde{\bX}\org)
				\le C_0,
\end{equation}
and
\begin{equation}
\label{eqn:singular-gap-as}
\liminf_{n \to \infty, m/n \to \gamma} 
	\frac{1}{(mn)^{1/4}}
		\min\{\sigma_k(\wtilde{\bX}\org) - \sigma_{k+1}(\wtilde{\bX}\org), 
			~\sigma_k(\wtilde{\bX}\org) - \tau_{\eps}\big\}
				\ge \vartheta_0.
\end{equation}
\end{lemma}
\begin{proof}
First of all, the first claim at Eq~\eqref{eqn:top-singular-bound} follows directly from 
Lemma~\ref{lemma:singular-value-of-bA}. To prove the second claim at 
Eq~\eqref{eqn:singular-gap-as}, fix any $\delta_0^\prime$ such 
that, 
\begin{equation}
\label{eqn:def-delta-0}
\sigma_k > (1+\delta_0^\prime) \Info_W^{-1/2}. 
\end{equation}
Then Lemma~\ref{lemma:singular-value-of-bA} implies that, 
\begin{equation}
\label{eqn:sin-upper-bound}
\lim_{\eps \to 0}\liminf_{n\to \infty, m/n\to \gamma} 
	\frac{1}{(mn)^{1/4}}\sigma_k(\wtilde{\bX}\org) 
		\ge \Info_{W}^{1/2}H(1+ \delta_0^\prime)
\end{equation}
and
\begin{equation}
\label{eqn:sin-lower-bound}
\lim_{\eps \to 0} \limsup_{n\to \infty, m/n\to \gamma} 
	\frac{1}{(mn)^{1/4}}\sigma_{k+1}(\wtilde{\bX}\org)
		 \le \Info_{W}^{1/2}H(1). 
\end{equation}
Now, by Theorem~\ref{proposition:upper-bound-proposition}, we 
know the following convergence
\begin{equation}
\lim_{\eps \to 0} \limsup_{n\to \infty, m/n\to \gamma} 
	\big|\hat{\Info}_{W, \eps} - \Info_W\big| = 0,
\end{equation}
this implies that, 
\begin{equation}
\label{eqn:tau-eps-limit}
\lim_{\eps \to 0} \limsup_{n\to \infty, m/n\to \gamma} 
	\big|\tau_{\eps} -  \Info_{W}^{1/2} H(1)\big| = 0.
\end{equation}
Thus, if we set $\vartheta_0$ to be 
\begin{equation}
\vartheta_0 = \half \Info_{W}^{1/2} \left(H(1+\delta_0^\prime) - H(1)\right) > 0,
\end{equation}
then Eq~\eqref{eqn:sin-upper-bound}, Eq~\eqref{eqn:sin-lower-bound} and 
Eq~\eqref{eqn:tau-eps-limit} imply that 
\begin{equation}
\lim_{\eps \to 0}\liminf_{n \to \infty, m/n \to \gamma} 
	\frac{1}{(mn)^{1/4}}
		\min\{\sigma_k(\wtilde{\bX}\org) - \sigma_{k+1}(\wtilde{\bX}\org), 
			~\sigma_k(\wtilde{\bX}\org) - \tau_{\eps}\big\}
				> \vartheta_0.
\end{equation}
This implies the second claim at Eq~\eqref{eqn:singular-gap-as}.
\end{proof}

Now we are ready to apply Theorem~\ref{theorem:nonlinear-PCA-perturbation} to 
the matrices $\bA = \wtilde{\bX}\org(\bY)$, $\wtilde{\bA} = \what{\bX}\org(\bY)$ and the 
function $f = f_{\eps}$. Indeed pick $\eps_0, \vartheta_0, C_0$ such that the 
statement of Lemma~\ref{lemma:singular-gap-bA} holds. By Eq~\eqref{eqn:as-zero-Delta-A},
we know for any $\bar{\Delta} > 0$, there exists some $\eps_1 > 0$ such that for 
$\eps \le \eps_1$, 
\begin{equation}
\limsup_{n \to \infty, m/n \to \gamma} \frac{1}{(mn)^{1/4}}\opnorm{\bDelta} 
	\le \bar{\Delta}. 
\end{equation}
Note that by definition, 
\begin{equation}
\what{\bX}(\bY) = f_{\eps}(\what{\bX}\org(\bY))
~~\text{and}~~ \wtilde{\bX}(\bY) = f_{\eps}(\wtilde{\bX}\org(\bY)).
\end{equation}
By Lemma~\ref{lemma:CBF-g-eps}, Lemma~\ref{lemma:singular-gap-bA} and 
Theorem~\ref{theorem:nonlinear-PCA-perturbation}, we know that, when 
$\eps \le \eps_0 \wedge \eps_1$ and $\bar{\Delta} < \vartheta_0/2$,
\begin{align}
&\limsup_{n \to \infty, m/n \to \gamma} \frac{1}{(mn)^{1/4}}
	\opnormbig{\wtilde{\bX}_{k}(\bY) - \what{\bX}_k(\bY)}  \\
&\le \left(\limsup_{n \to \infty, m/n \to \gamma}\hat{\Info}_{W, \eps}^{-1}\right)
	\cdot 4k C_0^{3/4}\bar{\Delta}^{1/4} + 
	 \left(\limsup_{n \to \infty, m/n \to \gamma}f_{\eps}(C_0)\right) 
	 	\cdot \frac{2}{\vartheta_0}\bar{\Delta}. 
\label{eqn:crazy-PCA-pert-bound-f-g-eps-Delta}
\end{align}
Now, denote the quantity $\tau$ and the function $f$ to be  
\begin{equation}
\label{eqn:def-g-eps-limit}
\tau = \Info_W^{-1/2} H(1),~~
\bar{f}(\sigma) =\begin{cases}
		\Info_W^{-1/2} H^{-1}(\Info_W^{-1/2}\sigma)~~&\text{if $\sigma \ge \tau$}\\
		\Info_W^{-1/2} &\text{otherwise}
	\end{cases}.
\end{equation}
As Theorem~\ref{proposition:upper-bound-proposition} shows the following convergence
\begin{equation}
\label{eqn:again-Info-W-converge}
\lim_{\eps \to 0}\limsup_{n \to \infty, m/n \to \gamma} 
	\big|\hat{\Info}_{W, \eps} - \Info_W\big| = 0, 
\end{equation}
we get for any $a > 0$,
\begin{equation}
\label{eqn:limit-of-f-g-eps-function}
\lim_{\eps \to 0} \limsup_{n \to \infty, m/n \to \gamma} 
	\big(f_{\eps}(a) - \bar{f}(a)\big) \le 0.
\end{equation}
Now, using Eq~\eqref{eqn:again-Info-W-converge} and 
Eq~\eqref{eqn:limit-of-f-g-eps-function}, and taking
$\eps \to 0$ on both sides of Eq~\eqref{eqn:crazy-PCA-pert-bound-f-g-eps-Delta},
we get that, 
\begin{equation}
\label{eqn:one-step-up-to-pert-limit-Delta}
\lim_{\eps \to 0}\limsup_{n \to \infty, m/n \to \gamma} \frac{1}{(mn)^{1/4}}
	\opnormbig{\what{\bX}_k(\bY)- \wtilde{\bX}_{k}}
\le 4k \Info_W^{-1}  C_0^{3/4}\bar{\Delta}^{1/4} + 
	 \bar{f}(C_0) \cdot \frac{2\bar{\Delta}}{\vartheta_0}
\end{equation}
Note that the LHS of above is independent of $\bar{\Delta} < \vartheta_0/2$. 
Taking $\bar{\Delta}\to 0$ in Eq~\eqref{eqn:one-step-up-to-pert-limit-Delta} 
gives us
\begin{equation}
\lim_{\eps \to 0}\limsup_{n \to \infty, m/n \to \gamma} 
	\frac{1}{(mn)^{1/4}}\opnormbig{\what{\bX}_k(\bY)- \wtilde{\bX}_{k}} = 0.
\end{equation}
This proves the desired claim of Lemma~\ref{lemma:proof-upper-bound-second-step}.

\subsubsection{Proof of Lemma~\ref{lemma:proof-upper-bound-third-step}}
\label{sec:proof-lemma-proof-upper-bound-third-step}
In the proof below, we assume without loss of generality that $\gamma \le 1$. 
Define the auxiliary matrix $\wtilde{\bX}_k\aux \in \R^{m \times n}$ by 
\begin{equation}
\wtilde{\bX}_k\aux = (mn)^{1/4}\wtilde{\bU}_k \bSigma_k \wtilde{\bV}_k^{\sT}
\end{equation}
We divide our proof of Lemma~\ref{lemma:proof-upper-bound-third-step}
into two cases. In the first case, we consider 
the situation where the top singular values $\{\sigma_i\}_{i \in [k]}$ are distinct, i.e., 
\begin{equation}
\sigma_1 > \sigma_2 > \ldots > \sigma_k > \Info_W^{-1/2}. 
\end{equation}
By triangle inequality, we know that, 
\begin{align}
\label{eqn:dif-op-tilde-A}
 \frac{1}{(mn)^{1/4}}\norm{\wtilde{\bX}_k -  \bX}_{\op}
	&\le \frac{1}{(mn)^{1/4}}\norm{\wtilde{\bX}_k- \wtilde{\bX}_k\aux }_{\op}
		+  \frac{1}{(mn)^{1/4}}\norm{\wtilde{\bX}_k\aux - \bX}_{\op} \notag \\
	&= \opnorm{\wtilde{\bSigma}_k - \bSigma_k}
		+ \opnorm{\wtilde{\bU}_k \bSigma_{k} \wtilde{\bV}_k^{\sT} - \bU \bSigma \bV^{\sT}}.
\end{align}
Now, we upper bound the RHS of Eq~\eqref{eqn:dif-op-tilde-A} separately. The goal 
is to show that, 
\begin{equation}
\label{eqn:first-error-bound}
\lim_{\eps \to 0} \lim_{n\to \infty, m/n\to \gamma}
	\opnorm{\wtilde{\bSigma}_k - \bSigma_k} = 0.
\end{equation}
and 
\begin{equation}
\label{eqn:second-error-bound}
\lim_{\eps \to 0} \limsup_{n\to \infty, m/n\to \gamma}
	\opnorm{\wtilde{\bU}_k \bSigma_{k} \wtilde{\bV}_k^{\sT} - \bU \bSigma \bV^{\sT}} \le
		\max\{\gamma^{1/4}, \gamma^{-1/4}\} \Info_W^{-1/2}.
\end{equation}
We start by proving Eq~\eqref{eqn:first-error-bound}.
By Lemma~\ref{lemma:singular-value-of-bA} and 
Theorem~\ref{proposition:upper-bound-proposition}, 
we know that for $i \in [k]$
\begin{equation}
\label{eqn:one-funny-convergence}
\lim_{n\to \infty, m/n\to \gamma}
	\wtilde{\bSigma}\org_{i, i} 
= \lim_{n\to \infty, m/n\to \gamma}
	\frac{1}{(mn)^{1/4}}\sigma_i(\wtilde{\bX}\org(\bY))
	= \Info_W^{1/2} H\big(\Info_W^{1/2}\sigma_i\big)
\end{equation}
and 
\begin{equation}
\label{eqn:two-funny-convergence}
\lim_{\eps \to 0} \lim_{n\to \infty, m/n\to \gamma} \hat{\Info}_{W, \eps} = \Info_W.
\end{equation}
Since $\sigma \to H^{-1}(\sigma)$ is continuous on $[H(1), \infty)$, 
Eq~\eqref{eqn:one-funny-convergence} and Eq~\eqref{eqn:two-funny-convergence} 
give for $i \in [k]$,  
\begin{equation}
\label{eqn:first-error-bound-key}
\lim_{\eps \to 0} \lim_{n\to \infty, m/n\to \gamma} 
	\tilde{\Sigma}_{i, i} = \Info_W^{-1/2} H^{-1}\left(H\left(\Info_W^{1/2}\sigma_i\right) \right)
		= \sigma_i,
\end{equation}
where the last identity uses the fact that $\sigma_i > \Info_W^{-1/2}$ for $i \in [k]$.
Hence, using Eq~\eqref{eqn:first-error-bound-key}, we have, 
\begin{equation}
\lim_{\eps \to 0} \lim_{n\to \infty, m/n\to \gamma}
	\opnorm{\wtilde{\bSigma}_k - \bSigma_k} 
= \lim_{\eps \to 0} \lim_{n\to \infty, m/n\to \gamma}\max_{i \in [k]}
	\left|\tilde{\Sigma}_{i, i} - \sigma_i\right| = 0, 
\end{equation}
which gives the desired bound in Eq~\eqref{eqn:first-error-bound}.

Next, we show Eq~\eqref{eqn:second-error-bound}. Fix $\eps > 0$ first. Now we construct two 
auxiliary orthonormal matrices $\bQ_u \in \R^{m \times m}$ and $\bQ_v \in \R^{n \times n}$.
First, define the two sets, 
\begin{equation}
S_{\bu} = \{\wtilde{\bu}_1,\ldots,\wtilde{\bu}_{k}, \bu_1,\ldots,\bu_r\}
	~~\text{and}~~
S_{\bv} = \{\wtilde{\bv}_1,\ldots,\wtilde{\bv}_{k}, \bv_1,\ldots,\bv_r\}.
\end{equation}
Now, run the Gram-Schmidt orthogonalization process on the sets of vectors $S_{\bu}$
and $S_{\bv}$ separately. As the first $k$ vectors of $S_{\bu}$ (and same for $S_{\bv}$) 
are already orthogonal to each other, we may assume the output vectors of the process 
takes the form of 
\begin{equation}
S_{\bu}^{\rm GS} = \{\wtilde{\bu}_1,\ldots,\wtilde{\bu}_{k},\bu_{1}^{\prime},
\ldots,\bu_{r}^{\prime}\}~~\text{and}~~
S_{\bv}^{\rm GS} = \{\wtilde{\bv}_1,\ldots,\wtilde{\bv}_{k},\bv_{1}^{\prime},
\ldots,\bv_{r}^{\prime}\}.
\end{equation}
The orthonormal matrices $\bQ_u$ and $\bQ_u$ are defined by any two orthonormal 
matrices whose first $k+r$ columns are $S_{\bu}^{\rm GS}$ and $S_{\bv}^{\rm GS}$
respectively. Since both $\bQ_u \in \R^{m \times m}$ and $\bQ_v \in \R^{n \times n}$
are orthonormal matrices, we have, 
\begin{equation}
\label{eqn:transformation-GS}
\opnorm{\wtilde{\bU}_k \bSigma_{k} \wtilde{\bV}_k^{\sT} - \bU \bSigma \bV^{\sT}}
= \opnorm{\bQ_u^{\sT}\wtilde{\bU}_k \bSigma_{k} \wtilde{\bV}_k^{\sT}\bQ_v - \bQ_u^{\sT}\bU \bSigma \bV^{\sT}\bQ_v}
\end{equation}
Denote the matrix $\bGamma \in \R^{m \times n}$ to be
\begin{equation}
\label{eqn:def-b-Gamma-sub}
\bGamma = \bQ_u^{\sT}\wtilde{\bU}_k \bSigma_{k} \wtilde{\bV}_k^{\sT}\bQ_v - \bQ_u^{\sT}\bU \bSigma \bV^{\sT}\bQ_v.
\end{equation}
Then, by construction of $\bQ_u$ and $\bQ_v$, we know that the matrix $\bGamma$ has 
$0$ entries except its top left $(k+r) \times (k+r)$ sub-matrix. Denote this sub-matrix to be 
$\bGamma\subs \in \R^{(k+r) \times (k+r)}$. For each $i \in [k]$, define the angles 
$\alpha_i, \beta_i \in (0, \pi/2)$ through the following: 
\begin{equation}
\cos{\alpha_i} =G^{(1)} (\sigma_i;\Info_W)~~\text{and}~~\cos{\beta_i} = G^{(2)} (\sigma_i;\Info_W),
\end{equation}
where we define the functions $G^{(1)}(\sigma; t)$ and $G^{(2)}(\sigma; t)$ for 
$\sigma, t > 0, \sigma^2 t \ge 1$ by
\begin{equation}
G^{(1)}(\sigma; t) = \left(\frac{1-t^{-2}\sigma^{-4}}{1+\gamma^{1/2}t^{-1}\sigma^{-2}}\right)^{1/2}
~\text{and}~
G^{(2)}(\sigma; t) = \left(\frac{1-t^{-2}\sigma^{-4}}{1+\gamma^{-1/2}t^{-1}\sigma^{-2}}\right)^{1/2}
\end{equation}
Moreover, introduce the following two by two matrices for $i \in [k]$,
\begin{equation}
\overline{\bDelta}^{(i)} = \sigma_i
\begin{bmatrix}
(1 - \cos \alpha_i \cos \beta_i) &  \sin{\alpha_i} \cos \beta_i \\
\sin{\beta_i} \cos \alpha_i &- \sin \alpha_i \sin{\beta_i}
\end{bmatrix}.
\end{equation}
Now we prove that, almost surely we have 
\begin{equation}
\label{eqn:converge-Gamma-sub}
\lim_{n \to \infty} \bGamma\subs = 
	\begin{bmatrix}
		\overline{\Delta}^{(1)}_{1, 1} &0 & \cdots & 0 & \overline{\Delta}^{(1)}_{1, 2}
			& 0 & \cdots &  0 & \cdots & 0\\
		0 & \overline{\Delta}^{(2)}_{1, 1} & \cdots & 0 & 0 & \overline{\Delta}^{(2)}_{1, 2}
		 & \cdots & 0 &\cdots & 0\\
		0 & 0 & \cdots & 0 & 0 & 0 & \cdots & 0 & \cdots & 0\\
		0 &0 & \cdots &  \overline{\Delta}^{(k)}_{1, 1} & 0 & 0 & \cdots &  
			\overline{\Delta}^{(k)}_{1, 2} &\cdots & 0\\
		\overline{\Delta}^{(1)}_{2, 1} &0 & \cdots & 0 & \overline{\Delta}^{(1)}_{2, 2}
			& \cdots & 0 & 0 & \cdots & 0\\
		0 & \overline{\Delta}^{(2)}_{2, 1} & \cdots & 0 & 0 & \overline{\Delta}^{(2)}_{2, 2}
		 & \cdots & 0 & \cdots & 0\\
		0 & 0 & \cdots & 0 & 0 & 0 & \cdots & 0 & \cdots & 0\\
		0 &0 & \cdots &  \overline{\Delta}^{(k)}_{2, 1} & 0 & 0 & \cdots &  
			\overline{\Delta}^{(k)}_{2, 2} &  \cdots & 0\\
		0 & 0 & 0 & 0 & 0 & 0 & 0 & 0 &  \cdots & 0\\
		\cdots & \cdots & \cdots & \cdots & \cdots & \cdots & \cdots & \cdots & \cdots & 0\\
		0 & 0 & 0 & 0 & 0 & 0 & 0 & 0 & \cdots & 0
	\end{bmatrix}
\end{equation}
Since $\gamma \le 1$, by Theorem~\ref{theorem:general-deform}, we 
can without loss of generality (by flipping the sign of $\{\wtilde{\bu}_i\}_{i\in [k]}$
and $\{\wtilde{\bv}_i\}_{i\in [k]}$ if necessary) assume that, for $i \in [k]$ 
and $j \in [r]$, 
\begin{equation}
\label{eqn:corr-tilde-u-u}
\langle \wtilde{\bu}_i, \bu_j\rangle \asto 
	\begin{cases}
		\cos(\alpha_i)~~&\text{if $i \ne j$} \\
		0~~&\text{otherwise}
	\end{cases}
~~\text{and}~~
\langle \wtilde{\bv}_i, \bv_j\rangle \asto 
	\begin{cases}
		\cos(\beta_i)~~&\text{if $i \ne j$} \\
		0~~&\text{otherwise}
	\end{cases}
\end{equation}
By definition of the Gram-Schmidt orthogonalization, we can without loss 
of generality (by flipping the sign of $\{\bu_i^\prime\}_{i\in [k]}$
and $\{\bv_i^\prime\}_{i\in [k]}$ if necessary) assume that, for
$i \in [k]$ and $j \in [r]$, 
\begin{equation}
\label{eqn:corr-u-prime-u}
\langle \bu_i^\prime, \bu_j\rangle \asto 
	\begin{cases}
		\sin(\alpha_i)~~&\text{if $i \ne j$} \\
		0~~&\text{otherwise}
	\end{cases}
~~\text{and}~~
\langle \bv_i^\prime, \bv_j\rangle \asto 
	\begin{cases}
		\sin(\beta_i)~~&\text{if $i \ne j$} \\
		0~~&\text{otherwise}
	\end{cases}
\end{equation}
Now, using Eq~\eqref{eqn:corr-tilde-u-u} and Eq~\eqref{eqn:corr-u-prime-u}, we see that, 
\begin{equation}
\label{eqn:converge-U-V-easy}
\bQ_u^{\sT}\wtilde{\bU}_k \asto 
	\begin{bmatrix}
	1 & 0 & \ldots & 0 \\
	0 & 1 & \ldots & 0 \\
	\cdots & \cdots & \cdots & \cdots \\
	0 & 0 & \ldots & 1 \\
	0 & 0 & \ldots & 0 \\
	\cdots & \cdots & \cdots & \cdots \\
	0 & 0 & \ldots & 0
	\end{bmatrix}
~~\text{and}~~
\bQ_v^{\sT}\wtilde{\bV}_k \asto 
	\begin{bmatrix}
	1 & 0 & \ldots & 0 \\
	0 & 1 & \ldots & 0 \\
	\cdots & \cdots & \cdots & \cdots \\
	0 & 0 & \ldots & 1 \\
	0 & 0 & \ldots & 0 \\
	\cdots & \cdots & \cdots & \cdots \\
	0 & 0 & \ldots & 0
	\end{bmatrix}
\end{equation}
and
\begin{equation}
\label{eqn:converge-U-V-hard}
\bQ_u^{\sT}\bU \asto 
\begin{bmatrix}
	\cos(\alpha_1) & 0 & \ldots & 0 \\
	0 & \cos(\alpha_2) & \ldots & 0 \\
	\cdots & \cdots & \cdots & \cdots \\
	0 & 0 & \ldots & \cos(\alpha_k) \\
	\sin(\alpha_1) & 0 & \ldots & 0 \\
	0 & \sin(\alpha_2) & \ldots & 0 \\
	\cdots & \cdots & \cdots & \cdots \\
	0 & 0 & \ldots & \sin(\alpha_k) \\
	0 & 0 & \ldots & 0 \\
	\cdots & \cdots & \cdots & \cdots \\
	0 & 0 & \ldots & 0
	\end{bmatrix}
~~\text{and}~~
\bQ_u^{\sT}\bU \asto 
\begin{bmatrix}
	\cos(\beta_1) & 0 & \ldots & 0 \\
	0 & \cos(\beta_2) & \ldots & 0 \\
	\cdots & \cdots & \cdots & \cdots \\
	0 & 0 & \ldots & \cos(\beta_k) \\
	\sin(\beta_1) & 0 & \ldots & 0 \\
	0 & \sin(\beta_2) & \ldots & 0 \\
	\cdots & \cdots & \cdots & \cdots \\
	0 & 0 & \ldots & \sin(\beta_k) \\
	0 & 0 & \ldots & 0 \\
	\cdots & \cdots & \cdots & \cdots \\
	0 & 0 & \ldots & 0
	\end{bmatrix}
\end{equation}
Now, Eq~\eqref{eqn:converge-Gamma-sub} follows by Eq~\eqref{eqn:converge-U-V-easy}, 
Eq~\eqref{eqn:converge-U-V-hard} and the definition of $\bGamma\subs$ in 
Eq~\eqref{eqn:def-b-Gamma-sub}.

Following Eq~\eqref{eqn:transformation-GS} and Eq~\eqref{eqn:converge-Gamma-sub}, this 
shows that, 
\begin{equation}
\label{eqn:converge-to-two-by-two}
\lim_{n\to \infty, m/n\to \gamma}
	\opnormbig{\wtilde{\bU}_k \bSigma_{k} \wtilde{\bV}_k^{\sT} - \bU \bSigma \bV^{\sT}}
=  \lim_{n\to \infty, m/n\to \gamma}\opnormbig{\bGamma\subs}
	= \max_{i \in [k]} \opnormbig{\bDelta^{(i)}}. 
\end{equation}
Now, we evaluate the RHS of Eq~\eqref{eqn:converge-to-two-by-two}. To do so, 
for any $\alpha, \beta \in (0, \pi)$, define the matrix 
\begin{equation}
\label{eqn:block-singular-value}
\bT(\alpha, \beta) = 
\begin{bmatrix} 
    1- \cos \alpha \cos \beta & \sin{\alpha}\cos{\beta} \\
    \cos{\alpha}\sin{\beta} & -\sin{\alpha} \sin{\beta}
\end{bmatrix}.
\end{equation}
By elementary computations, the two singular values of $\bT(\alpha, \beta)$ are exactly
\begin{equation}
2\cos\left(\alpha /2\right)\sin\left(\beta/2\right)~~\text{and}~~
2\cos\left(\beta/2 \right)\sin\left(\alpha/2 \right)\, ,
\end{equation}
and hence if we define the function $J(\alpha, \beta)$ by 
\begin{equation}
J(\alpha, \beta) = 
2\max\left\{\cos\left(\frac{\alpha}{2}\right)\sin\left(\frac{\beta}{2}\right), 
	\cos\left(\frac{\beta}{2}\right)\sin\left(\frac{\alpha}{2}\right)\right\}.
\end{equation}
then we have for $\alpha, \beta \in (0, \pi)$, $\opnormbig{\bT(\alpha, \beta)} 
= J(\alpha, \beta)$. Hence, using Eq~\eqref{eqn:converge-to-two-by-two}, we 
get that 
\begin{equation}
\lim_{n\to \infty, m/n\to \gamma}
	\opnormbig{\wtilde{\bU}_k \bSigma_{k} \wtilde{\bV}_k^{\sT} - \bU \bSigma \bV^{\sT}}= 
	\max_{i \in [k]} \opnormbig{\bDelta^{(i)}} = 
	\max_{i \in [k]}  2\sigma_i J(\alpha_i, \beta_i)
\end{equation}
Hence, to prove Eq~\eqref{eqn:second-error-bound}, it suffices to show the 
estimate below, 
\begin{equation}
\label{eqn:crucial-estimate}
\max_{i \in [k]} 2\sigma_i J(\alpha_i, \beta_i) 
	\le \max\left\{\gamma^{1/4}, \gamma^{-1/4}\right\}\Info_W^{-1/2}. 
\end{equation}
Denote the function $T: \R_+ \times [0, 1] \to \R_+$ to be
\begin{equation*}
T(t, y) \defeq 
	\left(1 + \left(\frac{1- y^2}{1 + t^{-1} y}\right)^{1/2}\right) 
		\left( 1- \left(\frac{1- y^2}{1 + t y}\right)^{1/2} \right).
\end{equation*}
By elementary computations, we know that, for $i \in [k]$, 
\begin{equation}
\label{eqn:representation-of-J-into-T}
2\sigma_i J(\alpha_i, \beta_i) = \sigma_i
	T^{1/2}\left(\max\{\gamma^{1/2}, \gamma^{-1/2}\}, 
		\sigma_i^{-2} \Info_W^{-1}\right). 
\end{equation}
Now we show that, for any $t \ge 1$ and $y \in [0, 1]$, 
\begin{equation}
\label{eqn:elementary-funny-claim}
T(t, y) \le ty.
\end{equation}
Indeed, introduce the function $\bar{T}(t, y)\defeq T(t, y) - ty$. Then, by computation, 
we have
\begin{equation}
\label{eqn:derivative-T}
\frac{\partial }{\partial t} \bar{T}(t, y)
	= y \tilde{T}(t,y), 
\end{equation}
where the function $\tilde{T}(t,y)$ is defined by
\begin{align}
\tilde{T}(t,y) &\defeq \half (1- y^2)^{1/2} \left( (1+t^{-1}y)^{-3/2}t^{-2} + 
	(1+ty)^{-3/2} \right)  \notag \\
	&+ \half  (1- y^2) (1+t^{-1}y)^{-3/2}(1+ty)^{-3/2} (1-t^{-2}) - 1.
\end{align}
Note that when $t\ge 1$, $\tilde{T}(t, y)\le 0$ as $\tilde{T}(t, 1) = 0$ and by inspection, 
$\tilde{T}(t, y)$ is nonincreasing in $y$ for any fix $t \in [1, \infty)$. Thus, by 
Eq~\eqref{eqn:derivative-T}, we know that, when $t \ge 1$, 
\begin{equation}
\frac{\partial }{\partial t} \bar{T}(t, y) \le 0, 
\end{equation} 
which implies that the function $t \to \bar{T}(t, y)$ is decreasing in $t \in [1, \infty)$ for 
any fix $y$. Since $\bar{T}(1, y) = T(1, y) - y = 0$ for any $y \in [0, 1]$, this implies
that $\bar{T}(t, y) \le 0$ for any $t\ge 1$ and $y \in [0, 1]$, giving the desired claim 
at Eq~\eqref{eqn:elementary-funny-claim}. Now, with Eq~\eqref{eqn:representation-of-J-into-T}
and Eq~\eqref{eqn:elementary-funny-claim}, we conclude that, for $i \in [k]$, 
\begin{equation}
2\sigma_i J(\alpha_i, \beta_i) 
	\le \sigma_i \left(\max\big\{\gamma^{1/2}, \gamma^{-1/2}\big\}\sigma_i^{-2}\Info_W^{-1}\right)^{1/2}
	= \max\left\{\gamma^{1/4}, \gamma^{-1/4}\right\}\Info_W^{-1/2}.
\end{equation} 
This proves Eq~\eqref{eqn:crucial-estimate}, and as mentioned, it leads to the 
desired claim at Eq~\eqref{eqn:second-error-bound}.

In the second case, we consider the situation where some elements of $\{\sigma_i\}_{i \in [k]}$
may coincide. We reduce ourselves to the first case using a perturbation argument outlined below. Indeed, for any $\{\pert_i\}_{i \in [k]}$ 
such that $\{\sigma_i + \pert_i\}_{i \in [k]}$ are distinct, we define
\begin{equation}
\bX(\pert) = \bX + (mn)^{1/4}\sum_{i=1}^k \pert_i \, \bu_i \bv_i^{\sT}.
\end{equation}
Now, for such $\{\pert_i\}_{i \in [k]}$, define $\wtilde{\bX}\org(\pert) \in \R^{m\times n}$ 
and its SVD decomposition of $\wtilde{\bX}^{(0)}$ to be 
\begin{equation}
\wtilde{\bX}\org(\pert) = \Info_W \bX(\pert) + \sqrt{\Info_W} \bZ
~~\text{and}~~
\wtilde{\bX}^{(0)}(\pert) = (mn)^{1/4}\wtilde{\bU}(\pert) \wtilde{\bSigma}\org(\pert) \wtilde{\bV}(\pert)^{\sT},
\end{equation}
Denote $\wtilde{\bSigma}_k(\pert) \in \R^{k \times k}$ to be the diagonal matrix
with diagonal entries defined by, 
\begin{equation}
\wtilde{\Sigma}_{i, i}(\pert) = \begin{cases}
	\hat{\Info}_{W, \eps}^{-1/2} H^{-1}(\hat{\Info}_{W, \eps}^{-1/2}\wtilde{\Sigma}_{i, i}\org(\pert))
		&\text{if $\wtilde{\Sigma}_{i, i}\org(\pert) \ge H(1) \hat{\Info}_{W, \eps}^{1/2}$} \\
	0 &\text{otherwise}
	\end{cases},
\end{equation}
and the rank $k$ matrix $\wtilde{\bX}_k(\pert) \in \R^{m \times n}$ by
\begin{equation}
\wtilde{\bX}_k(\pert) = (mn)^{1/4}\wtilde{\bU}_k(\pert) \wtilde{\bSigma}_k(\pert) \wtilde{\bV}_k(\pert)^{\sT}.
\end{equation}

For notational shorthand, denote $\pert\imax = \max_{i \in [k]} 
\big|\pert_i\big|$. Since $\sigma_k > \sigma_{k+1}$, by construction, when the 
value $\pert\imax$ is sufficiently small, the set of the top $k$ singular values of 
$\bX(\pert)$ is precisely the set $\{\sigma_i + \pert_i\}_{i \in [k]}$. Thus by our 
choice of $\{\pert_i\}_{i \in [k]}$, the top $k$ singular values of $\bX(\pert)$
are pairwise different. Hence, we may use the established result in the first case to 
conclude that, 
\begin{equation}
\lim_{\eps \to 0} \limsup_{n\to \infty, m/n\to \gamma}
	\frac{1}{(mn)^{1/4}}\norm{\wtilde{\bX}_k(\pert) -  \bX}_{\op}
		\le \max\{\gamma^{1/4}, \gamma^{-1/4}\} \Info_W^{-1/2}.	
\end{equation}
To prove our desired Eq~\eqref{eqn:second-error-bound}, it suffices to show that,
\begin{equation}
\label{eqn:perturbation-not-matter-PCA}
\lim_{\pert \to 0}\lim_{\eps \to 0} \limsup_{n\to \infty, m/n\to \gamma}
	\frac{1}{(mn)^{1/4}}\norm{\wtilde{\bX}_k(\pert) -  \wtilde{\bX}_k}_{\op}
		= 0.
\end{equation}
To prove Eq~\eqref{eqn:perturbation-not-matter-PCA}, we use the same idea 
as that of proving Lemma~\ref{lemma:proof-upper-bound-second-step}. Viewing 
$\wtilde{\bX}\org(\pert)$ as a perturbed version of $\wtilde{\bX}\org$, 
we may use the perturbation result Theorem~\ref{theorem:nonlinear-PCA-perturbation}
to upper bound 
\begin{equation}
\opnorm{\wtilde{\bX}(\pert) -  \wtilde{\bX}} 
= \opnormBig{f_{\eps}\left(\wtilde{\bX}\org(\pert)\right) -  f_{\eps}\left(\wtilde{\bX}\org\right)},
\end{equation}
where we recall the definition of $f_{\eps}$ at Eq~\eqref{eqn:def-f-eps}. It turns 
out that such bound suffices for our proof of Eq~\eqref{eqn:perturbation-not-matter-PCA}.

To be more concrete, note first that, by definition
\begin{equation}
\opnormbig{\wtilde{\bX}(\pert) -  \wtilde{\bX}} = \pert\imax.
\end{equation} 
Pick $\eps_0, \vartheta_0, C_0 > 0$ such that the statement of Lemma 
\ref{lemma:singular-gap-bA} holds. Recall $f_{\eps}$ at Eq~\eqref{eqn:def-f-eps}. Since 
\begin{equation}
\wtilde{\bX}(\pert) =  f_{\eps}\left(\wtilde{\bX}\org(\pert)\right)
	~\text{and}~
\wtilde{\bX} = f_{\eps}\left(\wtilde{\bX}\org\right).
\end{equation}
Lemma~\ref{lemma:CBF-g-eps},
Lemma~\ref{lemma:singular-gap-bA} and 
Theorem~\ref{theorem:nonlinear-PCA-perturbation} imply that when
$\eps \le \eps_0$ and $\pert\imax < \vartheta_0/2$,
\begin{align}
&\limsup_{n \to \infty, m/n \to \gamma} \frac{1}{(mn)^{1/4}}
	\opnormbig{\wtilde{\bX}_k(\pert) -  \wtilde{\bX}_k}  \nonumber\\
&\le \left(\limsup_{n \to \infty, m/n \to \gamma} \hat{\Info}_{W, \eps}^{-1}\right)
		\cdot  4kC_0^{3/4} (\pert\imax)^{1/4}
	+ \left( \limsup_{n \to \infty, m/n \to \gamma}
		f_{\eps}(C_0) \right)\cdot \frac{2\pert\imax}{\vartheta_0}
\label{eqn:crazy-PCA-pert-bound-f-g-eps}
\end{align}
Recall $\bar{f}$ at Eq~\eqref{eqn:def-g-eps-limit}. By Eq~\eqref{eqn:limit-of-f-g-eps-function}, 
if we take limit $\eps \to 0$ in Eq~\eqref{eqn:crazy-PCA-pert-bound-f-g-eps}, we get  
\begin{equation}
\label{eqn:one-step-up-to-pert-limit}
\lim_{\eps \to 0}\limsup_{n \to \infty, m/n \to \gamma} \frac{1}{(mn)^{1/4}}
	\opnormbig{\wtilde{\bX}_k(\pert) -  \wtilde{\bX}_k}
\le 4k\Info_W^{-1} C_0^{3/4} (\pert\imax)^{1/4}
+ \bar{f}(C_0) \cdot \frac{2\pert\imax}{\vartheta_0}
\end{equation}
Take $\pert\to 0$ on both sides of Eq~\eqref{eqn:one-step-up-to-pert-limit}. This shows
\begin{equation}
\lim_{\pert \to 0}\lim_{\eps \to 0}\limsup_{n \to \infty, m/n \to \gamma} 
	\frac{1}{(mn)^{1/4}}\opnormbig{\wtilde{\bX}_k(\pert) -  \wtilde{\bX}_k} = 0,
\end{equation}
the desired claim at Eq~\eqref{eqn:perturbation-not-matter-PCA}. 
As mentioned, this proves Eq~\eqref{eqn:second-error-bound}.


\section{Proof of Lemma~\ref{lemma:upper-bound-ui}}
\label{sec:proof-lemma-upper-bound-ui}
First, note that, by H\"{o}lder's inequality, we have for any $\nu \in (0, 1]$: 
\begin{equation}
\label{eqn:ui-upper-bound-one}
\frac{1}{(mn)^{(1+\nu)/4}}  \E \left[
	\|\what{\bX}(\bY)-\bX\|_{\op}^{(1+\nu)}\right]
\le \frac{2}{(mn)^{(1+\nu)/4}} \left(\E \left[\|\what{\bX}(\bY)\|_{\op}^{(1+\nu)}\right]
	+ \E \left[\|\bX\|_{\op}^{(1+\nu)}\right]\right).
\end{equation}
Since $\bX \in \cF_{m,n}(r, M, \eta)$, we know that, for any $\nu \in (0, 1]$, 
\begin{equation}
 \frac{1}{(nm)^{(1+\nu)/4}} \E \left[\|\bX\|_{\op}^{(1+\nu)} \right] \le M^{(1+\nu)} < \infty.
\end{equation}
Hence, using Eq~\eqref{eqn:ui-upper-bound-one}, it suffices to show for some 
$\nu \in (0, 1]$, we have for any $\eps, \delta$ small enough
\begin{equation}
\label{eqn:ui-upper-bound-goal-one}
\limsup_{n\to \infty, m/n\to \gamma}
	\frac{1}{(mn)^{(1+\nu)/4}}\E \left[ \|\what{\bX}(\bY)\|_{\op}^{(1+\nu)}\right]
		< \infty.
\end{equation}
To show Eq~\eqref{eqn:ui-upper-bound-goal-one}, we start by providing a 
deterministic upper bound on $\|\what{\bX}(\bY)\|_{\op}$, which is, 
\begin{equation}
\label{eqn:ui-upper-bound-two}
\opnormbig{\what{\bX}(\bY)}
	\le \eps^{-1}\opnormbig{\what{\bX}\org(\bY)}.
\end{equation}
Indeed, we first note that, 
\begin{equation}
\opnormbig{\what{\bX}(\bY)} 
	= \begin{cases}
		\hat{\Info}_{W, \eps}^{-1/2} H^{-1}\Big(\hat{\Info}_{W, \eps}^{-1/2}
			\opnormbig{\what{\bX}\org(\bY)}\Big) ~~&\text{if 
				$\opnormbig{\what{\bX}\org(\bY)} \ge 
					(1+\delta)H(1)\hat{\Info}_{W, \eps}^{1/2}$} \\
		0~~&\text{otherwise}.
	\end{cases}
\label{eqn:relate-what-to-what-org}
\end{equation}
Next, by inspection $H(\sigma) \ge \sigma$ for $\sigma \ge 1$. Thus 
$H^{-1}(\sigma) \le \sigma$ for $\sigma > H(1)$. Hence, 
Eq~\eqref{eqn:relate-what-to-what-org} implies 
\begin{equation}
\opnormbig{\what{\bX}(\bY)}
\le \hat{\Info}_{W, \eps}^{-1} \opnormbig{\what{\bX}\org(\bY)}
\end{equation}
Since by definition $\hat{\Info}_{W, \eps} \ge \eps$, this implies the desired 
Eq~\eqref{eqn:ui-upper-bound-two}. Now, Eq~\eqref{eqn:ui-upper-bound-two} 
implies for $\nu \in (0, 1]$, 
\begin{equation}
\opnormbig{\what{\bX}(\bY)}^{(1+\nu)}
	\le \eps^{-(1+\nu)}\opnormbig{\what{\bX}\org(\bY)}^{1+\nu}.
\end{equation}
Now to show Eq~\eqref{eqn:ui-upper-bound-goal-one}, it suffices to show
for some $\nu \in (0, 1], c > 0$, we have for any $\eps, \delta \le c$
\begin{equation}
\label{eqn:ui-upper-bound-goal-two}
\limsup_{n\to \infty, m/n\to \gamma}
	\frac{1}{(mn)^{(1+\nu)/4}}\E \left[\opnormbig{\what{\bX}\org(\bY)}^{1+\nu}\right]
		< \infty.
\end{equation}

Now, we prove Eq~\eqref{eqn:ui-upper-bound-goal-two}. By Theorem
\ref{proposition:upper-bound-proposition}, there exist some constants $\eps_0 > 0, 
\nu_0 \in (0, 1]$ such that for any $\eps \le \eps_0$, the matrix $\what{\bX}\org(\bY) 
= \what{f}_Y(\bY)$ has the decomposition: 
\begin{equation}
\what{\bX}\org(\bY) = \Info_W \bX + \sqrt{\Info_W}\bZ + \bDelta, 
\end{equation}
where $\bZ$ satisfies Eq~\eqref{eqn:Z-matrix-bound} and $\bDelta$ satisfies
Eq~\eqref{eqn:funny-bound-error-decomposition-main} for $\nu = \nu_0$. 
Fix this $\eps_0, \nu_0$. By H\"{o}lder's inequality, 
\begin{align}
&\frac{1}{(mn)^{(1+\nu_0)/4}}\E \opnormbig{\what{\bX}\org(\bY) }^{(1+\nu_0)} \nonumber \\
&	\le \frac{3}{(mn)^{(1+\nu_0)/4}}\left(\Info_W^{(1+\nu_0)}
		\opnorm{\bX}^{(1+\nu_0)}+ \Info_W^{(1+\nu_0)/2}
		\E\opnorm{\bZ}^{(1+\nu_0)} + \E\opnormbig{\bDelta}^{(1+\nu_0)} \right).
\label{eqn:holder-step-goal-two}
\end{align}
First, since $\bX \in \cF_{m,n}(r, M, \eta)$, we know that 
\begin{equation}
\label{eqn:first-exp-upper-bound}
\limsup_{n\to \infty, m/n\to \gamma} 
	\frac{1}{(mn)^{(1+\nu_0)/4}} \opnormbig{\bX}^{(1+\nu_0)} \le M^{(1+\nu_0)}. 
\end{equation}
Next, when $\eps \le \eps_0$, we know from Eq~\eqref{eqn:Z-matrix-bound} 
that $\normmax{\bZ}< C\eps^{-1}$ for some constant $C > 0$ independent 
of $m, n, \eps$. Thus, when $\eps \le \eps_0$, \cite{Guionnet}
\begin{equation}
\label{eqn:second-exp-upper-bound}
\limsup_{n\to \infty, m/n\to \gamma}\frac{1}{(mn)^{(1+\nu_0)/4}} 
	 \E\opnorm{\bZ}^{(1+\nu_0)} = (\gamma^{1/4} + \gamma^{-1/4})^{(1+\nu_0)}.
\end{equation}
Finally according to Eq~\eqref{eqn:funny-bound-error-decomposition-main}, we know
for sufficiently small $\eps$, 
\begin{equation}
\label{eqn:third-exp-upper-bound}
\limsup_{n\to \infty, m/n\to \gamma} \frac{1}{(mn)^{(1+\nu_0)/4}} 
	\E \opnorm{\bDelta}^{(1+\nu_0)} < \infty. 
\end{equation}
Substituting Eq~\eqref{eqn:first-exp-upper-bound}, 
Eq~\eqref{eqn:second-exp-upper-bound} and 
Eq~\eqref{eqn:third-exp-upper-bound} into Eq~\eqref{eqn:holder-step-goal-two}, 
one proves Eq~\eqref{eqn:ui-upper-bound-goal-two}. As mentioned, 
Eq~\eqref{eqn:ui-upper-bound-goal-two} implies Eq~\eqref{eqn:ui-upper-bound-goal-one}, 
which implies the desired Lemma~\ref{lemma:upper-bound-ui}.


\section{Proof of Lemma~\ref{lemma:upper-bound-lemma-one}}
\label{sec:proof-upper-bound-lemma-one}
To start with, since $Y_{i, j} = X_{i, j} + W_{i, j}$ for $i \in [m], j\in [n]$, Taylor's expansion 
gives 
\begin{equation}
\label{eqn:taylor-expansion-f-dezero}
\targetpert(Y_{i, j}) = \targetpert(W_{i, j}) + X_{i, j} \targetpert^\prime(W_{i, j})
	+ \half X_{i, j}^2 \targetpert^{\prime\prime}(W_{i, j} + \theta_{i, j}), 
\end{equation}
for some $\theta_{i, j}$ satisfying $|\theta_{i, j}| \leq |X_{i, j}|$. Now, denote 
$\wbar{\mu}_{\eps}$, $\wbar{\mu}_{\eps}^\prime$ and $\wbar{\nu}_{\eps}^2$ 
to be
\begin{equation}
\wbar{\mu}_{\eps}= \E[\targetpert(W_{i, j})],~ 
\wbar{\mu}_{\eps}^\prime= \E[\targetpert(W_{i, j})]
~\text{and}~
\wbar{\nu}_{\eps}^2 = \Var[\targetpert(W_{i, j})].
\end{equation}
A key observation is that $\bar{\mu}_{\eps} = 0$. Thus, if we define
\begin{equation}
\bZ \defeq \wbar{\nu}_{\eps}^{-1}\left(\targetpert(\bW) - \bar{\mu}_{\eps}\right)
	= \wbar{\nu}_{\eps}^{-1}\targetpert(\bW),
\end{equation}
Taylor's expansion at Eq~\eqref{eqn:taylor-expansion-f-dezero} implies the 
following decomposition, 
\begin{equation}
\label{eqn:final-expansion}
\targetpert(\bY) = \Info_W\bX + \sqrt{\Info_W} \bZ + \bar{\bDelta}^{(1)} + \bar{\bDelta}^{(2)}
	+ \bar{\bDelta}^{(3)} + \bar{\bDelta}^{(4)},
\end{equation}
where we introduce the matrices $\bar{\bDelta}^{(1)}, \bar{\bDelta}^{(2)}, 
\bar{\bDelta}^{(3)}, \bar{\bDelta}^{(4)} \in \R^{m\times n}$ to be:
\begin{align*}
\bar{\bDelta}^{(1)}_{i, j} \defeq \left(\targetpert^\prime(W_{i, j}) - \bar{\mu}_{\eps}^\prime\right)X_{i, j},~
\bar{\bDelta}^{(2)}_{i, j} \defeq \half X_{i, j}^2 \targetpert^{\prime\prime}(W_{i, j} + \theta_{i, j})
\end{align*}
and 
\begin{align*}
\bar{\bDelta}^{(3)}_{i, j} \defeq (\wbar{\nu}_{\eps}- \sqrt{\Info_W}) Z_{i, j},~~
\bar{\bDelta}^{(4)}_{i, j} \defeq (\wbar{\mu}_{\eps}^\prime- \Info_W) X_{i, j}.
\end{align*}
Define the matrix $\bar{\bDelta}$ by
\begin{equation}
\bar{\bDelta} \defeq \bar{\bDelta}^{(1)} +  \bar{\bDelta}^{(2)} +  \bar{\bDelta}^{(3)} + \bar{\bDelta}^{(4)}.
\end{equation} 
In the rest of the proof, we show that $\bZ$ and $\bar{\bDelta}$ satisfy the statements 
of the lemma. 

Consider the matrix $\bZ$ first. By definition, $\bZ$ is a random matrix whose entries 
are i.i.d mean $0$ and variance $1$. Moreover, by dominated convergence theorem
\begin{equation}
\label{eqn:limit-dominated-eps}
\lim_{\eps \to 0} \Info_{W, \eps} = \Info_W~\text{and}~
	\lim_{\eps \to 0} \wbar{\nu}_{\eps}^2 = \Info_W.
\end{equation}
Thus for some $\eps_0 > 0$, we have 
$(\wbar{\nu}_{\eps})^{-1} \le 2  \Info_W^{-1/2}$ for all $\eps \le \eps_0$. 
Hence, for $\eps \le \eps_0$,
\begin{equation}
\label{eqn:normmax-Z}
\normmax{\bZ}  \le \wbar{\nu}_{\eps}^{-1} \norm{f_{W, \eps}(\cdot)}_\infty
	\le 2\eps^{-1}\Info_W^{-1/2}\norm{p_W^\prime(\cdot)}_\infty.
\end{equation}
Since $\norm{p_W^\prime(\cdot)}_\infty \le M_2$ by Assumption~{\sf A2}, this shows 
that for some constant $C > 0$
\begin{equation}
\normmax{\bZ} \le C\eps^{-1}
\end{equation}
for $\eps \le \eps_0$. Together, we prove that $\bZ$ satisfies the statements of the lemma. 

Next, we consider the matrix $\bar{\bDelta}$. We need to prove that $\bar{\bDelta}$
satisfies Eq~\eqref{eqn:funny-bound-error-decomposition-as} and
Eq~\eqref{eqn:funny-bound-error-decomposition}.
To start with, by triangle inequality, we have, 
\begin{equation}
\opnorm{\bar{\bDelta}} \le \opnormbig{\bar{\bDelta}^{(1)}} + \opnormbig{\bar{\bDelta}^{(2)}}
		+ \opnormbig{\bar{\bDelta}^{(3)}}+ \opnormbig{\bar{\bDelta}^{(4)}}.
\end{equation}
and 
\begin{equation}
\opnorm{\bar{\bDelta}}^2 \le 4 \left[\opnormbig{\bar{\bDelta}^{(1)}}^2 + \opnormbig{\bar{\bDelta}^{(2)}}^2
	+ \opnormbig{\bar{\bDelta}^{(3)}}^2 + \opnormbig{\bar{\bDelta}^{(4)}}^2\right].
\end{equation}
Thus, it suffices to show that for any $j \in [4]$, 
\begin{equation}
\label{eqn:funny-bound-error-sep-decomposition-as}
\lim_{\eps \to 0} \lim_{n \to \infty, m/n \to \gamma} \frac{1}{(mn)^{1/4}} 
	\opnormbig{\bar{\bDelta}^{(j)}} = 0
\end{equation}
and for any $j \in [4]$, 
\begin{equation}
\label{eqn:funny-bound-error-sep-decomposition}
\lim_{\eps \to 0} \lim_{n \to \infty, m/n \to \gamma} \frac{1}{(mn)^{1/2}} 
	\E \opnormbig{\bar{\bDelta}^{(j)}}^2 = 0.
\end{equation}
We show the proof of Eq~\eqref{eqn:funny-bound-error-sep-decomposition-as} 
and Eq~\eqref{eqn:funny-bound-error-sep-decomposition} for 
$\bar{\bDelta}^{(1)}$, $\bar{\bDelta}^{(2)}$, $\bar{\bDelta}^{(3)}$ 
and $\bar{\bDelta}^{(4)}$ one by one in the four consecutive paragraphs below.

%

\paragraph{Proof of Eq~\eqref{eqn:funny-bound-error-sep-decomposition-as}
and Eq~\eqref{eqn:funny-bound-error-sep-decomposition} for $\bar{\bDelta}^{(1)}$.}
The key part of the proof is to show the following moment bounds on $\opnormbig{\bar{\bDelta}^{(1)}}$: 
there exists some $\eps_0 > 0$ such that for for any $k \ge 2$, there exists a
constant $C_k > 0$ independent of $m, n$ such that, for all $m, n$ and $\eps \le \eps_0$, 
\begin{equation}
\label{eqn:funny-error-Delta-one}
\E \opnormbig{\bar{\bDelta}^{(1)}}^k \le C_k(m\vee n)^{(1/2 -\eta)k} \log(m \vee n)^k\eps^{-2k}.
\end{equation}
Deferring the proof of Eq~\eqref{eqn:funny-error-Delta-one} at the sequel of the paragraph, 
we first show how Eq~\eqref{eqn:funny-error-Delta-one} immediately implies that Eq 
\eqref{eqn:funny-bound-error-sep-decomposition-as} and 
Eq~\eqref{eqn:funny-bound-error-sep-decomposition} hold for $j=1$. First, 
pick some $k \ge 2$ such that $\eta k \ge 2$. Markov's inequality and 
Eq~\eqref{eqn:funny-error-Delta-one} imply for some constant 
$C_k > 0$, the inequality below holds for any $\delta > 0$, 
\begin{equation}
\prob\left(\frac{1}{(mn)^{1/4}} \opnormbig{\bar{\bDelta}^{(1)}} \ge \delta\right) 
	\le \delta^{-k} \E \left[ \frac{1}{(mn)^{k/4}} \opnormbig{\bar{\bDelta}^{(1)}}^k \right]
	\le C_k(m\vee n)^{-\eta k} \log(m \vee n)^k\eps^{-2k}\delta^{-k}.
\end{equation}
Since $\eta k > 2$, $m/n\to \gamma \in (0,\infty)$, we know that for any $\delta > 0$, 
\begin{equation}
\sum_{n \in \N} \prob\left(\frac{1}{(mn)^{1/4}} \opnormbig{\bar{\bDelta}^{(1)}} \ge \delta\right) < \infty.
\end{equation}
Borel Cantelli's lemma implies for any $\delta > 0$,
\begin{equation}
\lim_{n \to \infty, m/n\to \gamma}  \frac{1}{(mn)^{1/4}} \opnormbig{\bar{\bDelta}^{(1)}} \le \delta~~{\rm a.s.}.
\end{equation}
Taking $\delta \to 0$, we get that, 
\begin{equation}
\label{eqn:funny-bound-error-sep-decomposition-one-two}
\limsup_{n \to \infty, m/n\to \gamma}  \frac{1}{(mn)^{1/4}} \opnormbig{\bar{\bDelta}^{(1)}} = 0.
\end{equation}
Now, we take $\eps \to 0$ on both sides of Eq~\eqref{eqn:funny-bound-error-sep-decomposition-one-two} 
and get the desired claim of Eq~\eqref{eqn:funny-bound-error-sep-decomposition-as}
for $j= 1$.

Next, for any $\eps > 0$, Eq~\eqref{eqn:funny-error-Delta-one} implies that, 
\begin{equation}
\lim_{n \to \infty, m/n \to \gamma} \frac{1}{(mn)^{1/2}} \E \opnormbig{\bar{\bDelta}^{(1)}}^2 
	= 0.
\end{equation}
By taking $\eps \to 0$ on both sides of the equation, we get the desired claim of 
Eq~\eqref{eqn:funny-bound-error-sep-decomposition} for $j= 1$.

Now, we show the key moment bound at Eq~\eqref{eqn:funny-error-Delta-one}. 
To start with, we note that $\E \bar{\bDelta}^{(1)} = \zero$. 
Now, to upper bound $\E \opnormbig{\bar{\bDelta}^{(1)}}^k$, we use the following lemma, whose 
proof is given in Section~\ref{sec:proof-opnorm-expectation-nonasymmetric}.

\begin{lemma}
\label{lemma:opnorm-expectation-nonsymmetric}
For any $k \ge 2$, there exists some constant $C_k > 0$ depending solely on $k$ such 
that for any matrix $\bX \in \R^{m \times n}$ such that $\E \bX = 0$, and either (1) 
$\bX$ has independent entries or (2) $\bX$ is symmetric (in this case $m= n$) and 
has independent upper triangular entries, 
\begin{equation}
\label{eqn:opnorm-expectation-upper-bound}
\E \opnorm{\bX}^k \le C_k \log (m+n)^k 
	\Bigg[\max_{i \in [m]} \Big(\E \sum_{j=1}^n X_{i, j}^2\Big)^{1/2} + 
	\max_{j \in [n]} \Big(\E \sum_{i=1}^m X_{i, j}^2\Big)^{1/2} 
		+ \E \max_{i, j} |X_{i, j}| \Bigg]^k.
\end{equation}
\end{lemma} \noindent\noindent
A direct application of Lemma~\ref{lemma:opnorm-expectation-nonsymmetric} shows for 
any $k \ge 2$, there exists a constant $C_k > 0$, such that
\begin{align}
 &\E \opnorm{\bar{\bDelta}^{(1)}}^k 
 \le C_k \log (m+n)^k
 	\Bigg[\max_{i \in [m]} \Big(\E \sum_{j=1}^n \big(\Delta^{(1)}_{i, j}\big)^2\Big)^{1/2} + 
\max_{j \in [n]} \Big(\E \sum_{i=1}^m \big(\Delta^{(1)}_{i, j}\big)^2\Big)^{1/2} + \E 
	\normmax{\bar{\bDelta}^{(1)}} \Bigg]^k.
\label{eqn:opnorm-expectation-Delta-one}
\end{align}
To upper bound the RHS of Eq~\eqref{eqn:opnorm-expectation-Delta-one}, we first show that, 
for some constant $C, \eps_0$ independent of $m, n$, we have, for all $\eps \le \eps_0$, 
\begin{equation}
\label{eqn:infty-target-prime}
\norm{\targetpert^\prime(\,\cdot\,)}_\infty \le C\eps^{-2}.
\end{equation}
In fact, for all $t \in \R$, we have, 
\begin{equation*}
\left|\targetpert^\prime(t)\right| 
	=  \left|\frac{p_W^{\prime\prime}(t)}{p_W(t) + \dezero}
		- \frac{(p_W^\prime(t))^2}{(p_W(t) + \dezero)^2}\right|
	\leq \dezero^{-2} \left(\dezero ||p_W^{\prime\prime}(\cdot)||_\infty
		+ \norm{p_W^\prime(\cdot)}_\infty^2\right) \le 2\dezero^{-2}(M_2 \vee \dezero)^2,
\end{equation*}
where the last inequality follows from Assumption~{\sf A2}. 
This gives the desired bound at Eq~\eqref{eqn:infty-target-prime}. 

Now, by Eq~\eqref{eqn:infty-target-prime}, we know that, for $\eps$ 
small enough, the random variable $f_{W, \eps}^\prime(W_{1, 1})$ satisfies, 
\begin{equation}
\Var^{1/2}\Big(f_{W, \eps}^\prime(W_{1, 1}) \Big) \le \norm{\targetpert^\prime(\,\cdot\,)}_\infty
	\le C\eps^{-2}.
\end{equation}
Hence we have the bound that,  
\begin{equation}
\label{eqn:bound-expection-ell-two-infty}
\max_{i \in [m]} \bigg(\E \sum_{j=1}^n \Big(\Delta^{(1)}_{i, j}\Big)^2\bigg)^{1/2}
	= \norm{\bX}_{\ell_2 \to \ell_\infty} \times \Var^{1/2}\left(f_{W, \eps}^\prime(W_{1, 1}) \right))
	\le C \eps^{-2}\norm{\bX}_{\ell_2 \to \ell_\infty}.
\end{equation}
Similarly, 
\begin{equation}
\label{eqn:bound-expection-ell-infty-two}
\max_{j \in [n]}\E \bigg(\sum_{i=1}^m \Big(\Delta^{(1)}_{i, j}\Big)^2\bigg)^{1/2}
	= \norm{\bX{^\sT}}_{\ell_2 \to \ell_\infty} \times \Var^{1/2}\left((f_{W, \eps}^\prime(W_{1, 1}) \right)
	\le C \eps^{-2}\norm{\bX{^\sT}}_{\ell_2 \to \ell_\infty}.
\end{equation}
Lastly, since by Eq~\eqref{eqn:infty-target-prime}, we have 
$\normbig{\bar{\bDelta}^{(1)}}_{\rm max}  \le C\eps^{-2} \normmax{\bX}$ almost surely, 
this gives that
\begin{equation}
\label{eqn:bound-expectation-max}
\E \normbig{\bar{\bDelta}^{(1)}}_{\rm max} \le C\eps^{-2} \normmax{\bX}, 
\end{equation}
Now, substituting Eq~\eqref{eqn:bound-expection-ell-two-infty}, Eq~\eqref{eqn:bound-expection-ell-infty-two}
and Eq~\eqref{eqn:bound-expectation-max} into Eq~\eqref{eqn:opnorm-expectation-Delta-one}, 
we get for any $k > 0$, there exists some constant $C_k > 0$ such that  
\begin{align*}
\E \opnormbig{\bar{\bDelta}^{(1)}}^k 
	&\le C_k \log(m+n)^k \eps^{-2k} \max\{\norm{\bX}_{\ell_2 \to \ell_\infty}, 
		\norm{\bX}_{\ell_\infty \to \ell_2}\}^k \\
	&\le C_k (m\vee n)^{(1/2-\eta) k} \log(m \vee n)^k \eps^{-2k},
\end{align*}
giving the desired claim at Eq~\eqref{eqn:funny-error-Delta-one}.

\paragraph{Proof of Eq~\eqref{eqn:funny-bound-error-sep-decomposition-as} and
Eq~\eqref{eqn:funny-bound-error-sep-decomposition} for $\bar{\bDelta}^{(2)}$.}
In the proof, we show the following strengthened result: for some constant $C$ independent 
of $m, n$, we have for all $m, n \in \N$ and $\eps > 0$, 
\begin{equation}
\label{eqn:funny-error-Delta-two}
\opnorm{\bar{\bDelta}^{(2)}} \le C\eps^{-3}(m \vee n)^{1/2-\eta}.
\end{equation}
Indeed, Eq~\eqref{eqn:funny-error-Delta-two} clearly implies for any $\eps > 0$, 
\begin{equation}
\label{eqn:funny-error-Delta-two-one}
\lim_{n \to \infty, m/n \to \gamma} \frac{1}{(mn)^{1/4}} 
	\opnormbig{\bar{\bDelta}^{(2)}} = 0.
\end{equation}
and 
\begin{equation}
\label{eqn:funny-error-Delta-two-two}
\lim_{n \to \infty, m/n \to \gamma} 
	\frac{1}{(mn)^{1/2}} \E \opnormbig{\bar{\bDelta}^{(2)}}^2 = 0.
\end{equation}
Now taking limit on both sides of Eq~\eqref{eqn:funny-error-Delta-two-one}
and Eq~\eqref{eqn:funny-error-Delta-two-two} shows that the desired claim 
at Eq 
\eqref{eqn:funny-bound-error-sep-decomposition-as} and Eq
\eqref{eqn:funny-bound-error-sep-decomposition} hold for $j = 2$.

In the rest of the proof, we show Eq~\eqref{eqn:funny-error-Delta-two}.
Indeed, we first show that, for some constant $C, \eps_0$ independent of 
$m, n, \eps$, we have, for $\eps \le \eps_0$, 
\begin{equation}
\label{eqn:infty-target-prime-prime}
\norm{\targetpert^{\prime\prime}\left(\,\cdot\,\right)}_\infty \le C\eps^{-3}.
\end{equation}
In fact, for all $t \in \R$, 
\begin{align}
\left|\targetpert^{\prime\prime}(t)\right| 
&= \left|\frac{p_W^{\prime\prime\prime}(t)}
	{p_W(t)+\dezero} - 
	\frac{3 p_W^\prime(t)p_W^{\prime\prime}(t)}{(p_W(t) + \dezero)^2}
	+ \frac{2 (p_W^\prime(t))^3}{(p_W(t) + \dezero)^3}\right|  \nonumber \\
&\leq \dezero^{-3} \left(||p_W^{\prime\prime\prime}(\cdot)||_\infty \dezero^{2} 
	+ 3 ||p_W^\prime(\cdot)||_\infty |p_W^{\prime\prime}(\cdot)||_\infty \dezero + 
		2 ||p_W^\prime(\cdot)||_\infty^3\right) \nonumber\\ 
&\le  6\dezero^{-3}(M_2 \vee \eps)^3,
\end{align}
where the last inequality follows from Assumption~{\sf A2}. 
This proves the desired bound at Eq~\eqref{eqn:infty-target-prime-prime}. 

Now, fix the constant $C > 0$ such that Eq~\eqref{eqn:infty-target-prime-prime} 
holds. Introduce the auxiliary matrix $\wtilde{\bDelta}^{(2)}$ by 
\begin{equation}
\wtilde{\bDelta}^{(2)}_{i, j} = \half C \eps^{-3} X_{i, j}^2 . 
\end{equation}
By construction of $\bDelta$ and $\wtilde{\bDelta}^{(2)}$, we know that, 
there exists some $\eps_0$ such that for any $\eps \le \eps_0$, 
\begin{equation}
\label{eqn:dominance-Delta-wDelta}
\left|\bDelta^{(2)}_{i, j}\right| \le 
	\half X_{i, j}^2 \norm{\targetpert^{\prime\prime}\left(\,\cdot\,\right)}_\infty
	\le  \wtilde{\bDelta}^{(2)}_{i, j}.
\end{equation}
This immediately implies for any $\eps \le \eps_0$, and any $l \in \N$, 
\begin{equation}
\tr \left\{\left(\bDelta^{(2)} (\bDelta^{(2)}){^\sT}\right)^l\right\}
\le 
\tr\left\{\left(\wtilde{\bDelta}^{(2)} (\wtilde{\bDelta}^{(2)}){^\sT}\right)^l\right\}.
\end{equation}
Since $\opnorm{\bA} = \lim_{l \to \infty} \tr\{(A^{\sT} A)^l\}^{\frac{1}{2l}}$ holds for any 
matrix $\bA$, we get that, 
\begin{equation}
\opnormbig{\bDelta^{(2)}} = 
\lim_{l \to \infty}\left(\tr \left\{\left(\bDelta^{(2)} (\bDelta^{(2)}){^\sT}\right)^l\right\}\right)^{\frac{1}{2l}} 
	\stackrel{(i)}{\le} 
\lim_{l \to \infty}\left(\tr\left\{\left(\wtilde{\bDelta}^{(2)} (\wtilde{\bDelta}^{(2)}){^\sT}\right)^l\right\}\right)^{\frac{1}{2l}} 
	= \opnormbig{\wtilde{\bDelta}^{(2)}}
\end{equation}
The above estimate immediately implies that 
\begin{equation}
\opnormbig{\bar{\bDelta}^{(2)}} \le  \opnormbig{\wtilde{\bDelta}^{(2)}}
	= C\eps^{-3} \opnormbig{\bX \odot \bX}
	\le  CM\eps^{-3}(m \vee n)^{1/2-\eta},
\end{equation}
where the last inequality uses the assumption $X \in \cF_{m, n}(r, M, \eta)$. 
This proves Eq~\eqref{eqn:funny-error-Delta-two} as desired.

\paragraph{Proof of Eq~\eqref{eqn:funny-bound-error-sep-decomposition-as} 
and Eq~\eqref{eqn:funny-bound-error-sep-decomposition} for $\bar{\bDelta}^{(3)}$.}
Now that $\bZ$ is a random matrix whose entries are i.i.d mean $0$ and variance $1$. 
Moreover, by Eq~\eqref{eqn:normmax-Z}, we know for any fix $\eps > 0$, there exists 
some constant $C_{\eps}$ (independent of $m, n$) such that $\normmax{\bZ} \le C_{\eps}$.
Thus, random matrix theory show that, for any fix $\eps > 0$, almost surely we have, 
\begin{equation}
\label{eqn:funny-bound-error-Delta-three-one-as}
\lim_{n\to \infty, m/n\to \gamma} 
	\frac{1}{(mn)^{1/4}} \opnorm{Z} = \gamma^{1/4} + \gamma^{-1/4}. 
\end{equation}
Moreover, for any fix $\eps > 0$, we have
\begin{equation}
\label{eqn:funny-bound-error-Delta-three-one-E}
\lim_{n\to \infty, m/n\to \gamma} \frac{1}{(mn)^{1/2}} \E\opnorm{Z}^2 = (\gamma^{1/4} + \gamma^{-1/4})^2. 
\end{equation}
Since by Eq~\eqref{eqn:limit-dominated-eps}, we have, 
\begin{equation}
\label{eqn:funny-bound-error-Delta-three-two}
\lim_{\eps \to 0} \bar{\nu}_{\eps} = \Info_W^{1/2},
\end{equation}
Eq~\eqref{eqn:funny-bound-error-Delta-three-one-as} and 
Eq~\eqref{eqn:funny-bound-error-Delta-three-one-E} immediately imply that, 
\begin{equation}
\lim_{\eps \to 0}\lim_{n\to \infty, m/n\to \gamma} 
	\frac{1}{(mn)^{1/4}} \opnormbig{\bar{\bDelta}^{(3)}}
= (\gamma^{1/4} + \gamma^{-1/4})
	\cdot \lim_{\eps \to 0} (\bar{\nu}_{\eps}- \Info_W^{1/2}) = 0,
\end{equation}
and moreover, 
\begin{equation}
\lim_{\eps \to 0}\lim_{n\to \infty, m/n\to \gamma} 
	\frac{1}{(mn)^{1/2}} \E\opnormbig{\bar{\bDelta}^{(3)}}^2
= (\gamma^{1/4} + \gamma^{-1/4})^2 \cdot 
	\lim_{\eps \to 0} (\bar{\nu}_{\eps}- \Info_W^{1/2})^2 = 0, 
\end{equation}
This proves the desired claim of Eq~\eqref{eqn:funny-bound-error-sep-decomposition-as}
and Eq~\eqref{eqn:funny-bound-error-sep-decomposition} for $\bar{\bDelta}^{(3)}$.

\paragraph{Proof of Eq~\eqref{eqn:funny-bound-error-sep-decomposition-as} 
and Eq~\eqref{eqn:funny-bound-error-sep-decomposition} for $\bar{\bDelta}^{(4)}$.}
By definition, $
\opnormbig{\bDelta^{(4)}} = \left|\mu_{\eps}^\prime - \Info_W\right| \opnorm{\bX}.
$
Since $\bX \in \cF_{m, n}(r, M, \eta)$, we know that, 
\begin{equation}
\label{eqn:funny-bound-error-Delta-four-one}
\lim_{n\to \infty, m/n\to \gamma} \frac{1}{(mn)^{1/4}} \opnormbig{\bDelta^{(4)}} 
	\le \left|\mu_{\eps}^\prime - \Info_W\right| M. 
\end{equation}
and 
\begin{equation}
\label{eqn:funny-bound-error-Delta-four-two}
\lim_{n\to \infty, m/n\to \gamma} \frac{1}{(mn)^{1/2}} \E \opnormbig{\bDelta^{(4)}}^2 
	\le (\mu_{\eps}^\prime - \Info_W)^2 M^2. 
\end{equation}
Since we have the expression below for $\wbar{\mu}_{\eps}^\prime$, 
\begin{equation}
\wbar{\mu}_{\eps}^\prime = \int_\R f_{W, \eps}^\prime(w) p_W(w) \rmd w =  
- \int_\R f_{W, \eps}(w)p_W^\prime(w) \rmd w =
 \int_\R \frac{(p_W^\prime(w))^2}{p_W(w)+\eps} \rmd w ,
\end{equation}
we know by dominated convergence theorem
\begin{equation}
\lim_{\eps\to 0} \big|\wbar{\mu}_{\eps}^\prime - \Info_W\big| = 0.
\end{equation}
Hence, by taking $\eps \to 0$ on both sides of 
Eq~\eqref{eqn:funny-bound-error-Delta-four-one} and 
Eq~\eqref{eqn:funny-bound-error-Delta-four-two},  
we prove the desired claim of Eq~\eqref{eqn:funny-bound-error-sep-decomposition-as}
and Eq~\eqref{eqn:funny-bound-error-sep-decomposition} for $\bar{\bDelta}^{(4)}$.

\section{Proof of Lemma~\ref{lemma:upper-bound-lemma-two}}
\label{sec:proof-upper-bound-lemma-two}
Denote the quantities $\wbar{W}$, $\Info_{W, \eps}$ and $\difI$ to be
\begin{align}
\label{eqn:def-Info-W-eps}
\wbar{W} =  \frac{1}{mn} \sum_{i \in [m], j \in [n]} W_{i, j},~~
\Info_{W, \eps} = \int_{\R} \frac{(p_W^\prime(w))^2}{p_W(w) + \eps} \de w
~~\text{and}~~\difI = |\Info_{W, \eps} - \Info_W|.
\end{align}
Denote $\res_{n, 1}$ and $\res_{n, 2}$ to be
\begin{align}
\res_{n, 1} = h_n^2 + (mn h_n)^{-1/2}\log (mn)
	~\text{and}~
\res_{n, 2} =  (h_n^\prime)^2 + (mn (h_n^\prime)^3)^{-1/2}\log (mn).
	\label{eqn:def-res-n}
\end{align}
Denote $\res_n(\eps)$ and $\resQ_n(\eps; \kappa)$ to be 
\begin{align}
\res_n(\eps) =\eps^{-1} \res_{n, 1} + \eps^{-2}\res_{n, 2}
~~\text{and}~~
\resQ_n(\eps; \kappa) = \res_n(\eps) + \eps^{-2}(mn)^{-\kappa/2} + \difI + \eps. 
\end{align}
The crux of the proof is to provide high probability bounds onto the quantities
\begin{equation}
\opnorm{\what{f}_{Y, \eps}(\bY) - \targetpert(\bY)}
~~\text{and}~~
\big|\hat{\Info}_{W, \eps} - \Info_W\big|. 
\end{equation}
The bounds are made more precise in the technical lemma below. As the proof 
is lengthy, we defer its proof into Section~\ref{sec:proof-upper-bound-two-high-prob-key}. 
\begin{lemma}
\label{lemma:upper-bound-two-high-prob-key}
Assume that $h_n, h_n^\prime$ are chosen in a way such that, 
\begin{equation}
\label{eqn:choice-of-h-n-h-n-prime-limit}
\lim_n \res_{n, 1} = \lim_n \res_{n, 2} = 0. 
\end{equation}
Fix $\kappa \in (0, 1/2)$. There exist some constants $C, c > 0$ independent of 
$m, n, \eps$ (but can be dependent of $\kappa$ and underlying distribution $\prob$), 
such that whenever $\eps \le c$, there exists $n_0 = n_0(\eps)$ such that, for any $n \ge n_0$, we have,
with probability at least $1-(mn)^{-(1-\kappa)}$, 
\begin{align}
&\opnorm{\what{f}_{Y, \eps}(\bY) - \targetpert(\bY)} 
	\le C (mn)^{1/2}(\res_n(\eps) + \eps^{-2}\big|\wbar{W}\big|).
\label{eqn:upper-bound-proposition-one}
\end{align}
and 
\begin{equation}
\label{eqn:upper-bound-Info-W-high-prob}
\big|\hat{\Info}_{W, \eps} - \Info_W\big| \le C\resQ_n(\eps; \kappa)
\end{equation}
\end{lemma}\noindent\noindent
Now, we prove Eq~\eqref{eqn:upper-bound-lemma-two-one}, 
Eq~\eqref{eqn:upper-bound-lemma-two-two} and 
Eq~\eqref{eqn:upper-bound-lemma-two-three} 
in the three paragraphs below. This gives the desired result of 
Lemma~\ref{lemma:upper-bound-lemma-two}. 

\paragraph{Proof of Eq~\eqref{eqn:upper-bound-lemma-two-one}}
Let $\kappa = 1/4$. According to Lemma~\ref{lemma:upper-bound-two-high-prob-key}, 
we know that, 
there exist some constants $c, C > 0$, such that for any $\eps \le c$, the event 
$\Lambda_n(\eps)$ defined by
\begin{equation}
\label{eqn:def-Lambda-n}
\Lambda_n(\eps) = \left\{\opnorm{\what{f}_{Y, \eps}(\bY) - \targetpert(\bY)} 
	\ge C (mn)^{1/2}(\res_n(\eps) + \eps^{-2}\big|\wbar{W}\big|)\right\}
\end{equation}
happens with probability at most $(mn)^{-3/4}$ for large enough $n$. As
$m/n \to \gamma$, this shows that, 
\begin{equation}
\sum_{n=1}^{\infty}\prob \left(\Lambda_n(\eps) \right) < \infty. 
\end{equation} 
Applying Borel Canteli's lemma, the above shows that, for any $\eps \le c$, 
the events $\{\Lambda_n(\eps)\}_{n=1}^\infty$ happen at most finite 
times almost surely. Thus, for $\eps \le c$, we know that almost surely
\begin{align}
\limsup_{n\to \infty, m/n\to \infty}\left(\frac{1}{(mn)^{1/4}}\opnorm{\what{f}_{Y, \eps}(\bY) - \targetpert(\bY)} 
- C (mn)^{1/4}\left(\res_n(\eps) + \eps^{-2}\big|\wbar{W}\big|\right)\right) \le 0.
\label{eqn:as-limit-key-part-one}
\end{align}
Now that, since $\{W_{i, j}\}_{i\in [m], j\in [n]}$ are i.i.d mean $0$ with finite second 
moments, we know by the law of the iterated logarithm that, 
\begin{equation}
\label{eqn:as-limit-bar-W}
\limsup_{n \to \infty, m/n\to \gamma} (mn)^{1/4}\big|\wbar{W}\big| \asto 0. 
\end{equation}
Moreover, our choice of $h_n = n^{-\eta_1}$ and $h_n^\prime = n^{-\eta_2}$ for 
$\eta_1 \in (1/4, 1)$ and $\eta_2 \in (1/4, 1/3)$ gives
\begin{equation}
\limsup_{n \to \infty, m/n\to \gamma} (mn)^{1/4} \max\{ \res_{n, 1}, \res_{n, 2}\} \to 0
\end{equation}
which by definition of $\res_n(\eps)$, implies that 
\begin{equation}
\label{eqn:limit-res-n-eps}
\limsup_{n \to \infty, m/n\to \gamma} (mn)^{1/4}\res_n(\eps) = 0
\end{equation}
Hence, Eq~\eqref{eqn:as-limit-key-part-one}, Eq~\eqref{eqn:as-limit-bar-W} and 
Eq~\eqref{eqn:limit-res-n-eps} together imply that, for $\eps \le c$, 
almost surely,
\begin{equation}
\limsup_{n\to \infty, m/n\to \infty}
		\frac{1}{(mn)^{1/4}}\opnorm{\what{f}_{Y, \eps}(\bY) - \targetpert(\bY)}  = 0.
\end{equation}
By taking $\eps \to 0$, we derive the desired claim at 
Eq~\eqref{eqn:upper-bound-lemma-two-one}.

\paragraph{Proof of Eq~\eqref{eqn:upper-bound-lemma-two-two}}
In the proof, we set $\kappa = \kappa_0$ where 
$\kappa_0 \in (0, 1/24)$. We show that Eq~\eqref{eqn:upper-bound-lemma-two-two}
holds for $\nu_0 = \nu_0$ for any $\nu_0 \in (0, 1/24)$. 
To start with, by Lemma~\ref{lemma:upper-bound-two-high-prob-key}, we know that, 
for some constants $c, C > 0$ independent of $m, n, \eps$, we have for any 
$\eps \le c$, the event $\Lambda_n(\eps; \kappa_0)$ defined by
\begin{equation}
\label{eqn:def-Lambda-n-kappa-0}
\Lambda_n(\eps; \kappa_0) = \left\{\opnorm{\what{f}_{Y, \eps}(\bY) - \targetpert(\bY)} 
	\ge C (mn)^{1/2}(\res_n(\eps) + \eps^{-2}\big|\wbar{W}\big|) \right\}
\end{equation}
happens with probability at least $(mn)^{-(1-\kappa_0)}$. To prove the desired 
Eq~\eqref{eqn:upper-bound-lemma-two-two}, it suffices to show  
\begin{equation}
\label{eqn:upper-bound-two-two-one}
\lim_{\eps \to 0} \limsup_{n \to \infty, m/n\to \gamma} 
	\frac{1}{(mn)^{(1+\nu_0)/4}}\E \left[\opnorm{\what{f}_{Y, \eps}(\bY) - \targetpert(\bY)}^{(1+\nu_0)}
		\indic{\Lambda_n(\eps; \kappa_0)^c}\right] = 0.
\end{equation}
and 
\begin{equation}
\label{eqn:upper-bound-two-two-two}
\lim_{\eps \to 0} \limsup_{n \to \infty, m/n\to \gamma} 
	\frac{1}{(mn)^{(1+\nu_0)/4}} \E \left[\opnorm{\what{f}_{Y, \eps}(\bY) - \targetpert(\bY)}^{(1+\nu_0)}
		\indic{\Lambda_n(\eps; \kappa_0)} \right] = 0.
\end{equation}
First, we show Eq~\eqref{eqn:upper-bound-two-two-one}. Note that, 
Eq~\eqref{eqn:limit-res-n-eps} shows that under our careful choice of 
$h_n = n^{-\eta_1}$ and $h_n^\prime = n^{-\eta_2}$ for 
$\eta_1 \in (1/4, 1)$ and $\eta_2 \in (1/4, 1/3)$, we have 
\begin{equation}
\label{eqn:R-n-eps-exp}
\lim_{\eps \to 0} \lim_{n \to \infty, m/n\to \gamma} 
	\left((mn)^{1/4}\res_n(\eps)\right)^2 = 0. 
\end{equation}
Next, by Assumption~{\sf A2}, $\{W_{i, j}\}_{i\in [m], j\in [n]}$ are i.i.d mean 
$0$ with bounded second moments. Thus, 
\begin{equation}
 \lim_{n \to \infty, m/n\to \gamma} 
	\E ((mn)^{1/4}\big|\wbar{W}\big|)^2 = 0,
\end{equation}
which by taking $\eps \to 0$, immediately implies that, 
\begin{equation}
\label{eqn:W-bar-exp}
\lim_{\eps \to 0} \lim_{n \to \infty, m/n\to \gamma} 
	\E \left[\Big((mn)^{1/4}\eps^{-2}\big|\wbar{W}\big|\Big)^2\right]
		= 0. 
\end{equation}
Now, Eq~\eqref{eqn:R-n-eps-exp}, Eq~\eqref{eqn:W-bar-exp} and 
H\"{o}lder's inequality immediately give that
\begin{equation}
\label{eqn:upper-bound-two-two-one-goal}
\lim_{\eps \to 0} \limsup_{n \to \infty, m/n\to \gamma} 
	\E \left[\left((mn)^{1/4}\left(\res_n(\eps) + \eps^{-2}\big|\wbar{W}\big|\right)
	\right)^2\right] = 0.
\end{equation}
Finally by definition of $\Lambda_n(\eps)$, the above equation
implies that, 
\begin{equation}
\label{eqn:upper-bound-two-two-one-Holder}
\lim_{\eps \to 0} \limsup_{n \to \infty, m/n\to \gamma} 
	\frac{1}{(mn)^{1/2}}\E \left[\opnorm{\what{f}_{Y, \eps}(\bY) - \targetpert(\bY)}^{2}
		\indic{\Lambda_n(\eps; \kappa_0)^c}\right] = 0,
\end{equation}
which by H\"{o}lder's inequality again, shows that Eq~\eqref{eqn:upper-bound-two-two-one}
holds for any $\nu_0 \in (0, 1/24)$.

Next, we prove Eq~\eqref{eqn:upper-bound-two-two-two}. The starting point of 
the proof is the following bound on the operator norm (see 
Lemma~\ref{lemma:operator-to-one-infty} for details), 
\begin{equation}
\label{eqn:upper-bound-two-two-two-start}
\frac{1}{(mn)^{1/2}}\opnorm{\what{f}_{Y, \eps}(\bY) - f_{W, \eps}(\bY)} \le 
\normmax{\what{f}_{Y, \eps}(\bY) - f_{W, \eps}(\bY)} \le
	\normmax{\what{f}_{Y, \eps}(\bY)} + \normBig{f_{W, \eps}(\bY)}_{\rm max}. 
\end{equation}
To upper bound the RHS of Eq~\eqref{eqn:upper-bound-two-two-two-start}, we first 
note that, almost surely, 
\begin{equation}
\label{eqn:upper-bound-two-two-two-start-one}
\normmax{\what{f}_{Y, \eps}(\bY)} \le \norm{\what{f}_{Y, \eps}(\cdot)}_\infty 
	\le \eps^{-1} \norm{p_W^\prime(\cdot)}_\infty \le 
	\eps^{-1} (h_n^\prime)^{-2} \norm{K^\prime(\cdot)}_\infty 
		\le M\eps^{-1} (h_n^\prime)^{-2}.
\end{equation}
Moreover, by Assumption~{\sf A2}, we have, 
\begin{equation}
\label{eqn:upper-bound-two-two-two-start-two}
\normmax{f_{W, \eps}(\bY)} \le \norm{f_{W, \eps}(\cdot)}_\infty
	\le \eps^{-1} \norm{p_W^\prime(\cdot)}_\infty
	\le M_2 \eps^{-1}.
\end{equation}
Substituting Eq~\eqref{eqn:upper-bound-two-two-two-start-one} and 
Eq~\eqref{eqn:upper-bound-two-two-two-start-two} into Eq
\eqref{eqn:upper-bound-two-two-two-start}, we get for some constant 
$C$ independent of $m, n, \eps$, the estimate below holds almost surely, 
\begin{equation}
\frac{1}{(mn)^{1/4}}\opnorm{\what{f}_{Y, \eps}(\bY) - f_{W, \eps}(\bY)} \le 
	C (mn)^{1/4}\eps^{-1}\left((h_n^\prime)^{-2} +  1\right). 
\end{equation}
Hence, this implies 
\begin{align}
&\frac{1}{(mn)^{(1+\nu_0)/4}} \E \left[\opnorm{\what{f}_{Y, \eps}(\bY) - \targetpert(\bY)}^{1+\nu_0} 
	\indic{\Lambda_n(\eps; \kappa_0)} \right]  \\
&\le C^2 (mn)^{(1+\nu_0)/4}\eps^{-(1+\nu_0)}\left((h_n^\prime)^{-2} +1\right)^{(1+\nu_0)}
	\prob\left(\Lambda_n(\eps; \kappa_0)\right)\nonumber \\
&\le C^2\eps^{-(1+\nu_0)} 
	(mn)^{-(3/4-\kappa_0-\nu_0/4)}\left((h_n^\prime)^{-2} +1\right)^{(1+\nu_0)}.
\label{eqn:upper-bound-two-two-two-start-three}
\end{align}
Since $h_n^\prime = n^{-\eta_2}$ for $\eta_2 \in (1/4, 1/3)$, we know that 
for our choice of $\nu_0, \kappa_0 \in (0, 1/24)$, 
\begin{equation}
\lim_{n\to \infty, m/n\to \gamma}
	(mn)^{-(3/4-\kappa_0-\nu_0/4)}
		\left((h_n^\prime)^{-2} +1\right)^{(1+\nu_0)}
		= 0
\end{equation}
holds for any $\eps > 0$. Thus Eq~\eqref{eqn:upper-bound-two-two-two-start-three} shows 
that for the $\kappa_0, \nu_0 > 0$, 
\begin{equation}
\lim_{n\to \infty, m/n\to \gamma}\frac{1}{(mn)^{(1+\nu_0)/4}} 
	\E \left[\opnorm{\what{f}_{Y, \eps}(\bY) - \targetpert(\bY)}^{1+\nu_0} 
		\indic{\Lambda_n(\eps; \kappa_0)} \right] = 0
\end{equation}
holds for any $\eps > 0$. Taking $\eps \to 0$ gives the desired Eq~\eqref{eqn:upper-bound-two-two-two}.

\paragraph{Proof of Eq~\eqref{eqn:upper-bound-lemma-two-three}}
In the proof, we fix $\kappa = 1/4$. To simplify our notation, denote, 
\begin{equation}
\bar{\resQ}_n(\eps) = \resQ_n(\eps; 1/4) =
	 \res_n(\eps) + \eps^{-2}(mn)^{-1/8} + \difI + \eps. 
\end{equation}
According to Lemma~\ref{lemma:upper-bound-two-high-prob-key}, we know that, 
there exist some constants $c, C > 0$, such that for any $\eps \le c$, the event 
$\Lambda_n^{\Info}(\eps)$ defined by
\begin{equation}
\Lambda_n^{\Info}(\eps) = \left\{\big|\hat{\Info}_{W, \eps} - \Info_W\big| \ge C\bar{\resQ}_n(\eps) \right\}
\end{equation}
happens with probability at most $(mn)^{-3/4}$ for large enough $n$. As
$m/n \to \gamma$, this shows that, 
\begin{equation}
\sum_{n=1}^{\infty}\prob \left(\Lambda^{\Info}_n(\eps) \right) < \infty. 
\end{equation} 
Applying Borel Canteli's lemma, the above shows that, for any $\eps \le c$, 
the events $\{\Lambda^{\Info}_n(\eps)\}_{n=1}^\infty$ happen at most finite 
times almost surely. Thus, for $\eps \le c$, almost surely
\begin{equation}
\limsup_{n\to \infty, m/n\to \gamma}
	\big|\hat{\Info}_{W, \eps} - \Info_W\big| \le 
C\limsup_{n\to \infty, m/n\to \gamma}\bar{\resQ}_n(\eps) 
\le C(\delta_{W, \eps} + \eps).
\end{equation}
By taking $\eps \to 0$, we know that 
\begin{equation}
\lim_{\eps \to 0} \limsup_{n\to \infty, m/n\to \gamma}
	\big|\hat{\Info}_{W, \eps} - \Info_W\big|
\le C\lim_{\eps \to 0}(\delta_{W, \eps} + \eps) = 0,
\end{equation}
which gives the desired claim of Eq~\eqref{eqn:upper-bound-lemma-two-three}.

\subsection{Proof of Lemma~\ref{lemma:upper-bound-two-high-prob-key}}
\label{sec:proof-upper-bound-two-high-prob-key}
\subsubsection{Notation}
Throughout the section, we make use of the following notation. We use 
constants $C, c$ to denote constants that are independent of $m, n, \eps$, 
but can be dependent on the underlying distribution $\prob$, the model parameters 
$M$, $M_1$, $M_2$, $r$, $\eta$ and $\kappa$. We 
also use $C, c$ to denote numerical constants in some situations, where 
we shall point this fact explicitly. It is understood that all those constants $C, c$ 
might not be the same at each occurrence. For any matrix $\bA \in \R^{m \times n}$, 
we denote 
\begin{equation}
\label{eqn:notation-matrix}
\wbar{A} = \frac{1}{mn} \sum_{i \in [m], j \in [n]} A_{i, j}, 
	~~\normmax{\bA} = \max_{i \in [m], j\in [n]} |A_{i, j}|~~\text{and}~~
	\wtilde{\bA} = \bA - \wbar{A} \one_{m} \one_n{^\sT}.
\end{equation}

\subsubsection{Proof}
The crux of the proof is the key Lemma~\ref{lemma:high-prob-entrywise-bound}
and its companion Lemma~\ref{lemma:high-prob-entrywise-bound-info}, which we 
present below. The proof of the two lemma are deferred to 
Section~\ref{sec:proof-lemma-high-prob-entrywise-bound}.

\begin{lemma}
\label{lemma:high-prob-entrywise-bound}
For some constant $C$, we have with probability at least $1- (mn)^{-(1-\kappa)}$,
\begin{align}
&\normmax{\hat{p}_Y(\bY) - p_W(\bY)} \le 
	C  \left(\res_{n, 1} +  \big|\wbar{W}\big|\right).
	\label{eqn:high-prob-entrywise-bound-p}
\end{align}
Moreover, for the same constant $C$, we have with probability at least $1- (mn)^{-(1-\kappa)}$,
\begin{align}
&\normmax{\hat{p}_Y^\prime(\bY) - p_W^\prime(\bY)} \le 
	C \left(\res_{n, 2} +   \big|\wbar{W}\big|\right).
	\label{eqn:high-prob-entrywise-bound-p-prime}
\end{align}
\end{lemma}

\begin{lemma}
\label{lemma:high-prob-entrywise-bound-info}
For some constant $C$, we have with probability at least $1- (mn)^{-(1-\kappa)}$,
\begin{align}
&\normmax{\hat{p}_Y(\wtilde{\bY}) - p_W(\wtilde{\bY})} \le 
	C  \left(\res_{n, 1} +  \big|\wbar{W}\big|\right),
	\label{eqn:high-prob-entrywise-bound-p-info}
\end{align}
Moreover, for the same constant $C$, we have with probability at least $1- (mn)^{-(1-\kappa)}$,
\begin{align}
&\normmax{\hat{p}_Y^\prime(\wtilde{\bY}) - p_W^\prime(\wtilde{\bY})} \le 
	C \left(\res_{n, 2} +  \big|\wbar{W}\big|\right).
	\label{eqn:high-prob-entrywise-bound-p-prime-info}
\end{align}
\end{lemma} \noindent\noindent
Given Lemma~\ref{lemma:high-prob-entrywise-bound} and 
Lemma~\ref{lemma:high-prob-entrywise-bound-info}, we 
show the desired Eq~\eqref{eqn:upper-bound-proposition-one} and Eq 
\eqref{eqn:upper-bound-Info-W-high-prob}.

\paragraph{Proof of Eq~\eqref{eqn:upper-bound-proposition-one}}
Our proof starts from the following upper bound on the operator norm 
(see Lemma~\ref{lemma:operator-to-one-infty} for a proof of this bound)
\begin{equation} 
\label{eqn:elementary-opnorm-bound}
\opnorm{\widehat{f}_{Y,\eps}(\bY) - f_{W,\eps}(\bY)} \leq 
	(mn)^{1/2} \normmax{\widehat{f}_{Y,\eps}(\bY) - f_{W,\eps}(\bY)}.
\end{equation}
According to Eq~\eqref{eqn:elementary-opnorm-bound}, it suffices to show that, 
with probability at least $1-(mn)^{-(1-\kappa)}$: 
\begin{align}
 \normmax{\widehat{f}_{Y,\eps}(\bY) - f_{W,\eps}(\bY)}
	\le C(\res_n(\eps) + \dezero^{-2} \big|\wbar{W}\big|).
\label{eqn:goal-uniform-entry-f-bound}
\end{align}
For notational simplicity, denote the function $g: \R \times \R_+ \to \R$ to be: 
\begin{equation}
\label{eqn:def-g}
g(s, t) = -\frac{s}{t + \eps}~\text{for any $s\in \R, t \in \R_+$}. 
\end{equation} 
Using this notation, the goal is equivalent to showing that, with probability at least 
$1-(mn)^{-(1-\kappa)}$: 
\begin{equation}
\normmax{g(\hat{p}_Y^\prime(\bY), \hat{p}_Y(\bY) - g(p_W^\prime(\bY), p_W(\bY))}
 	\le C (\res_n(\eps) + \dezero^{-2} \big|\wbar{W}\big|).
\label{eqn:goal-uniform-entry-g-bound}
\end{equation}
Motivated by Eq~\eqref{eqn:goal-uniform-entry-g-bound}, we present an elementary 
lemma regarding the function $g(\cdot)$, basically showing that the function $g$ is 
(pseudo)-Lipschitz with respect to its input arguments.
\begin{lemma}
\label{lemma:elementary-inequality}
For any $s_1, s_2 \in \R$ and, $t_1, t_2 \ge 0$, 
\begin{equation*}
\left|g(s_1, t_1) - g(s_2, t_2)\right| \leq \dezero^{-1}|s_1 - s_2| 
	+ \dezero^{-2}|t_1 - t_2| (|s_1| \wedge |s_2|).
\end{equation*}
\end{lemma}\noindent\noindent
Lemma~\ref{lemma:elementary-inequality} immediately gives us a control on the LHS 
of Eq~\eqref{eqn:goal-uniform-entry-g-bound}. In fact, for $i\in [m], j \in [n]$, 
\begin{align}
\widehat{f}_{Y,\eps}(Y_{i, j}) - f_{W,\eps}(Y_{i, j}) &= 
	\left|g(\hat{p}_W^\prime(Y_{i, j}), \hat{p}_W(Y_{i, j}))
		- g(p_W^\prime(Y_{i, j}), p_W(Y_{i, j}))\right| \nonumber \\
&\le \dezero^{-1} \left|\hat{p}_Y^\prime(Y_{i, j}) - p_W^\prime(Y_{i, j})\right|
	+  \dezero^{-2}\left|\hat{p}_Y(Y_{i, j}) - p_W(Y_{i, j})\right|
		(|\hat{p}_Y^\prime(Y_{i, j})| \wedge  |p_W^\prime(Y_{i, j})|) \nonumber \\
&= \dezero^{-1}\normmax{\hat{p}^\prime_Y(\bY) - p_W^\prime(\bY)}
	+ \dezero^{-2} \norm{p_W^\prime(\cdot)}_\infty \normmax{\hat{p}_Y(\bY) - p_W(\bY)}.
	 \nonumber
\end{align}
Since $\norm{p_W^\prime(\cdot)}_\infty \le M_2$ by Assumption~{\sf A2},  
the desired claim of Eq~\eqref{eqn:upper-bound-proposition-one} follows 
from the above estimate and Lemma~\ref{lemma:high-prob-entrywise-bound}.

\paragraph{Proof of Eq~\eqref{eqn:upper-bound-Info-W-high-prob}}
Recall the definition of $\delta_{W, \eps}$ at Eq~\eqref{eqn:def-Info-W-eps}. 
Our proof starts from the estimate
\begin{equation}
\label{eqn:upper-bound-Info-W-first-step}
\big|\hat{\Info}_{W, \eps} - \Info_W\big| \le 
	\left|\hat{\Info}_{W,\eps} - \Info_{W,\eps} \right| + \delta_{W, \eps}.
\end{equation}
Recall the definition of function $g:\R^2 \to \R$ in Eq~\eqref{eqn:def-g}. 
We can express the estimator $\hat{\Info}_{W,\eps}$ as, 
\begin{equation}
\hat{\Info}_{W,\eps} = \bar{\Info}_{W, \eps} + \eps
\end{equation}
where we define 
\begin{equation}
\bar{\Info}_{W, \eps} \defeq \frac{1}{mn}\sum_{i \in [m], j \in [n]} 
	g^2\left(\hat{p}_Y^\prime(\wtilde{Y}_{i, j}), \hat{p}_Y(\wtilde{Y}_{i, j})\right).
\end{equation}
Now, define the auxiliary random variable $\wtilde{\Info}_{W,\eps}$ to be: 
\begin{equation}
\wtilde{\Info}_{W,\eps} = \frac{1}{mn}\sum_{i \in [m], j \in [n]} 
	g^2\left(p_W^\prime(W_{i, j}), p_W(W_{i, j})\right).
\end{equation}
By triangle inequality and Eq~\eqref{eqn:upper-bound-Info-W-first-step}, 
we can upper bound the difference between $\hat{\Info}_{W,\eps}$
and $\Info_{W} $ by
\begin{equation}
\label{eqn:info-triangle-inequality}
\left|\hat{\Info}_{W,\eps} - \Info_{W} \right| \le \errorone + \errortwo +
	(\delta_{W, \eps} + \eps)~~\text{for}~
	\errorone \defeq \left|\wtilde{\Info}_{W,\eps} - \Info_{W,\eps} \right|~\text{and}~
	\errortwo \defeq \left|\bar{\Info}_{W,\eps} - \wtilde{\Info}_{W,\eps} \right|.
\end{equation}
We upper bound the error $\errorone$ first. Indeed, the next lemma provides a high 
probability onto $\errorone$, whose proof is deferred in Appendix Section 
\ref{sec:proof-lemma-f-abs-bar-bound}.
\begin{lemma}
\label{lemma:f-abs-bar-bound}
The following inequality holds for all $t > 0$, 
\begin{equation}
\label{eqn:dif-Infos}
\P \left( \errorone > t + \delta_{W, \eps}\right) \le 
	2 \exp \left(- \frac{mn t^2 \eps^{4}}{2M_2^4}\right). 
\end{equation}
\end{lemma}\noindent\noindent
Now, if we plug in $t= 2\eps^{-2}M_2^2(mn)^{-1/2}\log(mn)$ into 
Eq~\eqref{eqn:dif-Infos}, we know that, for some constant $C > 0$,
we have with probability at least $1-(mn)^{-(1-\kappa)}$, 
\begin{equation}
\label{eqn:final-error-two-bound-info}
\errorone \leq C\eps^{-2}(mn)^{-1/2}\log(mn) + \difI.
\end{equation}
Next, we upper bound error $\errortwo$. 
We introduce the auxiliary quantities $\resT_{n, 1}$ and $\resT_{n, 2}$ below, 
\begin{align}
\label{eqn:def-T-n-1}
\resT_{n, 1}^2 = \normmax{\hat{p}_Y(\wtilde{\bY}) - p_W(\wtilde{\bY})}^2 +  
		\frac{1}{mn} \norm{X}_F^2 + \wbar{W}^2, \\
\resT_{n, 2}^2 = \normmax{\hat{p}^\prime_Y(\wtilde{\bY}) - p^\prime_W(\wtilde{\bY})}^2 +  
		\frac{1}{mn} \norm{X}_F^2 + \wbar{W}^2.	
\end{align}
The following Lemma~\ref{lemma:upper-bound-error-one} provides a deterministic 
upper bound on the error term~$\errortwo$. The proof is based on tedious calculations,
and is deferred into appendix Section~\ref{sec:proof-upper-bound-error-one}.
\begin{lemma}
\label{lemma:upper-bound-error-one}
There exists some constant $C$ depending on $M, M_1, M_2$ such that, 
\begin{equation}
\label{eqn:upper-bound-error-one-first-step}
\errortwo \le C (\dezero^{-2}\resT_{n, 1} +  \dezero^{-1}\resT_{n, 2}) 
	\left(\wtilde{\Info}_{W,\eps}^{1/2}  + \dezero^{-2}\resT_{n, 1} +  \dezero^{-1}\resT_{n, 2}\right). 
\end{equation}
\end{lemma} \noindent\noindent
Motivated by Lemma~\ref{lemma:upper-bound-error-one}, we study high probability 
upper bounds onto $\wtilde{\Info}_{W,\eps}$, $T_{n, 1}$ and $T_{n, 2}$.
We start by giving high probability bound on $\wtilde{\Info}_{W,\eps}$. Indeed, 
by triangle inequality and Eq~\eqref{eqn:final-error-two-bound-info}, we know with 
probability at least $1-(mn)^{-(1-\kappa)}$,
\begin{equation}
\wtilde{\Info}_{W,\eps} \le C\eps^{-2}(mn)^{-1/2}\log(mn) + \Info_{W, \eps} + \difI.
\end{equation}
Now that by definition $ \Info_{W, \eps}  \le \Info_W$ and thus by triangle inequality 
$\difI \le 2 \Info_W$. This shows that, for some constant $C > 0$, with probability 
at least $1-(mn)^{-(1-\kappa)}$, 
\begin{equation}
\label{eqn:high-prob-bound-f-W}
\wtilde{\Info}_{W,\eps} \le C\eps^{-2}(mn)^{-1/2}\log(mn) + 3 \Info_W.
\end{equation}
Next, we give high probability upper bounds on $\resT_{n, 1}$ and $\resT_{n, 2}$. To do so, we
note the lemma below. The proof is given in appendix Section~\ref{sec:lemma-W-bar-bound}. 
\begin{lemma} 
\label{lemma:W-bar-bound}
The following inequality holds for all $t > 0$, 
\begin{equation*}
\P \left(|\wbar{W}| \geq t (mn)^{-1/2}\right)\leq t^{-2}M_1.
\end{equation*}
\end{lemma}
Now, according to Lemma~\ref{lemma:W-bar-bound}, $|\wbar{W}| \le M_1^{1/2} (mn)^{-\kappa/2}$ with 
probability at least $1-(mn)^{-(1-\kappa)}$. Since by assumption $\bX \in \cF_{m,n}(r, M, \eta)$, we have 
\begin{equation}
\label{eqn:bound-X-Fnorm}
\norm{\bX}_F \leq \rank^{1/2}(\bX) \norm{\bX}_{\op} \leq r^{1/2} M (mn)^{1/4}.
\end{equation}
Lemma~\ref{lemma:high-prob-entrywise-bound-info} now implies that with probability at least $1-(mn)^{-(1-\kappa)}$, 
\begin{equation}
\label{eqn:high-prob-bound-on-T-n}
\resT_{n, 1} \le C(\res_{n, 1} + (mn)^{-\kappa/2})~~
	\text{and}~~\resT_{n, 2} \le C(\res_{n, 2} +  (mn)^{-\kappa/2}).
\end{equation}
Plugging Eq~\eqref{eqn:high-prob-bound-f-W} and Eq~\eqref{eqn:high-prob-bound-on-T-n} into 
Eq~\eqref{eqn:upper-bound-error-one-first-step}, we get with probability at least $1- (mn)^{-(1-\kappa)}$
\begin{equation}
\label{eqn:final-error-one-bound-info}
\errortwo \le C \left(\res_n(\eps) + \eps^{-2}(mn)^{-\kappa/2}\right)
	\left(\res_n(\eps) + \eps^{-2}(mn)^{-\kappa/2} + \eps^{-1}(mn)^{-1/4}\log(mn)^{1/2} 
		+  \Info_W^{1/2}\right).
\end{equation}
Now for some constant $C, c > 0$, we know when $\eps \le c$, there exists 
$n_0 = n_0(\eps)$ such that for $n \ge n_0$, 
\begin{equation*}
\res_n(\eps) + \eps^{-2}(mn)^{-\kappa/2} + \eps^{-1}(mn)^{-1/4}\log(mn)^{1/2} + \Info_W^{1/2} \le C,
\end{equation*}
and hence by Eq~\eqref{eqn:final-error-one-bound-info}, 
\begin{equation}
\errortwo \le C \left(\res_n(\eps) + \eps^{-2}(mn)^{-\kappa/2}\right)
\end{equation}
The desired result in Eq~\eqref{eqn:upper-bound-Info-W-high-prob} now 
follows by plugging Eq~\eqref{eqn:final-error-one-bound-info} and 
Eq~\eqref{eqn:final-error-two-bound-info} into Eq~\eqref{eqn:info-triangle-inequality}.
%

\subsubsection{Proof of Lemma~\ref{lemma:high-prob-entrywise-bound} and  
Lemma~\ref{lemma:high-prob-entrywise-bound-info}}
\label{sec:proof-lemma-high-prob-entrywise-bound}
We only prove the high probability result for Eq~\eqref{eqn:high-prob-entrywise-bound-p} 
and Eq~\eqref{eqn:high-prob-entrywise-bound-p-info} since we can show the similar results 
for Eq~\eqref{eqn:high-prob-entrywise-bound-p-prime} and Eq
\eqref{eqn:high-prob-entrywise-bound-p-prime-info} in a similar way. 
For any matrix $\bA \in \R^{m \times n}$, denote the following auxiliary function 
\begin{equation}
\hat{p}_W(x; \bA) = \frac{1}{mn h_n} \sum_{i \in [m], j\in [n]} 
	K\left(\frac{W_{i, j} + A_{i, j} - \wbar{A} - x}{h_n}\right)
~~\text{and}~~\hat{q}_W(x; \bA) = \E \hat{p}_W(x; \bA),
\end{equation}
where the expectation in definition of $\hat{q}_W$ is taken with respect to the 
random matrix $\bW$. There is a connection between $\hat{p}_W$ 
and $\hat{p}_Y$: we have $\hat{p}_Y(x) = \hat{p}_W(x+\wbar{W}; \bX)$ for $x\in \R$. 
Now, using those notation, our target is to show that, for some constant $C$, with probability 
at least $1-(mn)^{-(1-\kappa)}$,
\begin{equation}
\label{eqn:proof-target-high-prob-bound}
\normmax{\hat{p}_W(\bY + \wbar{W}\one_m \one_n{^\sT}; \bX) - p_W(\bY)} \le 
	C \left(\res_{n, 1} +  \big|\wbar{W}\big|\right),
\end{equation}
and with probability at least $1-(mn)^{-(1-\kappa)}$,
\begin{equation}
\label{eqn:proof-target-high-prob-bound-Info}
\normmax{\hat{p}_W(\wtilde{\bY} + \wbar{W}\one_m \one_n{^\sT}; \bX) - p_W(\wtilde{\bY})} \le 
	C \left(\res_{n, 1} +  \big|\wbar{W}\big|\right).
\end{equation}
The proof of Eq~\eqref{eqn:proof-target-high-prob-bound} and Eq~\eqref{eqn:proof-target-high-prob-bound-Info} 
is based on standard arguments in empirical process theory. It is convenient to list our proof strategies into the 
following three steps.  
\begin{enumerate}
\item In the first step, we show that, with high probability the magnitude of $\wbar{W}$ is `small' and 
that of $\bY$ is not `too large'. 
More precisely, for $\eta_n = M_1^{1/2}(mn)^{-\kappa/2}$ and 
$T_n = 2M_1^{1/2}(mn)$, with probability at least $1-(mn)^{-(1-\kappa)}$, 
\begin{equation}
\label{eqn:high-prob-bound-argument}
|\wbar{W}| \le \eta_n,~\normmax{\bY} \le T_n~\text{and}~\normbig{\wtilde{\bY}}_{\rm max} \le T_n.
\end{equation}
Note that $T_n \ge \eta_n$ from our definition.

\item In the second step, we show that, with high probability $\hat{p}_W(x; \bX)$ 
is `close' to $\hat{q}_W(x; \bX)$ on the interval $x \in  [-2T_n, 2T_n]$. More precisely, 
we show there exists a numerical constant $c > 0$, such that if we denote
$Q_n = 4(mn)^{1/2}h_n^{-3/2}T_n$, then for any $t > 0$, 
\begin{equation}
\label{eqn:high-prob-bound-error-variance}
\sup_{x \in [-2T_n, 2T_n]} \left|\hat{p}_W(x; \bX) - \hat{q}_W(x; \bX)\right| \le 3(mn h_n)^{-1/2}t
\end{equation}
holds with probability at least $1- MQ_nt^{-1} \exp(-ct^2/(M(t+M_2))$.
\item In the last step, we prove a uniform upper bound on the difference between 
$\hat{q}_W(x + \wbar{W}; \bX)$ to $p_W(x)$. More precisely, we show that, with 
probability one, 
\begin{equation}
\label{eqn:high-prob-bound-error-bias}
\sup_{x \in \R} \left|\hat{q}_W(x; \bX) - p_W(x - \wbar{W})\right| \le 
	M_2 \left(M h_n^2 + \frac{1}{mn} \norm{\bX}_F^2 + |\wbar{W}|\right).
\end{equation}
\end{enumerate}
Now, we prove the desired target in Eq~\eqref{eqn:proof-target-high-prob-bound} using 
Eq~\eqref{eqn:high-prob-bound-argument}, Eq~\eqref{eqn:high-prob-bound-error-variance} 
and Eq~\eqref{eqn:high-prob-bound-error-bias}.
Indeed, for some numerical constant $C$ that is sufficiently large if we plug 
$t = C (M \vee M_2) \log(mn)/3$ into Eq~\eqref{eqn:high-prob-bound-error-variance}, 
we get with probability at least $1-(mn)^{-(1-\kappa)}$,
\begin{equation}
\label{eqn:high-prob-bound-error-variance-special-t}
\sup_{x \in [-2T_n, 2T_n]} \left|\hat{p}_W(x; \bX) - \hat{q}_W(x; \bX)\right| \le C (M \vee M_2)
	(mn h_n)^{-1/2}\log (mn).
\end{equation}
Thus, by triangle inequality, we know that, with probability at least $1-(mn)^{-(1-\kappa)}$, 
\begin{align}
&\normmax{\hat{p}_W(\bY + \wbar{W}\one_m \one_n{^\sT}; \bX) - p_W(\bY)} \nonumber \\
	& \le C (M \vee M_2)(mn h_n)^{-1/2}\log (mn) + M_2 \left(Mh_n^2 + \frac{1}{mn} \norm{\bX}_F^2
		+ \big|\wbar{W}\big|\right).
\end{align}
Eq~\eqref{eqn:proof-target-high-prob-bound} now follows by Eq \eqref{eqn:bound-X-Fnorm}. 
Note that, Eq~\eqref{eqn:proof-target-high-prob-bound-Info} 
can be proven in a similar way. 

\paragraph{Proof of Eq~\eqref{eqn:high-prob-bound-argument}}
The following lemma gives high probability upper bound on $\normmax{W}$. It
is proven in appendix Section~\ref{sec:lemma-W-bar-bound}.
\begin{lemma} 
\label{lemma:W-max-bound}
The following inequality holds for all $t > 0$, 
\begin{equation*}
\prob \left(\normmax{\bW} \ge tM_1^{1/2} (mn)^{1/2} \right) \le t^{-2}.
\end{equation*}
\end{lemma}
\noindent\noindent
Now, according to triangle inequality, we have, 
\begin{equation*}
\normmax{\bY} \le \normmax{\bW} + \normmax{\bX} \le \normmax{\bW}+ \norm{\bX}_F,
\end{equation*} 
and
\begin{equation*}
\normbig{\wtilde{\bY}}_{\rm max} \le \normbig{\wtilde{\bW}}_{\rm max} + \normbig{\wtilde{\bX}}_{\rm max}
	\le \normmax{\bW} + \left|\wbar{W}\right| +  \norm{\bX}_F .
\end{equation*}
where in the last step we use $\normbig{\wtilde{\bX}}_{\rm max} \le \normbig{\wtilde{\bX}}_F \le \norm{\bX}_F$. 
Now, Eq~\eqref{eqn:high-prob-bound-argument} follows from the high probability bound 
in Lemma~\ref{lemma:W-bar-bound}, Lemma~\ref{lemma:W-max-bound} and Eq \eqref{eqn:bound-X-Fnorm}.

\paragraph{Proof of Eq~\eqref{eqn:high-prob-bound-error-variance}}
The proof is based on standard uniform convergence arguments. We use 
$I_n$ to denote the interval $I_n = [-2T_n, 2T_n]$. Our proof proceeds in three steps. 
\begin{enumerate}
\item  First, there exists some numerical constant $c > 0$, such that for any fix $x \in \R$ 
	and $t > 0$
	\begin{equation}
	\label{eqn:pointwise-high-prob-bound-error-variance}
		\prob \left(\left|\hat{p}_W(x; \bX) - \hat{q}_W(x; \bX)\right| \ge (mn h_n)^{-1/2}t\right) \le 
			2\exp\left(-\frac{ct^2 }{M(t+M_2)}\right).
	\end{equation}
\item Next, we show both $\hat{p}_W(\,\cdot\,; \bX)$ and 
	$\hat{q}_W(\,\cdot\,; \bX)$ are Lipschitz functions with Lipschitz constant $L_n = Mh_n^{-2}$. 
\item Lastly, we use the covering type argument to show the uniform convergence 
	result in Eq~\eqref{eqn:high-prob-bound-error-variance}.
\end{enumerate}
The proof of our first step follows by a more general result (which we state as Lemma 
\ref{lemma:kernel-pointwise-high-prob}), whose proof is deferred into 
Section \ref{sec:proof-of-prop-kernel-ptwise}.
\begin{lemma}
\label{lemma:kernel-pointwise-high-prob} 
Let $\{X_i\}_{i=1}^n$ be independent continuous random variables with densities 
$\{p_{X_i}\}_{i=1}^n$. Let $K(\,\cdot\,)$ be square integrable on $\R$. 
Denote $\sigma^2$, $p_\infty$, $M_\infty$ to be the following quantities: 
\begin{equation*}
p_\infty = \max_{i \in [n]} \norm{p_{X_i}(\cdot)}_\infty, ~~M_\infty = %
\norm{K(\cdot)}_\infty ~~\text{and}~~\sigma^2 = \int_\R K^2(z) \rmd z,
\end{equation*}
where $p_{X_i}$ is the density function of $X_i$. For some $h > 0$, 
consider the following function 
\begin{equation*}
Z_n (x) = \frac{1}{nh} \sum_{i=1}^n K_{h, X_i}(x)~~ \text{where}~~K_{h,
X_i}(x) \defeq K\left(\frac{x-X_i}{h}\right).
\end{equation*}
Assume $nh \ge 1$. Then, for some numerical constant $c > 0$, we have
for all $x\in \R$ and $t > 0$, 
\begin{equation}  
\label{eqn:kernel-pointwise-high-prob}
\prob \left( \left|Z_n (x) - \E Z_n(x)\right| \geq (nh)^{-1/2} t \right)
	\leq	2\exp {\left( -\frac{c t^2}{\sigma^2 p_\infty + M_\infty t} \right)}.
\end{equation}
\end{lemma}\noindent\noindent
To be precise, Eq~\eqref{eqn:pointwise-high-prob-bound-error-variance} follows 
by plugging $p_{\infty} = M_2$, $M_{\infty} = \sigma^2 = M$ into Eq 
\eqref{eqn:kernel-pointwise-high-prob}. 

Next, we show that both $\hat{p}_W(\,\cdot\,; \bX)$ and $\hat{q}_W(\,\cdot\,;\bX)$ are $L_n$
Lipschitz. Indeed, we first compute the derivative of $\hat{p}_W(\,\cdot\,; \bX)$, 
\begin{equation}
\frac{\rmd}{\rmd x} \hat{p}_W(x; \bX) = \frac{1}{mn h_n^2} \sum_{i \in [m], j \in [n]}
	K^\prime \left(\frac{W_{i, j} + \wtilde{X}_{i, j} - x}{h_n}\right),
\end{equation}
which we can see from triangle inequality that, 
\begin{equation}
\norm{\frac{\rmd}{\rmd x} \hat{p}_W(\,\cdot\,; \bX)}_\infty \le h_n^{-2} \norm{K^\prime}_\infty 
	\le h_n^{-2} M, 
\end{equation}
where the last inequality follows by our assumption on the kernel $K$. This shows 
that the function $\hat{p}_W(\,\cdot\,; \bX)$ is $L_n$ Lipschitz. The function 
$\hat{q}_W(\,\cdot\,; \bX) = \E \hat{p}_W(\,\cdot\,; \bX)$ is thus also $L_n$ Lipschitz. 

\newcommand{\cover}{\mathcal{D}}
Lastly, we show the desired result in Eq~\eqref{eqn:high-prob-bound-error-variance} 
by covering argument. Fix $t > 0$ and let $\Delta(n, t) \defeq 
L_n^{-1} (mn h_n)^{-1/2}t$. Denote $\cover(n, t)$ to be the minimal $\Delta(n, t)$ cover 
of the interval $I_n = [-2T_n, 2T_n]$. It is clear that, the 
cardinality of $\cover(n, t)$ satisfies, 
\begin{equation}
\left|\cover(n, t)\right| \le \Delta(n, t)^{-1}|I_n| = 4 \Delta(n, t)^{-1}T_n = t^{-1}MQ_n.
\end{equation} 
Now, note that, the high probability bound in Eq~\eqref{eqn:pointwise-high-prob-bound-error-variance}
holds for any fix $x \in \cover(n, t)$. Thus, if we denote the event $\event(n, t)$ to be, 
\begin{equation*}
\event(n, t) = \left\{\max_{x \in \cover(n, t)}\left|\hat{p}_W(x; \bX) - \hat{q}_W(x; \bX)\right| 
	\le (mn h_n)^{-1/2}t\right\}
\end{equation*}
then by union bound, we know that, 
\begin{align}
\prob \left(\event(n, t)^c \right)
	\le 2\left|\cover(n, t)\right| \exp\left(-\frac{ct^2 }{M(t+M_2)}\right)
	\le 2t^{-1} MQ_n \exp\left(-\frac{ct^2 }{M(t+M_2)}\right).
\end{align}
Now, the desired claim follows if we can show that, the event specified by 
Eq~\eqref{eqn:high-prob-bound-error-variance} always holds on event $\event(n, t)^c$.
This is simple, as for all $x \in I_n$, by definition there exists some $x^\prime 
\in \cover(n, t)$ such that $|x-x^\prime| \le \Delta(n, t)$. Now, the Lipschitz property
of both $\hat{p}_W(\,\cdot\,; \bX)$ and $\hat{q}_W(\,\cdot\,; \bX)$ gives that, 
\begin{equation}
\left|\hat{p}_W(x; \bX) - \hat{q}_W(x; \bX)\right|	\le 
	\left|\hat{p}_W(x^\prime; \bX) - \hat{q}_W(x^\prime; \bX)\right| + L_n\Delta(n, t)
	= 3(mn h_n)^{-1}t.
\end{equation}
As $x \in I_n$ is arbitrary, this gives Eq~\eqref{eqn:high-prob-bound-error-variance}.

\paragraph{Proof of Eq~\eqref{eqn:high-prob-bound-error-bias}}
We start by evaluating the function $\hat{q}_W(x; \bX)$. Indeed, for $x \in \R$, 
\begin{align}
\hat{q}_W(x; \bX) &= \frac{1}{mn h_n} \sum_{i \in [m], j \in [n]} \int_{\R} 
	K\left(\frac{w + \wtilde{X}_{i, j} - x}{h_n}\right) p_W(w) \rmd w \nonumber  \\
	&= \frac{1}{mn} \sum_{i \in [m], j \in [n]} \int_{\R} 
	K(t) p_W(h_n t + x - \wtilde{X}_{i, j}) \rmd t. 
\label{eqn:represent-hat-q-W}
\end{align}
Since $K(\,\cdot\,)$ is a first order kernel, and $\sum_{i \in [m], j\in [n]} \wtilde{X}_{i, j} = 0$, we have the identity, 
\begin{equation}
p_W(x) = \frac{1}{mn} \sum_{i \in [m], j\in [n]}
	\int_{\R} K(t) \left(p_W(x) - p_W^\prime(x) (h_n t + \wtilde{X}_{i, j})\right)\rmd t.
\label{eqn:represent-p-W}
\end{equation}
Now, take difference between Eq~\eqref{eqn:represent-hat-q-W} and 
Eq~\eqref{eqn:represent-p-W}. Jensen's inequality implies
\begin{equation}
\left|\hat{q}_W(x; \bX) - p_W(x)\right| \le \frac{1}{mn} \sum_{i \in [m], j\in [n]}
	\int_{\R} K(t) \left|p_W(h_nt + x - \wtilde{X}_{i, j}) - p_W(x) - p_W^\prime(x) (h_n t + \wtilde{X}_{i, j})\right| \rmd t.
\label{eqn:first-upper-bound-hat-q-p-diff}
\end{equation}
Now, by assumption~{\sf A2}, we know that, $\left|p_W^{\prime \prime}(x)\right| \le M_2$
for all $x \in \R$. Thus, intermediate value theorem gives that, for all $t, x \in \R$ and 
all $i\in [m], j \in [n]$,
\begin{equation} 
\left\vert p_W(h_n t + x - \wtilde{X}_{i, j}) - p_W(x)- p_W^\prime(x) (h_n t+\wtilde{X}_{i, j})  \right\vert
\leq  \half M_2 (h_n t+\wtilde{X}_{i, j})^2.
\label{eqn:taylor-pW-second-order}
\end{equation}
Plugging the estimate in Eq~\eqref{eqn:taylor-pW-second-order} into 
Eq~\eqref{eqn:first-upper-bound-hat-q-p-diff}, we get for $x \in \R$, 
\begin{align}
\left|\hat{q}_W(x; \bX) - p_W(x)\right| &\le \frac{M_2}{2mn} \sum_{i \in [m], j\in [n]}
	\int_{\R} K(t)(h_n t+\wtilde{X}_{i, j})^2 \rmd t\nonumber \\
	&= \half M_2 h_n^2 \int_{\R} K(t)t^2 \rmd t + \frac{M_2}{2mn} \sum_{i \in [m], j \in [n]}\wtilde{X}_{i, j}^2. 
\label{eqn:second-upper-bound-hat-q-p-diff}
\end{align}
Now, we note that, by definition of $\wtilde{\bX}$, we know that, 
\begin{equation}
\label{eqn:frob-center-bound}
 \sum_{i \in [m], j \in [n]}\wtilde{X}_{i, j}^2 = \norm{\wtilde{\bX}}_F^2 
 	\le \norm{\bX}_F^2
\end{equation}
By our assumption on the kernel $K(\cdot)$, we see that, $\int K(t) t^2 \rmd t \le M$ for some 
$M > 0$. Thus, Eq~\eqref{eqn:second-upper-bound-hat-q-p-diff} and Eq
\eqref{eqn:frob-center-bound} together give us, 
\begin{equation}
\sup_{x \in \R} \left|\hat{q}_W(x; \bX) - p_W(x)\right| \le \half M_2 \left(M h_n^2 + \frac{1}{mn} \norm{\bX}_F^2\right).
\end{equation}
Note that $p_W$ is $M_2$ Lipschitz as $\norm{p_W^\prime}_\infty \le M_2$ by Assumption~{\sf A2}. 
Thus, by triangle inequality,  
\begin{align}
\sup_{x \in \R} \left|\hat{q}_W(x; \bX) - p_W(x -\wbar{W})\right| &\le 
	\sup_{x \in \R} \left|\hat{q}_W(x; \bX) - p_W(x)\right| 
	+ \sup_{x \in \R} \left|p_W(x) - p_W(x -\wbar{W})\right| \nonumber \\
	&\le	M_2 \left(M h_n^2 + \frac{1}{mn} \norm{\bX}_F^2 + |\wbar{W}|\right).
\label{eqn:ultimate-bias-estimate}
\end{align}
Now, the desired claim in Eq~\eqref{eqn:high-prob-bound-error-bias} follows.


\newcommand{\gauss}{{\normal}}
\newcommand{\Vol}{{\rm Vol}}
\newcommand{\ve}{\varepsilon}
\newcommand{\edge}{\rm{edge}}
\newcommand{\cc}{\rm{cc}}
\newcommand{\nv}{\rm{nv}}
\newcommand{\ball}{\mathbb{B}}
\newcommand{\bI}{\bold{I}}
\newcommand{\bl}{\bold{l}}
\newcommand{\bi}{\bold{i}}
\newcommand{\bj}{\bold{j}}
\newcommand{\type}{\mathcal{T}}

\section{Proof of Theorem~\ref{theorem:general-deform}}
\label{app:RMT}
\subsection{Proof Outline}
Throughout the proof, we denote 
\begin{equation}
\bZ_n = (mn)^{-1/4}\bW_n
\end{equation} 
to be the normalized noise matrix.

First, under the setting where $\bZ_n$ is a Gaussian random matrix, we observe 
that the results in~\cite{BGN12} give the desired 
Eq~\eqref{eqn:singular-value-convergence},
Eq~\eqref{eqn:singular-vector-sign-issue},
Eq~\eqref{eqn:singular-vector-convergence-A}
and Eq~\eqref{eqn:singular-vector-convergence-B}. 
To be clear, an easy application of~\cite[Theorem 2.9, 2.10]{BGN12} gives 
Eq~\eqref{eqn:singular-value-convergence},
Eq~\eqref{eqn:singular-vector-convergence-A},
and Eq~\eqref{eqn:singular-vector-convergence-B}
when $\bZ_n$ is a Gaussian random matrix. 
Moreover, let $D: (1,\infty) \to \R_+$ be
\begin{equation}
D(\sigma) = \frac{G^{(1)}(\sigma) \vee G^{(2)}(\sigma)}{G^{(1)}(\sigma) \wedge G^{(2)}(\sigma)}
\end{equation}
Then when $\bZ_n$ is a Gaussian random matrix, Eq~(18) in the mid of the 
proof of~\cite[Theorem 2.10]{BGN12} (see~\cite[Section 5]{BGN12}) 
implies that for $i \in [k]$,
\begin{equation}
 \langle \hbu_i, \bu_i \rangle - D(\sigma_i) \langle \hbv_i, \bv_i \rangle \asto 0,
\end{equation}
which together with Eq~\eqref{eqn:singular-vector-convergence-A},
and Eq~\eqref{eqn:singular-vector-convergence-B}, proves the desired 
Eq~\eqref{eqn:singular-vector-sign-issue}.

Our proof follows closely to the proof of~\cite[Theorem 2.9, 2.10, Eq~(18)]{BGN12}. 
Denote $F_{\gamma}$ to be the normalized Marchenko-Pastur distribution
such that, for any subset $A \subseteq \R$, 
\begin{equation}
F_{\gamma}(A) = \begin{cases}
	\left(1 - \gamma^{-1}\right)\indic{0 \in A} + \nu(A)~~&\gamma \ge 1 \\
	\nu(A)~~&\gamma \le 1, 
\end{cases}
\end{equation}
where 
\begin{equation}
\rmd \nu(x) = \frac{1}{2\pi \gamma^{1/2} x} \sqrt{(\lambda_+ - x)(x- \lambda_{-})}
	\indic{x \in [\lambda_-, \lambda_+]}
\end{equation}
for $\lambda_\pm = (\gamma^{1/4} \pm \gamma^{-1/4})^2$. 
Denote $m(z), \tilde{m}(z): \C^+ \cup(-\infty, \lambda_-)  \cup (\lambda_+, \infty) 
\to \C \backslash \C_-$ to be the Stieltjes transform of the normalized Marchenko-Pastur 
distribution, i.e, 
\begin{equation}
m(z) = \int \frac{1}{\lambda - z} \rmd F_{\gamma}(\lambda)
~~\text{and}~~
\tilde{m}(z) = \int \frac{1}{\lambda - z} \rmd F_{\gamma^{-1}}(\lambda).
\end{equation}
Now, denote the open set $\C_{\gamma}$ to be: 
\begin{equation}
\C_{\gamma} = \{z \in \C: |z| > \lambda_+ = \gamma^{1/4} + \gamma^{-1/4}\}. 
\end{equation}
A careful checking of the proof of~\cite[Theorem 2.9, 2.10, Eq~(18)]{BGN12} 
shows that  it suffices to prove for any fixed sequences of unit vectors $\{\bu_n\}_{n\in \N}$, 
$\{\bu_n^\prime\}_{n\in \N}$, $\{\bv_n\}_{n\in \N}$ and $\{\bv_n^\prime\}_{n\in \N}$
and compact set $S \subseteq \C_{\gamma}$, the following convergence is uniform 
on the set $S$, 
\begin{enumerate}[(i)]
\item~\label{item:point-one} 
	$\langle \bu_n, (z^2 \id_m - \bZ_n \bZ_n^{\sT})^{-1} \bu_n^\prime \rangle
		+ m(z^2)\langle \bu_n, \bu_n^\prime\rangle \asto 0$.
\item~\label{item:point-two}
	$\langle \bv_n, (z^2 \id_n - \bZ_n^{\sT} \bZ_n)^{-1} \bv_n^\prime \rangle
		+ \tilde{m}(z^2) \langle \bv_n, \bv_n^\prime\rangle \asto 0$.
\item~\label{item:point-three}
 	$\langle \bu_n, (z^2 \id_m - \bZ_n \bZ_n^{\sT})^{-1} \bZ_n \bv_n \rangle \asto 0$. 
\item~\label{item:point-four}
	$\langle \bu_n, (z^2 \id_m - \bZ_n \bZ_n^{\sT})^{-2} \bu_n^\prime \rangle
		- m^\prime(z^2)\langle \bu_n, \bu_n^\prime\rangle \asto 0$.
\item~\label{item:point-five}
	$\langle \bv_n, (z^2 \id_n - \bZ_n^{\sT} \bZ_n)^{-2} \bZ_n^{\sT} \bZ_n \bv_n^\prime \rangle
		- (\tilde{m}^\prime(z^2)z^2 + \tilde{m}(z^2))\langle \bv_n, \bv_n^\prime\rangle \asto 0$.
\item~\label{item:point-six}
	$\langle \bu_n, (z^2 \id_m - \bZ_n \bZ_n^{\sT})^{-2} \bZ_n \bv_n \rangle \asto 0$.
\end{enumerate}
In fact, the proof of~\cite[Theorem 2.9, 2.10]{BGN12} indicates that the 
above coefficients in those almost sure limit (i.e., $m(z^2)$, $\tilde{m}(z^2)$,
$0$, $m^\prime(z^2)$, $m^\prime(z^2)z^2 + m(z^2)$ and 
$0$  respectively) determine the almost sure limits of 
$\left\{\hat{\sigma_i}\right\}_{i\in [r]}$, 
$\left\{\langle \hat{u}_i, u_j \rangle\right\}_{i, j\in [r]}$ and 
$\left\{\langle \hat{v}_i, v_j \rangle\right\}_{i, j\in [r]}$. Note that, 
the above assertions suggest that these limit are the same as the limits that we 
would have if $\bZ$ is Gaussian random matrix, and so consequently the 
almost sure limits of $\left\{\hat{\sigma_i}\right\}_{i\in [r]}$, 
$\left\{\langle \hat{u}_i, u_j \rangle\right\}_{i, j\in [r]}$ and 
$\left\{\langle \hat{v}_i, v_j \rangle\right\}_{i, j\in [r]}$ will be 
the same as if $\bZ$ is Gaussian random matrix, which is given 
by~\cite[Theorem 2.9, 2.10]{BGN12} and stated here in our desired claim
Eq~\eqref{eqn:singular-value-convergence},
Eq~\eqref{eqn:singular-vector-convergence-A} and 
Eq~\eqref{eqn:singular-vector-convergence-B}.

The rest of the proofs is devoted to prove point~\eqref{item:point-one}
to point~\eqref{item:point-six}. The organization is as follows. First, we prove 
point~\eqref{item:point-one} to point~\eqref{item:point-three}. Next, 
we show that point~\eqref{item:point-four} to point~\eqref{item:point-six} 
follows immediately from point~\eqref{item:point-one} to point~\eqref{item:point-three}
with the so-called~\emph{derivative trick} (see e.g.~\cite{DobribanWa18}). 

\paragraph{Proof of point~\eqref{item:point-one} to point~\eqref{item:point-three}}
In fact, point~\eqref{item:point-one} and point~\eqref{item:point-two} follows easily 
from~\cite[Theorem 2.4, 2.5]{BloemendalKnYaYi14}. Thus, we only need to prove 
point~\eqref{item:point-three}. 
Denote $\tilde{\bZ}_n = (z^2 \id_m - \bZ_n \bZ_n^{\sT})^{-1} \bZ_n$.
Previous results for controlling $\langle \bu_n, \tilde{\bZ}_n \bv_n\rangle$ rely 
heavily on the concentration bounds for random $\bu_n$ and $\bv_n$ (See for instance
\cite[Proposition 8.12]{BGN12} for details). In contrast, here $\bu_n$ and $\bv_n$
are fixed vectors and the only randomness here comes from the matrix 
$\tilde{\bZ}_n$. For this reason, we adopt the moment methods to 
control this quantity. In execution of this moment calculations, we develop
some combinatorial arguments. As the proof is a little bit lengthy, we defer it 
to section~\ref{section:proof-item-point-three}.
\paragraph{Proof of point~\eqref{item:point-four} to point~\eqref{item:point-six}}
Here, we employ the derivative trick~\cite[Lemma 2.14]{BaiSi10}, which is based 
on the Montel and Vitali's convergence theorem from complex analysis 
\cite[Theorem 3-4, Section 7]{Remmert13}. 
The idea is the following. Say we want to show point~\eqref{item:point-four}, 
i.e., $f_n(z)= \langle \bu_n, (z^2 \id_m - \bZ_n \bZ_n^{\sT})^{-2} \bu_n^\prime \rangle
		- m^\prime(z^2)\langle \bu_n, \bu_n^\prime\rangle \asto 0$. Then the 
idea is to construct a function $F_n(z)$ such that it satisfies the following 
properties: (1) the derivative of $F_n$ is precisely $f_n$ and (2) the limit of the 
functions $F_n$ is easy to evaluate. The lemma below shows that we can 
usually legitimately change the order between taking limits and derivatives. 

\begin{lemma}[Lemma 2.14~\cite{BaiSi10}]
\label{lemma:Vitali-lemma}
Let $f_1, f_2, \ldots$ be analytic functions on a domain $D$ in the complex plane 
satisfying $|f_n(z)| \le M$ for some constant $M$ and all $z\in D$. Suppose that 
there is an analytic function $f$ on $D$ such that $f_n(z) \to f(z)$ for all $z\in D$.
Then it holds that $f_n^\prime(z) \to f^\prime(z)$ and the convergence is uniform
for any compact set $S \subseteq D$. 
\end{lemma}
Now, we show point~\eqref{item:point-four} first. To do so, consider the function 
\begin{equation}
F_{n, 1}(t) = \langle \bu_n, (t \id_m - \bZ_n \bZ_n^{\sT})^{-1} \bu_n^\prime \rangle
	+ m(t) \langle \bu_n, \bu_n^\prime\rangle.
\end{equation}
Denote the set $\C_{\gamma}^2 = \left\{z^2: z\in \C_{\gamma}\right\}$. Fix 
some open set $S$ such that its closure $\bar{S} \subseteq \C_{\gamma}^2$. 
Now that~\cite[Theorem 5.2]{BaiSi10} shows that 
$\opnorm{\bZ_n} \asto \lambda_+ = \gamma^{1/4} + \gamma^{-1/4}$.
Moreover, since $\bar{S} \subseteq \C_{\gamma}^2$, we have 
$\inf_{z\in \bar{S}} |z| > \lambda_+^2$. Thus, for any $\omega \in \Omega$,
there exists sufficiently large $N (= N(\omega))$ and $M (= M(\omega))$ 
such that the function $F_{n, 1}$ is 
well-defined and analytic in $S$, and moreover, for all $n \ge N$, 
$|F_{n, 1}(z)| \le M$ for all $z \in S$. Now that $F_{n, 1}(t) \to 0$ for $t \in S$
by point~\eqref{item:point-one}. Thus, Lemma~\ref{lemma:Vitali-lemma}
implies that $F_{n, 1}^\prime(t) \to 0$ for $t \in S$, or equivalently, 
\begin{equation}
\langle \bu_n, (t \id_m - \bZ_n \bZ_n^{\sT})^{-2} \bu_n^\prime\rangle
	- m^\prime(t) \langle \bu_n, \bu_n^\prime\rangle \to 0,
\end{equation}
and such convergence is uniform over $z \in S^\prime$ for any compact 
subset $S^\prime \subseteq S$. This gives the desired claim of
point~\eqref{item:point-four}.

Next, we overview the proof for point~\eqref{item:point-five}
and point~\eqref{item:point-six}. Their proof strategies are 
the same as that for proving point~\eqref{item:point-four}. For 
point~\eqref{item:point-five}, we start with the identity 
\begin{equation*}
\langle \bv_n, (z^2 \id_n - \bZ_n^{\sT} \bZ_n)^{-2} \bZ_n^{\sT} \bZ_n \bv_n^\prime  \rangle
= z^2 \langle \bv_n,  (z^2 \id_n - \bZ_n^{\sT} \bZ_n)^{-2}  \bv_n^\prime \rangle
	-  \langle \bv_n, (z^2 \id_n - \bZ_n^{\sT} \bZ_n)^{-1}  \bv_n^\prime \rangle.
\end{equation*}
Since point~\eqref{item:point-two} already shows that 
$\langle \bv_n, (z^2 \id_n - \bZ_n^{\sT} \bZ_n)^{-1} \bv_n^\prime \rangle
+ \tilde{m}(z)\langle \bv_n, \bv_n^\prime\rangle \asto 0$. Thus, it suffices 
to show that, 
\begin{equation}
\label{eqn:intermediate-point-five}
\langle \bv_n, (z^2 \id_n - \bZ_n^{\sT} \bZ_n)^{-2} \bv_n^\prime \rangle 
	- \tilde{m}^\prime(z^2) \langle \bv_n, \bv_n^\prime\rangle \asto 0.
\end{equation}
To show Eq~\eqref{eqn:intermediate-point-five}, one needs to 
consider the function 
\begin{equation}
F_{n, 2}(t) = \langle\bv_n, (t \id_m - \bZ_n^{\sT} \bZ_n)^{-1} \bv_n^\prime\rangle
	+ \tilde{m}(t) \langle \bv_n, \bv_n^\prime\rangle,
\end{equation}
and apply the similar derivative trick to $F_{n, 2}$, noticing that 
$F_{n, 2} \asto 0$ by point~\eqref{item:point-two} and the class 
of functions $F_{n, 2}$ is uniformly bounded on any open set 
$S$ such that its closure $\bar{S} \subseteq \C_{\gamma}^2$. 
For point~\eqref{item:point-six}, 
one instead needs to consider the function 
\begin{equation}
F_{n, 3}(t) = \langle\bu_n,  (t \id_m - \bZ_n \bZ_n^{\sT})^{-1} \bZ_n \bv_n\rangle
	+ \tilde{m}(t) \langle \bv_n, \bv_n^\prime\rangle,
\end{equation}
and apply the similar derivative trick to $F_{n, 3}$, noticing that 
$F_{n, 3} \asto 0$ by point~\eqref{item:point-three} and the class 
of functions $F_{n, 3}$ is uniformly bounded on any open set 
$S$ such that its closure $\bar{S} \subseteq \C_{\gamma}^2$.

\subsection{Proof of Point~\eqref{item:point-three}}
\label{section:proof-item-point-three}
Denote the function 
\begin{equation}
f_n(z) = \bu_n^{\sT} \left(z^2\bI_{m \times m} - \bZ_n \bZ_n^{\sT}\right)^{-1} \bZ_n \bv_n.
\end{equation}
By~\cite[Theorem 5.2]{BaiSi10}, we know that 
$\opnorm{\bZ_n} \asto \lambda_+ = \gamma^{1/4} + \gamma^{-1/4}$.
Hence, for any compact set $S \subseteq \C_{\gamma}$ and 
$\omega \in \Omega$, there exists some $N (= N(\omega, S))$
such that the function $f_n$ is analytic in $S$ for $n \ge N$. 
Montel and Vitali's convergence theorem for analytic functions 
\cite[Theorem 3-4, Section 7]{Remmert13} imply that showing 
uniform convergence of $f_n$ to $0$
in $S$ is equivalent to showing for any fix $z \in S$, it holds
\begin{equation}
\label{eqn:converge-one-single-point}
\lim_{n \to \infty}  \bu_n^{\sT} \left(z^2\bI_{m \times m} - \bZ_n \bZ_n^{\sT} \right)^{-1} 
	\bZ_n \bv_n  \asto 0.
\end{equation}
As a result, our goal is to show that Eq~\eqref{eqn:converge-one-single-point}
holds for any fix $z \in \C_{\gamma}$.

For $z \in \C_{\gamma}$, set $r_z = \half(|z| + \lambda_+) < |z|$ and 
$\eps_z = \half(|z|- \lambda_+) > 0$. Take 
$N(\omega)\in \N$ such that $\opnorm{\bZ_n} \le r_z$ for all $n \ge N(\omega)$. 
Note that, for $n \ge N(\omega)$, the function $f_n(z)$ is 
analytic in $\ball(z, \epsilon_z)$ with Laurent expansion, 
\begin{equation}
\label{eqn:ZZT-decomposition}
\bu_n^{\sT} \left(z^2 \bI_{m \times m} - \bZ_n \bZ_n^{\sT} \right)^{-1} \bZ_n \bv_n
= \sum_{k = 0}^\infty z^{-2k-2} \bu_n^{\sT} (\bZ_n \bZ_n^{\sT})^k\bZ_n \bv_n.
\end{equation}
Now we prove that the desired claim of Eq~\eqref{eqn:converge-one-single-point} 
will follow if we can show that for all $k \in \N$, 
\begin{equation}
\label{eqn:goal-fix-k}
\lim_{n\to \infty} \bu_n^{\sT} (\bZ_n \bZ_n^{\sT})^k\bZ_n \bv_n \asto 0. 
\end{equation}
In fact, denote $A_{n, k} = z^{-2k-2} \bu_n^{\sT} (\bZ_n \bZ_n^{\sT})^k\bZ_n \bv_n$.
Then, for any $\omega \in \Omega$, we have 
$|A_{n, k}| \le B_k \defeq |z|^{-2k-2}r_z^{2k}$ for all $n \ge N(\omega)$. 
Note that $\sum_k B_k < \infty$ and for $k \in \N$, $A_{n, k} \to 0$ by assumption. 
Hence, Lebesgue's dominated convergence theorem implies that 
$\lim_{n\to \infty} \sum_k A_{n, k} \to 0$ for any $\omega \in \Omega$, giving 
the desired clam in Eq~\eqref{eqn:converge-one-single-point} thanks to the 
expansion~\eqref{eqn:ZZT-decomposition}.

Now, we show for each $k \in \N$, the desired Eq~\eqref{eqn:goal-fix-k} holds. 
To simplify our notation, define the sequence $\{c_p\}_{p\in \N}$ by 
\begin{equation}
c_p = \E |W_{i, j}|^p < \infty.
\end{equation}
By assumption, we know that $c_p$ is independent of $m, n$. As a consequence of 
Markov's inequality and Borel-Cantelli Lemma, it suffices 
to prove that for each $k \in \N$, there exists some constant $C_k$ independent 
of $n$ (but can be dependent of the sequence $\{c_p\}_{p\in \N}$) such that 
\begin{equation}
\label{eqn:fourth-moment}
\E (\bu_n^{\sT} (\bZ_n \bZ_n^{\sT})^k\bZ_n \bv_n)^4 \leq C_k n^{-2}.
\end{equation}
In the rest of the proof, we show Eq~\eqref{eqn:fourth-moment}. For notational 
simplicity, we use the compact notations $\bu \in \R^m$, $\bv \in \R^n$ and 
$\bZ \in \R^{m\times n}$ to represent the vectors $\bu_n, \bv_n$ and matrix $\bZ_n$, 
making their dependences on $n\in \N$ implicit.

\newcommand{\sing}{^{\rm s}}

The proof of Eq~\eqref{eqn:fourth-moment} is based on a moment 
calculation, and part of it mirrors the moment method proof for 
Marchenko-Pastur law~\cite[Section 3.1]{BaiSi10}. Our first step is to 
expand the LHS of Eq~\eqref{eqn:fourth-moment} for fix $k\in \N$. 
To do so, we introduce some notations. For each 
$\bi = (i_1, \ldots, i_k) \in [m]^k$ and $\bj = (j_1, \ldots, j_k) \in [n]^k$, we 
construct a graph $G_1(\bi, \bj)$ in the following way. Draw two parallel lines, 
referring to the $I$ line and the $J$ line. Plot $i_1, \ldots , i_k$ on the $I$ line 
and $j_1, \ldots ,j_k$ on the $J$ line, and draw $k$ (down) edges from $i_u$ 
to $j_u$, $u = 1, \ldots ,k$, and $k-1$ (up) edges from $j_u$ to 
$i_{u+1}, u = 1,\ldots, k$. Similarly, we denote by $G_2(\bj, \bi)$ the graph with 
$i_1, \ldots , i_k$ on the $I$ line, $j_1, \ldots ,j_k$ on the $J$ line, $k$ (up) edges 
from $j_u$ to $i_{u}, u = 1,\ldots, k$, and $k-1$ (down) edges from $i_u$ to 
$j_{u+1}$, $u = 1, \ldots ,k-1$. For $G_1(\bi, \bj)$ and $G_2(\bj, \bi)$, we define
the scalars $Z_{G_1(\bi, \bj)}, Z_{G_2(\bj, \bi)}$ by
\begin{align*}
Z_{G_1(\bi, \bj)} &= Z_{i_1 j_1} Z_{i_2 j_1} Z_{i_2 j_2} Z_{i_3 j_2} 
	\ldots Z_{i_k j_k}, \notag \\
Z_{G_2(\bj, \bi)} & = Z_{i_1 j_1} Z_{i_1 j_2} Z_{i_2 j_2} Z_{i_2 j_3} 
	\ldots Z_{i_k j_k}.
\end{align*}
For $\{\bi^{(l)}\}_{l \in [4]}$ and $\{\bj^{(l)}\}_{l \in [4]}$, let
$G(\bi^{(1:4)}, \bj^{(1:4)})$ be the graph that is the union of the graphs
$G_1(\bi^{(1)}, \bj^{(1)})$, $G_2( \bj^{(2)}, \bi^{(2)})$, 
$G_1(\bi^{(3)}, \bj^{(3)})$ and $G_2( \bj^{(4)}, \bi^{(4)})$.
For each graph $G(\bi^{(1:4)}, \bj^{(1:4)})$, define
\begin{align*}
Z_{G(\bi^{(1:4)}, \bj^{(1:4)})} 
	\defeq Z_{G_1(\bi^{(1)}, \bj^{(1)})} \cdot
		Z_{G_2( \bj^{(2)}, \bi^{(2)})}  \cdot
		Z_{G_1(\bi^{(3)}, \bj^{(3)})} \cdot
		Z_{G_2( \bj^{(4)}, \bi^{(4)})}.
\end{align*}


\begin{figure}[H]
	\centering
	\begin{subfigure}{} 
		\includegraphics[width=0.4\textwidth]{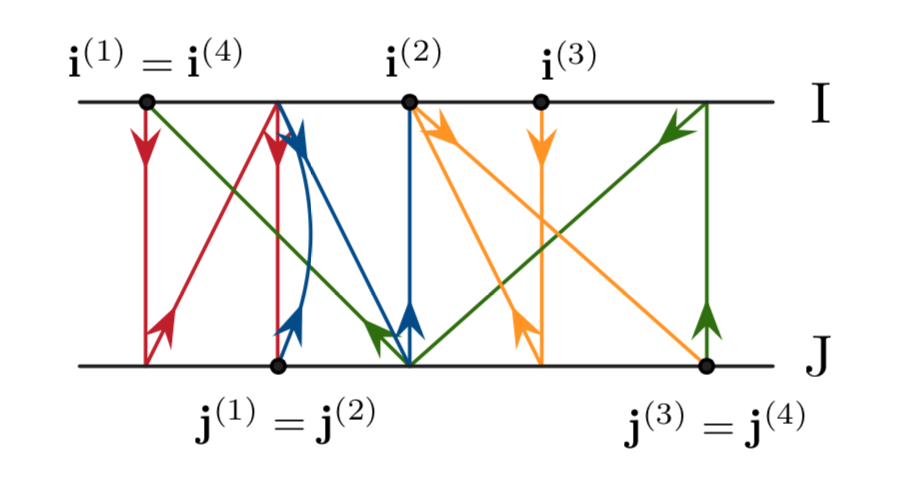}
	\end{subfigure}
	\caption{An example of $G(\bi^{(1:4)}, \bj^{(1:4)})$. The red, blue, 
		orange and blue edges correspond to the edges of the graphs
		$G_1(\bi^{(1)}, \bj^{(1)})$, $G_2( \bj^{(2)}, \bi^{(2)})$, 
		$G_1(\bi^{(3)}, \bj^{(3)})$ and $G_2( \bj^{(4)}, \bi^{(4)})$ respectively.}
\end{figure}


Lastly, for each $\bh \in [m]^4$ and $\bh^\prime \in [n]^4$, we define the set 
of graphs 
\begin{equation}
S(\bh, \bh^\prime) = \left\{G(\bi^{(1:4)}, \bj^{(1:4)}): 
	(\bi^{(l)}, \bj^{(l)}) \in \Gamma(h_l, h_l^\prime)~\text{for all}~l \in [4]\right\}
\end{equation}
where for $h \in [m]$ and $h^\prime \in [n]$, the set $\Gamma(h, h^\prime)$ is 
defined by
\begin{align}
\Gamma(h, h^\prime) = \{(\bi, \bj): \bi \in [m]^k, \bj \in [n]^k, \bi_1 = h, \bj_k = h^\prime\}.
\end{align}
Then we have the following expansion of the LHS of of 
Eq~\eqref{eqn:fourth-moment}: 
\begin{align}
\E \left( \bu^{\sT} (\bZ \bZ^{\sT})^k \bZ \bv \right)^4
&= \E \tr \left(\bu \bu^{\sT} (\bZ \bZ^{\sT})^k \bZ \bv \bv^{\sT} \bZ^{\sT} (\bZ \bZ^{\sT})^k 
\bu \bu^{\sT} (\bZ \bZ^{\sT})^k \bZ \bv \bv^{\sT} \bZ^{\sT} (\bZ \bZ^{\sT})^k \right)
	\nonumber \\
&= \sum_{\substack{\bh \in [m]^4 \\ \bh^\prime \in [n]}} 
	u_{h_1}u_{h_2}u_{h_3}u_{h_4} 
	v_{h_1^\prime} v_{h_2^\prime}v_{h_3^\prime} v_{h_4^\prime}  
	\sum_{G(\bi^{(1:4)}, \bj^{(1:4)}) \in S(\bh, \bh^\prime)} \E Z_{G(\bi^{(1:4)}, \bj^{(1:4)})}.
\label{eqn:square-expansion}
\end{align}
Now, we evaluate the RHS of Eq~\eqref{eqn:square-expansion}. For any graph
$G(\bi^{(1:4)}, \bj^{(1:4)})$ that contains a single edge, i.e., an edge not coincident 
with another edge, then we shall have
\begin{equation}
\label{eqn:zero-mean-observation}
\E Z_{G(\bi^{(1:4)}, \bj^{(1:4)})} = 0
\end{equation}
since the random variables $Z_{i, j}$ are assumed to be independent and mean $0$. 
Now, for any $\bh \in [m]^4$ and $\bh^\prime \in [n]^4$, denote 
\begin{equation*}
S\sing(\bh, \bh^\prime) = S(\bh, \bh^\prime) \cap \{G(\bi^{(1:4)}, \bj^{(1:4)}): 
	G(\bi^{(1:4)}, \bj^{(1:4)})~\text{does not have a single edge}\}
\end{equation*}
by Eq~\eqref{eqn:zero-mean-observation}, we have 
\begin{equation}
\label{eqn:reduce-to-single}
\sum_{G(\bi^{(1:4)}, \bj^{(1:4)}) \in S(\bh, \bh^\prime)} \E Z_{G(\bi^{(1:4)}, \bj^{(1:4)})}
= \sum_{G(\bi^{(1:4)}, \bj^{(1:4)}) \in S\sing(\bh, \bh^\prime)} \E Z_{G(\bi^{(1:4)}, \bj^{(1:4)})}.
\end{equation}
By H\"{o}lder's inequality, we know that, each $Z_{G(\bi^{(1:4)}, \bj^{(1:4)})}$ satisfies
\begin{equation}
\label{eqn:holder-step}
\left|\E Z_{G(\bi^{(1:4)}, \bj^{(1:4)})}\right| \le \E \left[\big(|Z_{i, j}|^{2k+1}\big)^4\right] 
	\le c_{4(2k+1)} (mn)^{-(2k+1)}. 
\end{equation}
Hence, if we set $N\sing(\bh, \bh^\prime) = |S\sing(\bh, \bh^\prime)|$, then Eq
\eqref{eqn:reduce-to-single} and Eq~\eqref{eqn:holder-step} show that 
\begin{equation}
\Bigg|\sum_{G(\bi^{(1:4)}, \bj^{(1:4)}) \in S(\bh, \bh^\prime)} \E Z_{G(\bi^{(1:4)}, \bj^{(1:4)})}\Bigg|
	\le C_k n^{-4k-2}N\sing(\bh, \bh^\prime).
\end{equation}
for some constant $C_k > 0$. Plugging it into Eq~\eqref{eqn:square-expansion}, we 
get that
\begin{equation}
\label{eqn:main-moment-method}
\E \left( \bu^{\sT} (\bZ \bZ^{\sT})^k \bZ \bv \right)^4
	\le C_k n^{-4k-2} \sum_{\substack{\bh \in [m]^4 \\ \bh^\prime \in [n]^4}} 
	N\sing(\bh, \bh^\prime)
	\prod_{l \in [4]} |u_{h_l} v_{h_l^\prime}|.
\end{equation}
Now, for $p \in \{m, n\}$, decompose the set
$[p]^4 = \cup_{l \in [5]} \type_{l, p}$, where we define $\{\type_{l, p}\}_{l\in [5]}$ by
\begin{enumerate}
\item $\type_{1, p} = \{\bl \in [p]^4: |\bl| = 4\}$
\item $\type_{2, p} = \{\bl \in [p]^4: |\bl| = 3\}$
\item $\type_{3, p} = \{\bl \in [p]^4: |\bl| = 2, \text{some entry of $(l_1, l_2, l_3, l_4)$ has multiplicity 3}\}$
\item $\type_{4, p} = \{\bl \in [p]^4: |\bl| = 2, \text{each entry of $(l_1, l_2, l_3, l_4)$ has multiplicity 2}\}$
\item $\type_{5, p} = \{\bl \in [p]^4: |\bl| = 1\}.$
\end{enumerate}
Then by Eq~\eqref{eqn:main-moment-method}, we have that 
\begin{equation}
\label{eqn:main-moment-method-new}
\E \left( \bu^{\sT} (\bZ \bZ^{\sT})^k \bZ \bv \right)^4
	\le C_k n^{-4k-2} \sum_{\substack{l_1 \in [5] \\ l_2 \in [5]}}
	M_m(l_1) M_n(l_2) N\sing(l_1, l_2), 
\end{equation}
where for each $l_1, l_2 \in [5]$, we define
\begin{equation}
M_m(l_1)\defeq \sum_{\bh \in \type_{l_1, m}}\prod_{l \in [4]} |u_{h_l}|,~~
M_n(l_2)\defeq \sum_{\bh \in \type_{l_2, n}}\prod_{l \in [4]} |v_{h_l}|
~~\text{and}~~ 
N\sing(l_1, l_2)\defeq \sup_{\substack{\bh \in \type_{l_1, m} \\ \bh^\prime \in \type_{l_2, n}}} 
	N\sing(\bh, \bh^\prime).
\end{equation}
Now, motivated by Eq~\eqref{eqn:main-moment-method-new}, we provide bounds 
on $M_p(l)$ for $p \in \{m, n\}$ and $l \in [5]$ and on $N(l_1, l_2)$ for $l_1, l_2 \in [5]$. 

\paragraph{Bounds on $M_p(l)$}  We have the following result. 
\begin{lemma}
\label{lemma:ultimate-bound-onto-M}
For some universal constant $C$, the bounds below hold for $p \in \{m, n\}$, 
\begin{equation}
M_p(l) \le C \cdot \begin{cases}
	p^2&~~\text{if $l = 1$} \\
	p&~~\text{if $l = 2$} \\
	p^{1/2}&~~\text{if $l = 3$} \\
	1&~~\text{if $l \in \{4, 5\}$}	
	\end{cases}.
\end{equation}
\end{lemma}
\begin{proof}
We only prove the case where $p = m$. The crux of the proof is the following: for 
any $l \in \N$
\begin{equation}
\label{eqn:moment-of-u}
\sum_{i} |u_i|^l \le \begin{cases}
		1&~~\text{if $l \ge 2$}\\
		m^{1/2}&~~\text{if $l = 1$}.
	\end{cases}
\end{equation}
In fact, when $l \ge 2$, since the vector $u$ lies on the unit sphere, we know that 
$\sum_{i} |u_i|^l \le \sum_{i} |u_i|^2 \le 1$. On the other hand, when $l = 1$, then 
we have $\sum_{i} |u_i| \le m^{1/2}(\sum_{i} |u_i|^2)^{1/2} \le m^{1/2}$ by Cauchy-Schwartz 
inequality. Now, using Eq~\eqref{eqn:moment-of-u}, it is easy to enumerate all 
the possibilities of $M_m(l)$:
\begin{enumerate}
\item $M_m(1) \le \big(\sum_{i} |u_i|\big)^4 \le m^2$.
\item $M_m(2) \le \choose{4}{2}\big(\sum_{i} |u_i|\big)^2 \big(\sum_{i} |u_i|^2\big) \le 6m$.
\item $M_m(3) \le \choose{4}{1}\big(\sum_{i} |u_i|\big)\big(\sum_{i} |u_i|^3\big) \le 4m^{1/2}$.
\item $M_m(4) \le \choose{4}{2}\big(\sum_{i} |u_i|^2\big)^2 \le 6$.
\item $M_m(5) \le \sum_{i} |u_i|^4 \le 1$.
\end{enumerate}
This gives the desired claim of the lemma. 
\end{proof}

\newcommand{\nav}{{\rm {nav}}}
\newcommand{\ver}{{\rm {ver}}}
\newcommand{\lab}{{\rm {lab}}}
\newcommand{\eqv}{{\rm {eqv}}}
\paragraph{Bounds on $N\sing(l_1, l_2)$}
We start by bounding $N\sing(\bh, \bh^\prime)$ for $\bh \in [m]^4, \bh^\prime\in [n]^4$. 
Define for $\bl \in \Z^4$
\begin{equation}
|\bl| = |\{l_1, l_2, l_3, l_4\}|.
\end{equation}
To bound $N\sing(\bh, \bh^\prime)$, we find the following characteristics useful: 
\begin{itemize}
\item Denote by $\chi_{\nav}(\bh, \bh^\prime)$ the maximum number of vertices of any 
	graph in the set $S\sing(\bh, \bh^\prime)$.
\item Denote by $\chi_{\edge}(\bh, \bh^\prime)$ the maximum number of edges of any 
	graph in the set $S\sing(\bh, \bh^\prime)$.
\item Denote by $\chi_{\cc}(\bh, \bh^\prime)$ the maximum number of components of 
	any graph in the set $S\sing(\bh, \bh^\prime)$.
\item Define $\chi_{\nv}(\bh, \bh^\prime) = |\bh| + |\bh^\prime|$.
\end{itemize}  
Our next lemma provides an upper bound onto $N\sing(\bh, \bh^\prime)$ based on 
the above characteristics.
\begin{lemma}
\label{lemma:basic-combinatorics}
There exists some constant $C_k > 0$ such that 
\begin{equation}
\label{eqn:N-rule}
N\sing(\bh, \bh^\prime) \leq C_k n^{(\chi_{\nav} - \chi_{\nv})(\bh, \bh^\prime)}
\le C_k n^{(\chi_{\edge} + \chi_{\cc} - \chi_{\nv})(\bh, \bh^\prime)}.
\end{equation}
\end{lemma}

\begin{figure}[H]
	\centering
	\begin{subfigure}{} 
		\includegraphics[width=0.8\textwidth]{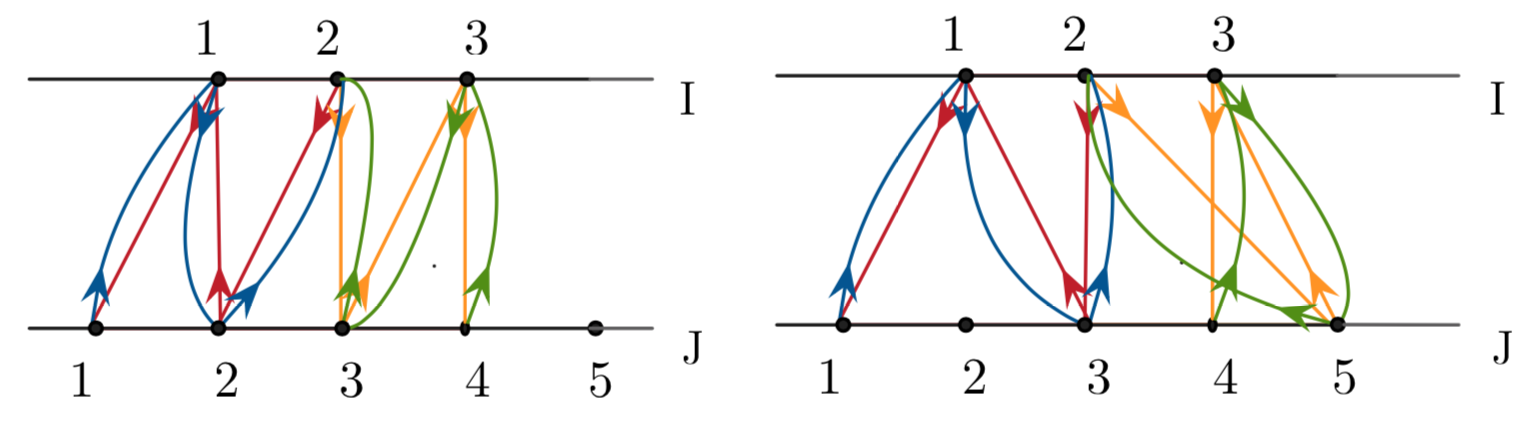}
	\end{subfigure}
	\vspace{1em}
	 
	\caption{An example of two isomorphic graphs when 
			$\bh = (2, 2, 2, 2)$ and $\bh^{\prime} = (1,1,4,4)$. 
			For each of the two graphs, the red, blue, orange and blue edges 
			correspond to edges of $G_1(\bi^{(1)}, \bj^{(1)})$, $G_2( \bj^{(2)}, \bi^{(2)})$, 
			$G_1(\bi^{(3)}, \bj^{(3)})$ and $G_2( \bj^{(4)}, \bi^{(4)})$ respectively. 
			Note that the permutation $(1,2,3,4,5) \to (1,3,5,4,2)$ maps the 
			labeling on indices of $J$ line of the graph on the left to that 
			of the graph on the right.  }
\end{figure}


\begin{proof}
For any graph $G \in S\sing(\bh, \bh^\prime)$, denote $\lab(G)$ to be the indices of the 
vertices of graph $G$. For any two graphs $G, G^\prime \in S\sing(\bh, \bh^\prime)$, we 
say $G$ and $G^\prime$ to be isomorphic if one becomes the other by permuting $\lab(G)$ to 
$\lab(G^\prime)$ without permuting the indices $\bh$ 
and $\bh^\prime$. Denote $C_\eqv(\bh, \bh^\prime)$ to be the number of different 
isomorphism class and $N_\eqv(\bh, \bh^\prime)$ to be the maximum of size of any 
isomorphism class.  Now, we show the following three claims:
\begin{enumerate}
\item\label{item:claim-one} 
	For some constant $C_k$ depending solely on $k$, $C_\eqv(\bh, \bh^\prime) \le C_k$.
\item\label{item:claim-two} 
	$N_\eqv(\bh, \bh^\prime)\le (m+n)^{(\chi_{\nav} - \chi_{nv})(\bh, \bh^\prime)}$
	and $(\chi_{\edge} + \chi_{\cc}) \le 16k$.
\item\label{item:claim-three}
	$\chi_{\nav}(\bh, \bh^\prime) \le (\chi_{\edge} + \chi_{\cc})(\bh, \bh^\prime)$.
\end{enumerate}
Clearly, the desired claim of the lemma follows by the above two claims since we have
by definition 
\begin{equation*}
N\sing(\bh, \bh^\prime) \le C_\eqv(\bh, \bh^\prime) N_\eqv(\bh, \bh^\prime).
\end{equation*}

Now, we prove the above claims. In fact, since $|\lab(G)| \le 8k$, the number of 
$C_\eqv(\bh, \bh^\prime)$ should be only dependent of $k$ and independent of $n$.
This gives claim~\ref{item:claim-one}. To show the first part of claim~\ref{item:claim-two},
we note that the size of the isomorphism class is bounded by 
the number of possible labeling that does not change the labels $\bh$ and 
$\bh^\prime$, which can be easily shown upper bounded by 
$(m+n)^{(\chi_{\nav} - \chi_{nv})(\bh, \bh^\prime)}$.
This proves the first part of claim~\ref{item:claim-two}. The second part of claim 
\ref{item:claim-two} follows by definition of the set $S\sing(\bh, \bh^\prime)$.
Finally, since for any graph $G$, its number of vertices is upper bounded by the sum 
of its number of edges and connected components. This gives the third 
claim \ref{item:claim-three}. 
\end{proof}

Following Lemma~\ref{lemma:basic-combinatorics}, our next lemma gives generic upper 
bounds onto $\chi_{\edge}(\bh, \bh^\prime)$ and $\chi_{\cc}(\bh, \bh^\prime)$. Call 
$\bl\in \Z^4$ odd, if say for some $j \in [4]$,  $|\{i \in [4]: l_i = l_j\}|$ is odd. 
\begin{lemma}
\label{lemma:basic-combinatorics-two}
For $\bh \in [m]^4$ and $\bh^\prime \in [n]^4$, we have
\begin{equation}
\label{eqn:trivial-bound-chi}
\chi_{\edge}(\bh, \bh^\prime) \le 4k+2~~\text{and}~~
\chi_{\cc}(\bh, \bh^\prime) \le \min \{|\bh|, |\bh^\prime|\}.
\end{equation}
In addition, 
\begin{enumerate}
\item if either $\bh$ or $\bh^\prime$ is odd, then
\begin{equation}
\label{eqn:special-chi-one}
	\chi_{\edge}(\bh, \bh^\prime) \le 4k+1.
\end{equation}
\item if either $|\bh| = 4$ or $|\bh^\prime| = 4$, then  
\begin{equation}
\label{eqn:special-chi-two}
	\chi_{\edge}(\bh, \bh^\prime) \le 4k.
\end{equation}
\end{enumerate}
\end{lemma}
\begin{proof}
The first part of Eq~\eqref{eqn:trivial-bound-chi} follows from the fact that any graph in the 
set $S\sing(\bh, \bh^\prime)$ has in total $8k+4$ edges with each edge appearing at least 
twice. The second part of Eq~\eqref{eqn:trivial-bound-chi} follows by the construction of 
graph in the set $S\sing(\bh, \bh^\prime)$. To prove Eq~\eqref{eqn:special-chi-one}, note 
that, when $\bh$ is odd, say $h_l$ ($l\in [4]$) appears odd times in 
$(h_1, h_2, h_3, h_4)$, it means that some edge connecting $h_l$ must appear at least 
$3$ times. Since each edge at least appears twice and there are in total $8k + 4$ edges, 
we get that the number of total edges in this case is upper bounded by $4k+1$, giving 
Eq~\eqref{eqn:special-chi-one}. To prove Eq~\eqref{eqn:special-chi-two}, we note that, 
when $|\bh| = 4$, then for each $l \in [4]$, there exists some edge connecting $h_l$ 
that appears at least three times. This shows that at least four different edges appear at 
least $3$ times. Since each edge at least appears twice and there are in total $8k + 4$ 
edges, we get that the number of total edges in this case is upper bounded by $4k$, 
giving Eq~\eqref{eqn:special-chi-two}.
\end{proof}

Lastly, we provide an estimate on $\chi_{\nav}(\bh, \bh^\prime)$ when 
$\bh \in \type_{(5, m)}$ and $\bh^\prime \in \type_{(5, n)}$. 
\begin{lemma}
\label{lemma:special-treatment}
For $\bh \in \type_{(5, m)}$, $\bh^\prime \in \type_{(5, n)}$, we have
\begin{equation}
\chi_{\nav}(\bh, \bh^\prime) \le 4k + 2.
\end{equation}
\end{lemma}
\begin{proof}
In fact, any graph $G \in S\sing(\bh, \bh^\prime)$ has at most 
$4k+2$ distinct edges. Based on this, we divide our discussion into two cases. 

\begin{enumerate}
\item The graph $G$ has exactly $4k+2$ distinct edges. Hence, each edge 
	has multiplicity $2$. Now consider the undirected graph $\tilde{G}$ 
	induced by $G$: it consists of the vertices and the
	$4k+2$ different edges of $G$. We claim that there exists a loop in 
 	$\tilde{G}$. This implies that the graph $\tilde{G}$ and hence graph 
	$G$ has at most $4k+2$ different vertices. 

\begin{figure}[H]
	\centering
	\begin{subfigure}{} 
		\includegraphics[width=0.8\textwidth]{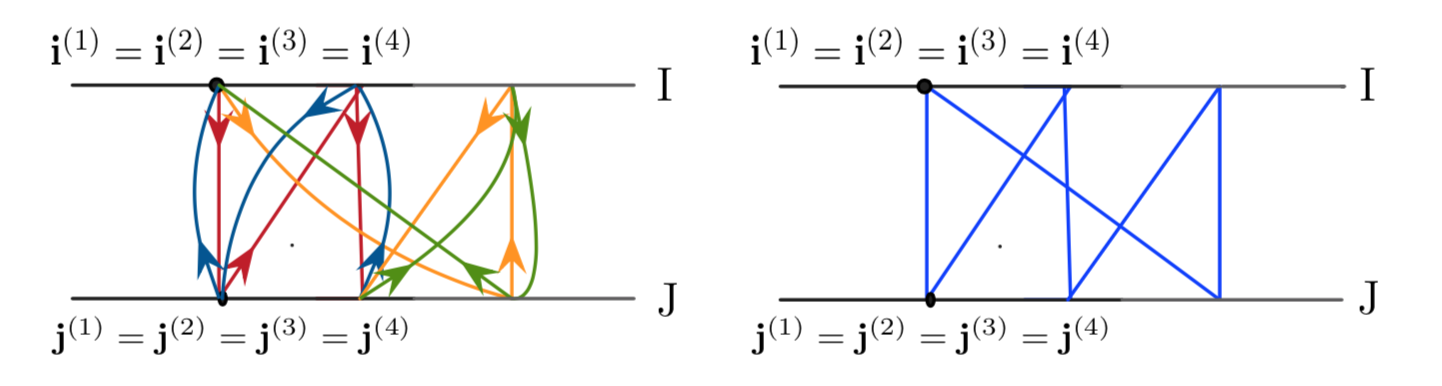}
	\end{subfigure}
	\vspace{1em}
	\caption{An illustration of $G$ and $\tilde{G}$ when the graph $G$ has 
		exactly $4k+2$ distinct edges. The left plot is the directed graph $G$ 
		and the right plot is the induced undirected graph $\tilde{G}$.
		Note that $\tilde{G}$ contains a loop. 
		}
\end{figure}
	
	To show that $\tilde{G}$ contains a loop, let $\bh = (h, h, h, h) \in [m]^4$ and
	$\bh^\prime = (h^\prime, h^\prime, h^\prime, h^\prime) \in [n]^4$ for 
	some $h\in [m]$ and $h^\prime \in [n]$. In fact, note that, by construction, 
	in graph $G$, the out-degree of vertex $h$ is at least $2$ and the in-degree of 
	vertex $h^\prime$ is at least $2$. Thus, there are two different path connecting 
	$h$ and $h^\prime$ in $\wtilde{G}$. This implies the existence of loop in the 
	graph $\tilde{G}$.
\item The graph $G$ has no more than $4k+1$ different edges. In this case, 
	since the graph is connected, the number of different vertices is at most 
	$4k+2$. 
\end{enumerate}
This concludes the proof. 
\end{proof}

Now, we summarize Lemma~\ref{lemma:basic-combinatorics}, Lemma
\ref{lemma:basic-combinatorics-two} and Lemma~\ref{lemma:special-treatment}
to upper bound $N\sing(l_1, l_2)$.
\begin{lemma}
\label{lemma:ultimate-bound-onto-N}
For some constant $C_k$, the bounds below hold for all $l_1, l_2 \in [4]$: 
\begin{equation}
N\sing(l_1, l_2) \le C_k n^{\chi(l_1, l_2)},
\end{equation}
where
\begin{equation}
\chi(l_1, l_2) = 
\begin{cases}
4k -4 &~~\text{if $\min\{l_1, l_2\} = 1$} \\
4k -2 &~~\text{if $\min\{l_1, l_2\} = 2$} \\
4k -1 &~~\text{if $\min\{l_1, l_2\} = 3$} \\
4k  &~~\text{if $\min\{l_1, l_2\} = 4$} \\
4k &~~\text{if $\min\{l_1, l_2\} = 5$}.
\end{cases}
\end{equation}
\end{lemma}
\begin{proof}
Define the following quantities
\begin{equation}
\chi_{\nav}(l_1, l_2) = \sup_{\bh \in \type_{(l_1, m)}, \bh^\prime \in \type_{(l_2, n)}}
	\chi_{\nav}(\bh, \bh^\prime)
\end{equation}
\begin{equation}
\chi_{\edge}(l_1, l_2) = \sup_{\bh \in \type_{(l_1, m)}, \bh^\prime \in \type_{(l_2, n)}}
	\chi_{\edge}(\bh, \bh^\prime)
\end{equation}
\begin{equation}
\chi_{\cc}(l_1, l_2) = \sup_{\bh \in \type_{(l_1, m)}, \bh^\prime \in \type_{(l_2, n)}}
	\chi_{\cc}(\bh, \bh^\prime)
\end{equation}
\begin{equation}
\chi_{\nv}(l_1, l_2) = \inf_{\bh \in \type_{(l_1, m)}, \bh^\prime \in \type_{(l_2, n)}}
	\chi_{\nv}(\bh, \bh^\prime)
\end{equation}
By Lemma~\ref{lemma:basic-combinatorics}, we know that, for all $l_1, l_2 \in [5]$,
\begin{equation}
\label{eqn:Ns-ell-bound}
N\sing(l_1, l_2) \le C_k n^{\chi_{\edge}(l_1, l_2) + \chi_{\cc}(l_1, l_2) - \chi_{\nv}(l_1, l_2)}.
\end{equation}
By Lemma~\ref{lemma:basic-combinatorics-two}, we can easily upper bound 
$\chi_{\edge}(l_1, l_2), \chi_{\cc}(l_1, l_2), \chi_{\nv}(l_1, l_2)$, and we enumerate 
those upper bound in the table below
(note by symmetry between $l_1$ and $l_2$, we only list the bounds for 
$l_1 \le l_2$)
\begin{equation}
   \begin{matrix}
    \hline
    (l_1, l_2)  & \chi_{\edge}(l_1, l_2) & \chi_{\cc}(l_1, l_2) & \chi_{\nv}(l_1, l_2) \\ \hline
    ({1}, {1}) &  \leq 4k  & \leq 4 & 8  \\
    
    ({1}, {2}) &  \leq 4k  & \leq 3 & 7   \\
    ({2}, {2}) &  \leq 4k+1  & \leq 3 & 6  \\
    
     ({1}, {3}) &  \leq 4k  & \leq 2 & 6  \\         
     ({2}, {3}) &  \leq 4k+1  & \leq 2 & 5  \\
     ({3}, {3}) &  \leq 4k+1  & \leq 2 & 4  \\
     
     ({1}, {4}) &  \leq 4k  & \leq 2 & 6  \\
     ({2}, {4}) &  \leq 4k+1  & \leq 2 & 5   \\
     ({3}, {4}) &  \leq 4k+1  & \leq 2 & 4   \\
     ({4}, {4}) &  \leq 4k+2  & \leq 2 & 4  \\

     ({1}, {5}) & \leq 4k  & 1 & 5 \\
     ({2}, {5}) &  \leq 4k+1  & 1 & 4  \\
     ({3}, {5}) &  \leq 4k+1  & 1 & 3   \\
     ({4}, {5}) &  \leq 4k+2  & 1 & 3   \\  
     ({5}, {5}) & \leq 4k+2 & 1 & 2  \\
     \hline
  \end{matrix}
 \end{equation}
Now, plugging the above estimates into Eq~\eqref{eqn:Ns-ell-bound}, it is easy 
to check that, we have for $l_1, l_2 \in [5]$,
\begin{equation}
\label{eqn:crude-estimate-Ns}
N\sing(l_1, l_2) \le C_k n^{\chi^\prime(l_1, l_2)}.
\end{equation}
where 
\begin{equation}
\chi^\prime(l_1, l_2) = 
\begin{cases}
\chi(l_1, l_2) &~~\text{if $\min\{l_1, l_2\} \le 4$} \\
4k +1 &~~\text{if $\min\{l_1, l_2\} = 5$}.
\end{cases}
\end{equation}
Note that, we can improve the estimate in Eq~\eqref{eqn:crude-estimate-Ns} 
when $l_1 = l_2 = 5$. In fact, by Lemma~\ref{lemma:special-treatment}, we 
know that $\chi_{\nav}(5, 5) \le 4k + 2$. Hence, by Lemma~\ref{lemma:basic-combinatorics}
we get that,
\begin{equation}
\label{eqn:crude-estimate-Ns-two}
N\sing(5, 5) \le C_k n^{\chi_{\nav}(5, 5) - \chi_{\nv}(5, 5)} \le C_k n^{4k}. 
\end{equation}
Now, the desired claim of the lemma follows by Eq~\eqref{eqn:crude-estimate-Ns}
and Eq~\eqref{eqn:crude-estimate-Ns-two}.
\end{proof}

\paragraph{Summary} Now, back to Eq~\eqref{eqn:main-moment-method-new}. 
By Lemma~\ref{lemma:ultimate-bound-onto-M} and Lemma~\ref{lemma:ultimate-bound-onto-N},
we can check easily that, there exists some constant $C_k > 0$ such that for all $l_1, l_2 \in [5]$, 
\begin{equation}
M_m(l_1) M_n(l_2) N\sing(l_1, l_2) \le C_k n^{-2}. 
\end{equation}
Substituting the above bound into Eq~\eqref{eqn:main-moment-method-new} concludes the 
proof of Point~\eqref{item:point-three}.


\section{Proof of Theorem~\ref{thm:subspace-estimation}}
\indent\indent
By Theorem~\ref{proposition:upper-bound-proposition}, 
$\what{\bX}\org(\bY) = \what{f}_{Y, \eps}(\bY)$ has the decomposition below: 
\begin{equation}
\what{\bX}\org(\bY) = \Info_W \bX + \sqrt{\Info_W}\bZ + \bDelta, 
\end{equation}
where $\bDelta$ is a random matrix satisfying 
\begin{equation}
\label{eqn:as-zero-Delta}
\lim_{\eps \to 0} \lim_{n \to \infty, m/n\to \gamma} \frac{1}{(mn)^{1/4}}
	\opnorm{\bDelta} = 0
\end{equation}
and $\bZ$ is some random matrix, whose entries are i.i.d bounded with mean $0$
variance $1$ and moreover, for some constants $\eps_0, C > 0$ independent of 
$m, n$, we have $\normmax{\bZ} \le C\eps^{-1}$ for all $\eps \le \eps_0$. For 
notational simplicity, in the rest of the proof, denote $\wtilde{\bX}\org(\bY) \in 
\R^{m\times n}$ and its SVD decomposition
\begin{equation}
\wtilde{\bX}\org(\bY) = \bX + \frac{1}{\sqrt{\Info_W}}\bZ
~~\text{and}~~
\wtilde{\bX}\org(\bY) = (mn)^{1/4} \wtilde{\bU} \wtilde{\bSigma}\org \wtilde{\bV}^{\sT}.
\end{equation}
Let $\wtilde{\bU}_l = (\wtilde{\bu}_1, \ldots, \wtilde{\bu}_l)^{\sT} \in 
\R^{m \times l}$ be the matrix consisting of the top $l$ left singular 
vectors of $\wtilde{\bX}\org(\bY)$.


We divide our proof into two cases. In the first case, we assume additionally that
the top singular values $\{\sigma_i\}_{i \in [l]}$ are pairwise different, i.e., 
\begin{equation}
\sigma_1 > \sigma_2 > \ldots > \sigma_l.
\end{equation}
Fix $\eps > 0$. By Theorem~\ref{theorem:general-deform}, we can without loss 
of generality (by flipping the sign of $\{\wtilde{\bu}_i\}$ if necessary) assume that, 
for any $i, j\in [l]$, 
\begin{equation}
\lim_{n \to \infty} \wtilde{\bu}_{i}^{\sT} \bu_j \asto 
	\begin{cases}
		G(\sigma_i ; \Info_W)~~&\text{if $i = j$}\\
		0~~&\text{if $i \neq j$}.
	\end{cases}
\end{equation}
As a consequence, we have, 
\begin{equation*}
\opnorm{\wtilde{\bU}_l^{\sT}\bU_{l} - \diag(G(\sigma_i ; \Info_W))_{i \in [l]}}\asto 0,
\end{equation*}
which, by Weyl's inequality, implies that, 
\begin{equation}
\label{eqn:key-without-perturb-eigenspace}
\sigma _{\min }\left( \wtilde{\bU}_l^{\sT}\bU_{l}\right)
	\xrightarrow{a.s.} G(\sigma _{l}; \Info_W).
\end{equation}
To pass the result in Eq~\eqref{eqn:key-without-perturb-eigenspace} to 
$\sigma _{\min }\big(\what{\bU}_{l}^{\sT}\bU_{l}\big)$, our idea is to view 
$\what{\bX}\org(\bY)$ as a perturbed version of $\wtilde{\bX}\org(\bY)$ and then
do some perturbation analysis to show that 
\begin{equation}
\sigma _{\min }\big(\what{\bU}_{l}^{\sT}\bU_{l}\big) \approx 
	\sigma _{\min }\left( \wtilde{\bU}_l^{\sT}\bU_{l}\right).
\end{equation}
More precisely, first, we note that
\begin{equation*}
\sigma _{\min }\left(\what{\bU}_l^{\sT} \bU_{l}\right)^2  = 
\sigma _{\min }\left( \bU_{l}^{\sT} \what{\bU}_l \what{\bU}_l^{\sT} \bU_{l}\right)
~~\text{and}~~
\sigma _{\min }\left(\wtilde{\bU}_l^{\sT} \bU_{l}\right)^2 = 
\sigma _{\min }\left( \bU_{l}^{\sT} \wtilde{\bU}_l \wtilde{\bU}_l^{\sT} \bU_{l}\right).
\end{equation*}
Then using Weyl's inequality and noting that $\bU_l$ is unitary, we get that 
\begin{equation}
\label{eqn:important-perturbation-bound}
\sigma _{\min }\left(\what{\bU}_l^{\sT} \bU_{l}\right)^2 
	\in \left[\sigma _{\min }\left(\wtilde{\bU}_l^{\sT} \bU_{l}\right)^2  \pm 
		\opnorm{\wtilde{\bU}_l\wtilde{\bU}_l^{\sT} - \widehat{\bU}_{l}\widehat{\bU}_{l}^{\sT}}\right],
\end{equation}
Now, by assumption $\sigma_l > \sigma_{l+1}$ and $\sigma_l > \Info_W^{-1/2}$. 
Thus Theorem~\ref{theorem:general-deform} (or Lemma~\ref{lemma:singular-value-of-bA}) 
implies for some constant $\vartheta_0 > 0$ (independent of $\eps, \delta$),  
\begin{equation}
\label{eqn:d-k-signal}
\lim_{n \to \infty, m/n\to \gamma}
	\frac{1}{(mn)^{1/4}}\left(\sigma_{l}( \wtilde{\bX}\org(\bY)) -  \sigma_{l+1}( \wtilde{\bX}\org(\bY))\right)
		\ge \vartheta_0.
\end{equation}
Fix this $\vartheta_0$. 
Next, by Eq~\eqref{eqn:as-zero-Delta}, we know for any $\bar{\Delta} > 0$, 
there exists some $\eps_0$ such that for $\eps \le \eps_0$
\begin{equation}
\label{eqn:d-k-noise}
\lim_{n \to \infty, m/n\to \gamma}
	\frac{1}{(mn)^{1/4}}\opnorm{\bDelta} 	\le \bar{\Delta}.
\end{equation}
According to Eq~\eqref{eqn:d-k-signal} and Eq~\eqref{eqn:d-k-noise}, 
we apply the Davis-Kahan Theorem (see lemma~\ref{lemma:Davis-Kahan})
to see that for any $\bar{\Delta}$ satisfying $\bar{\Delta} < \vartheta_0/2$, 
there exists some $\eps_0$ such that for all $\eps < \eps_0$,
\begin{equation}
\lim_{n\to \infty, m/n \to \gamma}
	\opnorm{\wtilde{\bU}_l\wtilde{\bU}_l^{\sT} - \widehat{\bU}_{l}\widehat{\bU}_{l}^{\sT}}	
		\le \frac{\bar{\Delta}}{\vartheta_0 - 2\bar{\Delta}}.
\end{equation} 
Since $\bar{\Delta} < \vartheta_0/2$ is arbitrary, we take $\bar{\Delta} \to 0$ on both 
sides and get that 
\begin{equation}
\label{eqn:limit-eps-eigenspace}
\lim_{\eps \to 0}\lim_{n\to \infty, m/n \to \gamma}
	\opnorm{\wtilde{\bU}_l\wtilde{\bU}_l^{\sT} - \widehat{\bU}_{l}\widehat{\bU}_{l}^{\sT}}	
		= 0.
\end{equation}
By Eq~\eqref{eqn:key-without-perturb-eigenspace},
Eq~\eqref{eqn:important-perturbation-bound} and 
Eq~\eqref{eqn:limit-eps-eigenspace},
we see that, 
\begin{equation}
\sigma _{\min }\left( \what{\bU}_l^{\sT}\bU_{l}\right)
	\asto G(\sigma _{l}; \Info_W), 
\end{equation}
giving the desired result of the Theorem. 

In the second case, we consider the situation where some elements of 
$\{\sigma_i\}_{i\in [k]}$ coincide. The idea is to reduce the second case 
to the first case, through the \emph{perturbation trick} that we shall describe. 
Indeed, for any $\{\pert_i\}_{i \in [l]}$ such that $\{\sigma_i + \pert_i\}_{i \in [l]}$ 
are distinct, we define
\begin{equation}
\bX(\pert) = \bX + (mn)^{1/4}\sum_{i=1}^r \pert_i \, \bu_i \bv_i^{\sT}.
\end{equation}
Now, for such $\{\pert_i\}_{i \in [l]}$, denote $\what{\bX}(\bY; \pert)$ and its 
SVD decomposition 
\begin{equation}
\what{\bX}\org(\bY; \pert)  = \Info_W \bX(\pert) + \sqrt{\Info_W}\bZ + \bDelta
~~\text{and}~~
\what{\bX}\org(\bY; \pert) = (mn)^{1/4} \what{\bU}(\pert)\what{\bSigma}(\pert)\what{\bV}(\pert)^{\sT}.
\end{equation}
Denote analogously $\widehat{\bU}_l(\pert)$ to be the matrix consisting of the top $l$ singular 
vectors of $\what{\bX}\org(\bY; \pert)$ and$\sigma_i(\pert)$ to be the top $i$th singular 
value of $\bX(\pert)$ for each $i$. Let $\pert\imax = \max_{i \in [l]}|\pert_i|$. Since by assumption
$\sigma_l > \sigma_{l+1}$, Weyl's Theorem implies when $\pert\imax$ is small enough, then 
the set of the top $l$ singular values of  $\bX(\pert)$ is precisely the set $\{\sigma_i + \pert_i\}_{i \in [l]}$. 
Moreover, the top $l$ singular values of $\bX(\pert)$ are pairwise different by our choice of 
$\{\pert_i\}_{i \in [l]}$. Thus, we may use the established result in the first case to conclude 
that, 
\begin{equation}
\label{eqn:perturb-result}
\lim_{\eps\to 0}\lim_{n\to \infty, m/n \to \gamma}
	\sigma _{\min }\left( \what{\bU}_l^{\sT}(\pert) \bU_{l}\right) = 
		G(\sigma_l(\pert); \Info_W).
\end{equation}
Now that, by a similar argument proving Eq~\eqref{eqn:important-perturbation-bound}
we can show that, 
\begin{equation}
\label{eqn:important-perturbation-bound-pert}
\sigma _{\min }\left(\what{\bU}_l^{\sT}(\pert) \bU_{l}\right)^2 
	\in \left[\sigma _{\min }\left(\what{\bU}_l^{\sT} \bU_{l}\right)^2  \pm 
		\opnorm{\what{\bU}_l(\pert)\what{\bU}_l^{\sT}(\pert)- 
			\widehat{\bU}_{l}\widehat{\bU}_{l}^{\sT}}\right],
\end{equation}
Moreover, Eq~\eqref{eqn:as-zero-Delta}, Eq~\eqref{eqn:d-k-signal} and Weyl's 
inequality, we have, for some constant $\eps_0, \vartheta_0 > 0$ 
(independent of $m, n$), we have for all $\eps \le \eps_0$, 
\begin{equation}
\label{eqn:d-k-signal-pert}
\lim_{n \to \infty, m/n\to \gamma}
	\frac{1}{(mn)^{1/4}}\left(\sigma_{l}(\what{\bX}\org(\bY)) -  
		\sigma_{l+1}(\what{\bX}\org(\bY)\right)
		\ge\delta_0^\prime.
\end{equation}
Now, viewing $\what{\bX}\org(\bY; \pert)$ as a perturbed version of
$\what{\bX}\org(\bY)$, Eq~\eqref{eqn:d-k-signal-pert} and Davis-Kahan-Theorem 
imply that for all $\pert$ such that $\pert\imax < \vartheta_0$, we have for 
$\eps \le \eps_0$, 
\begin{equation}
\lim_{n\to \infty, m/n \to \gamma}
	\opnorm{\what{\bU}_l(\pert)\what{\bU}_l^{\sT}(\pert) - \widehat{\bU}_{l}\widehat{\bU}_{l}^{\sT}}	
		\le \frac{\pert\imax }{\vartheta_0 - 2\pert\imax }.
\end{equation} 
By letting $\eps \to 0$ and $\pert\imax \to 0$, we get
\begin{equation}
\label{eqn:limit-eps-eigenspace-pert}
\lim_{\pert \to 0}\lim_{\eps \to 0}\lim_{n\to \infty, m/n \to \gamma}
	\opnorm{\what{\bU}_l(\pert)\what{\bU}_l^{\sT}(\pert) - \widehat{\bU}_{l}\widehat{\bU}_{l}^{\sT}}	
		= 0.
\end{equation}
By Eq~\eqref{eqn:perturb-result}, Eq~\eqref{eqn:important-perturbation-bound-pert} 
and Eq~\eqref{eqn:limit-eps-eigenspace-pert}, we know that, 
\begin{align}
\lim_{\eps\to 0}\lim_{n\to \infty, m/n \to \gamma}
	\sigma _{\min }\left( \what{\bU}_l^{\sT} \bU_{l}\right) 
&= \lim_{\pert \to 0}\lim_{\eps\to 0}\lim_{n\to \infty, m/n \to \gamma}
	\sigma _{\min }\left( \what{\bU}_l^{\sT}(\pert) \bU_{l}\right)  \nonumber \\
&= \lim_{\pert \to 0} G(\sigma_l(\pert); \Info_W)
=  G(\sigma_l; \Info_W) 
\end{align}
where the last identity since the function $\sigma \to G(\sigma; \Info_W)$ is continuous 
on $[\Info_W^{-1/2}, \infty)$. 
This concludes the proof of Theorem~\ref{thm:subspace-estimation}.


\newcommand{\bp}{\boldsymbol{p}}
\newcommand{\bq}{\boldsymbol{q}}
\newcommand{\sym}{^{\rm sym}}
\section{Proof of Theorem~\ref{theorem:nonlinear-PCA-perturbation}}
\label{sec:matrix-perturbation}
Define $\{\bp_1, \bp_2, \ldots, \bp_k\}$ and $\{\bq_1, \bq_2, \ldots, \bq_k\}$ by the columns of 
$\bP_k$, $\bQ_k$ respectively, i.e.,  
\begin{equation}
\bP_k = [\bp_1, \bp_2, \ldots, \bp_k]~\text{and}~\bQ_k = [\bq_1, \bq_2, \ldots, \bq_k].  
\end{equation}
For each $l < k$, we define the matrices $\bP_l \in \R^{m\times l}$ and $\bQ_l \in \R^{n\times l}$ by
\begin{equation}
\bP_l = [\bp_1, \bp_2, \ldots, \bp_l]~\text{and}~\bQ_l = [\bq_1, \bq_2, \ldots, \bq_l].  
\end{equation}
Similarly, we define the vectors $\{\wtilde{\bp}_l\}_{l \in [k]}$, 
$\{\wtilde{\bq}_l\}_{l \in [k]}$ and the matrices 
$\wtilde{\bP}_l \in \R^{m\times l}$, $\wtilde{\bQ}_l \in \R^{n\times l}$.
For each $l \le k$, we define the error matrices $\bDelta_l \in \R^{m \times n}$ and singular 
gap $\delta_l \in \R_+$ by
\begin{equation}
\bDelta_l = \bP_l\bQ_l^{\sT} - \wtilde{\bP}_l\wtilde{\bQ}_l^{\sT},
~\text{and}~
\delta_l = \sigma_l(\bA) - \sigma_{l+1}(\bA)
\end{equation}
By assumption in Eq~\eqref{eqn:singular-gap}, we know that $\delta_k > \vartheta > 2\opnorm{\bE}$. 

As our starting point, we prove the following claim. For each $l \in [k]$ such that 
$\delta_l > 2\opnorm{\bE}$, 
\begin{equation}
\label{eqn:perturb-U-V-prod}
\opnorm{\bDelta_l}
	\le \frac{2}{\delta_l} \opnorm{\bE}.
\end{equation}
To do so, define the matrices $\bA\sym, \wtilde{\bA}\sym, \bE\sym  \in \R^{(m+n) \times (m+n)}$ by 
\begin{equation}
\bA\sym \defeq 
    \begin{bmatrix}
    0 & \bA \\
    \bA^{\sT} & 0 \\
    \end{bmatrix}
,~~
\wtilde{\bA}\sym \defeq 
    \begin{bmatrix}
    0 & \wtilde{\bA} \\
    \wtilde{\bA}^{\sT} & 0 \\
    \end{bmatrix}
~\text{and}~
\bE\sym \defeq 
    \begin{bmatrix}
    0 & \bE \\
    \bE^{\sT} & 0 \\
    \end{bmatrix}
\end{equation}
Now that since $\wtilde{\bA} = \bA+ \bE$, we know that $\wtilde{\bA}\sym = 
\bA\sym + \bE\sym$. By standard result in matrix analysis~\cite[Theorem 4.2]{StewartSun90}, 
we know that the top $k+1$ eigenvalues of $\bA\sym$ are 
precisely the top $k+1$ singular values of $\bA$, 
and moreover, for any $l \in [k]$, the top $l$ eigenvectors of $\bA\sym$ and 
$\wtilde{\bA}\sym$ are columns of $\bP\sym, \wtilde{\bP}\sym \in \R^{(m+n) \times k}$ 
defined below
\begin{equation}
\bP_l\sym \defeq 
	\begin{bmatrix}
		\bP_l \\
		\bQ_l.
	\end{bmatrix}
~~\text{and}~~
\wtilde{\bP}_l\sym \defeq 
	\begin{bmatrix}
		\wtilde{\bP}_l \\
		\wtilde{\bQ}_l.
	\end{bmatrix}
\end{equation}
Now, we fix any $l$ such that $\delta_l > 2\opnorm{\bE}$.
By Davis-Kahan Theorem (see Lemma~\ref{lemma:Davis-Kahan}), we have
\begin{equation}
\opnorm{\wtilde{\bP}_l\sym(\wtilde{\bP}_l\sym)^{\sT} - \bP_l\sym(\bP_l\sym)^{\sT}}
	\le \frac{\opnorm{\bE\sym}}{\delta_l - \opnorm{\bE\sym}}, 
\end{equation}
which is equivalent to 
\begin{equation}
\opnorm{\begin{bmatrix}
\bP_l\bP_l^{\sT} - \wtilde{\bP}_l\wtilde{\bP}_l^{\sT}  & \bP_l \bQ_l^{\sT} - \wtilde{\bP}_l \wtilde{\bQ}_l^{\sT} \\
\bQ_l\bP_l^{\sT} - \wtilde{\bQ}_l\wtilde{\bP}_l^{\sT} & \bQ_l\bQ_l^{\sT} - \wtilde{\bQ}_l \wtilde{\bQ}_l^{\sT} \\
\end{bmatrix}}
\le \frac{\opnorm{\bE}}{\delta_l - \opnorm{\bE}}.
\end{equation}
Now that the bound above implies 
\begin{equation}
\label{eqn:perturb-U-V-prod-final}
\opnorm{\bDelta_l} \le \frac{\opnorm{\bE}}{\delta_l - \opnorm{\bE}}.
\end{equation}
Since $\delta_l > 2\opnorm{\bE}$, Eq~\eqref{eqn:perturb-U-V-prod-final} implies the desired claim at 
Eq~\eqref{eqn:perturb-U-V-prod}.

Now, we are ready to show the desired claim of Theorem~\ref{theorem:nonlinear-PCA-perturbation}.
Define the auxiliary matrix 
\begin{equation}
\bB_k = \wtilde{\bP}_k f(\bS_k) \wtilde{\bQ}_k^{\sT}.
\end{equation}
By triangle inequality, we have, 
\begin{equation}
\label{eqn:def-error-one-two}
\opnorm{f(\bA_k) - f(\wtilde{\bA}_k)}
	\le \opnorm{\bE_1} + \opnorm{\bE_2}~
		~\text{where}~ 
	\bE_1 ={f(\wtilde{\bA}_k) - \bB_k} 
	~\text{and}~
	\bE_2 =  {f(\bA_k) - \bB_k}.
\end{equation}
We bound error matrix $\bE_1$ first. By definition, we have 
\begin{equation}
\label{eqn:error-term-one}
\opnorm{\bE_1} = \opnorm{f(\wtilde{\bA}_k) - \bB_k}
= \opnorm{\wtilde{\bP}_k (f(\bS_k) - f(\wtilde{\bS}_k)) \wtilde{\bQ}_k}
\le \opnorm{f(\bS_k) - f(\wtilde{\bS}_k)}.
\end{equation}
Note that $\opnormbig{\bS_k - \wtilde{\bS}_k} \le \opnorm{\bE}$ by Weyl's inequality. Thus,
the assumption that $f$ is $(L, \alpha)$ H\"{o}lder continuous on $[\tau, \zeta]$ 
for $\zeta > \sigma_1(\bA)$ implies that 
\begin{equation}
\opnorm{f(\bS_k) - f(\wtilde{\bS}_k)} \le L \opnormbig{\bS_k - \wtilde{\bS}_k}^{\alpha}
	\le L \opnorm{\bE}^{\alpha}. 
\end{equation}
Substituting the above estimate into Eq~\eqref{eqn:error-term-one} gives the upper bound  
\begin{equation}
\label{eqn:error-term-one-final}
\opnorm{\bE_1} \le L \opnorm{\bE}^{\alpha}.
\end{equation}
Next, we bound the error matrix $\bE_2$. Indeed, by definition, we have, 
\begin{align}
f(\bA_k) - \bB_k &= \sum_{l \in [k]} f(\sigma_l(\bA)) 
	\left(\bp_l \bq_l^{\sT} -  \wtilde{\bp}_l \wtilde{\bq}_l^{\sT}\right) \nonumber \\
	&= \sum_{l < k} (f(\sigma_l(\bA)) - f(\sigma_{l+1}(\bA))) 
		\left(\bP_l \bQ_l^{\sT} - \wtilde{\bP}_l \wtilde{\bQ}_l^{\sT}\right)
		+ f(\sigma_k(\bA)) \left(\bP_k \bQ_k^{\sT} - \wtilde{\bP}_k \wtilde{\bQ}_k^{\sT}\right). \nonumber
\end{align}
Hence, by triangle inequality, we get the estimate below, 
\begin{align}
\label{eqn:error-term-two}
\opnorm{\bE_2} \le \sum_{l < k} \Big| f(\sigma_l(\bA)) - f(\sigma_{l+1}(\bA))\Big|
		\opnormBig{\bP_l \bQ_l^{\sT} - \wtilde{\bP}_l \wtilde{\bQ}_l^{\sT}}
		+ f(\sigma_k(\bA)) \opnormBig{\bP_k \bQ_k^{\sT} - \wtilde{\bP}_k \wtilde{\bQ}_k^{\sT}}. 
\end{align}
Now that $\delta_k  = \sigma_k(\bA) - \sigma_{k+1}(\bA) > \vartheta >  2\opnorm{\bE}$
by assumption. Eq~\eqref{eqn:perturb-U-V-prod} shows that, 
\begin{equation}
\label{eqn:trivial-claim-matrix-pert}
\opnormBig{\bP_k \bQ_k^{\sT} - \wtilde{\bP}_k \wtilde{\bQ}_k^{\sT}}
	\le \frac{2}{\delta_k} \opnorm{\bE} \le \frac{2}{\vartheta} \opnorm{\bE} 
\end{equation}
Now, we show that for any $l < k$, 
\begin{equation}
\label{eqn:crucial-claim-matrix-pert}
\Big| f(\sigma_l(\bA)) - f(\sigma_{l+1}(\bA))\Big|
		\opnormBig{\bP_l \bQ_l^{\sT} - \wtilde{\bP}_l \wtilde{\bQ}_l^{\sT}}
		\le  4L \opnorm{\bE}^{\alpha}.
\end{equation}
To show Eq~\eqref{eqn:crucial-claim-matrix-pert}, we divide it into two cases.
\begin{enumerate}
\item In the first case, we assume that $\delta_l = \sigma_l(\bA) - \sigma_{l+1}(\bA) \le 2\opnorm{\bE}$. 
	Since $f$ is $(L, \alpha)$ H\"{o}lder continuous (see Eq~\eqref{eqn:f-Holder-continuous}) 
	for $\alpha \in (0, 1]$, we know that, 
	\begin{equation}
	\label{eqn:crucial-claim-matrix-pert-case-one-one}
	\Big| f(\sigma_l(\bA)) - f(\sigma_{l+1}(\bA))\Big|\le L \left(2\opnorm{\bE}\right)^{\alpha}	
		\le 2L \opnorm{\bE}^{\alpha}.
	\end{equation}
	Since $\bP_l$, $\wtilde{\bP}_l$, $\bQ_l$ and $\wtilde{\bQ}_l$ are all orthonormal, 
	we know that, 
	\begin{equation}
	\label{eqn:crucial-claim-matrix-pert-case-one-two}
	\opnormBig{\bP_l \bQ_l^{\sT} - \wtilde{\bP}_l \wtilde{\bQ}_l^{\sT}}
		\le \opnormBig{\bP_l \bQ_l^{\sT}} + \opnormBig{\wtilde{\bP}_l \wtilde{\bQ}_l^{\sT}}
		\le 2. 
	\end{equation}
	Now the desired claim at Eq~\eqref{eqn:crucial-claim-matrix-pert} follows by 
	Eq~\eqref{eqn:crucial-claim-matrix-pert-case-one-one}
	and 
	Eq~\eqref{eqn:crucial-claim-matrix-pert-case-one-two}.
\item In the second case, we assume that $\delta_l = \sigma_l(\bA) - \sigma_{l+1}(\bA) > 2\opnorm{\bE}$. 
	Since $f$ is $(L, \alpha)$ H\"{o}lder continuous (see Eq~\eqref{eqn:f-Holder-continuous}), 
	we know that, 
	\begin{equation}
	\label{eqn:crucial-claim-matrix-pert-case-two-one}
	\Big| f(\sigma_l(\bA)) - f(\sigma_{l+1}(\bA))\Big|\le L \left(\sigma_l(\bA) - \sigma_{l+1}(\bA)\right)^{\alpha}	
		\le L \delta_l^{\alpha}.
	\end{equation}
	Since $\delta_l > 2\opnorm{\bE}$ and $\alpha \in (0, 1]$, Eq~\eqref{eqn:perturb-U-V-prod} shows that, 
	\begin{equation}
	\label{eqn:crucial-claim-matrix-pert-case-two-two}
	\opnormBig{\bP_l \bQ_l^{\sT} - \wtilde{\bP}_l \wtilde{\bQ}_l^{\sT}}
		\le \frac{2}{\delta_l}\opnorm{\bE} \le \frac{2}{\delta_l^{\alpha}}\opnorm{\bE}^{\alpha}.
	\end{equation}
	Now the desired claim at Eq~\eqref{eqn:crucial-claim-matrix-pert} follows by 
	Eq~\eqref{eqn:crucial-claim-matrix-pert-case-two-one}
	and 
	Eq~\eqref{eqn:crucial-claim-matrix-pert-case-two-two}.
\end{enumerate}
Substituting the bound at Eq~\eqref{eqn:crucial-claim-matrix-pert} and 
Eq~\eqref{eqn:trivial-claim-matrix-pert} into Eq~\eqref{eqn:error-term-two}, we get that, 
\begin{equation}
\label{eqn:error-term-two-final}
\opnorm{\bE_2} \le 4(k-1)L \opnorm{\bE}^{\alpha} + \frac{2}{\vartheta}f(\sigma_k(\bA))\opnorm{\bE}. 
\end{equation}
Now the desired claim of Theorem~\ref{theorem:nonlinear-PCA-perturbation} follows by plugging 
Eq~\eqref{eqn:error-term-one-final} and Eq~\eqref{eqn:error-term-two-final}
into Eq~\eqref{eqn:def-error-one-two}.


\section{Proofs of Technical Lemma}
\subsection{Proof of Lemma~\ref{lemma:opnorm-expectation-nonsymmetric}}
\label{sec:proof-opnorm-expectation-nonasymmetric}
We start by proving the case where $\bX$ is symmetric (in this case $m=n$). In this case, 
we define matrices $\{\bX^{(i, j)}\}_{1\le i \le j \le n}$ such that, 
\begin{equation*}
\bX_{k, l}^{(i, j)} = \begin{cases}
X_{i, j}~~&\text{if $(k, l) = (i, j)$ or $(k, l) = (j, i)$} \\
0~~&\text{otherwise}
\end{cases}.
\end{equation*}
Note that, $\{\bX^{(i, j)}\}_{1\le i\le j\le n}$ are all symmetric matrices, mean $\zero$, independent to 
each other and 
\begin{equation*}
\bX = \sum_{1 \le i\le j \le n} \bX^{(i, j)}.
\end{equation*}
Denote $\{\eps_{i, j}\}_{1\le i\le j\le n}$ be independent Radamacher random variables and 
$\wtilde{\bX^{(i, j)}} = \eps_{i, j} \bX^{(i, j)}$. By standard symmetrization argument, we have 
\begin{equation}
\label{eqn:opnorm-symmetrization-step}
\E \opnorm{\bX}^2 = \E \left\{\opnormbigg{\sum_{1 \le i\le j \le n} \bX^{(i, j)}}^2 \right\}
	\le 4\E \left\{\opnormbigg{\sum_{1\le i\le j\le n}\wtilde{\bX^{(i, j)}}}^2 \right\}
\end{equation}
Now, $\left\{\wtilde{\bX^{(i, j)}}\right\}_{1\le i\le j\le n}$ are symmetrically distributed symmetric matrices.
We may use Lemma~\ref{lemma:opnorm-expectation-chen} to get for some numerical constant 
$C > 0$, 
\begin{align}
&\E\left\{\opnormbigg{\sum_{1\le i\le j\le n}\wtilde{\bX^{(i, j)}}}^k \right\} \nonumber \\ 
&\le C^k \left[(\log n + k)^{1/2} \times 
	\opnormbigg{\E \bigg\{\sum_{1\le i\le j\le n} \wtilde{\bX^{(i, j)}}^2\bigg\}} + (\log n + k) \times 
	\E \left\{\max_{1\le i\le j\le n}\opnormbigg{\wtilde{\bX^{(i, j)}}}\right\}\right]^k \nonumber \\
&\le C^k k^k \log^k(n)  \left(\max_{i\in [n]} \E \Big[\sum_{j=1}^n X_{i, j}^2\Big]
	+ \E \Big[\max_{1\le i\le j\le n}\left|X_{i, j}\right|\Big]\right)^k 
\label{eqn:computation-of-chen-inequality}
\end{align}
This proves the result for symmetric matrix $\bX \in \R^{n \times n}$. The more general situation 
where $\bX \in \R^{m \times n}$ is asymmetric can be reduced to the symmetric case. 
In fact, define $\bX\sym \in \R^{(m + n) \times (m+n)}$
\begin{equation}
\bX\sym \defeq \begin{bmatrix}
0  & \bX \\
\bX^{\sT} &0
\end{bmatrix}.
\end{equation}
Then $\bX\sym$ is symmetric and satisfies $\opnorm{\bX\sym} = \opnorm{\bX}$. Now, by applying 
the already established result to the symmetric matrix $\bX\sym \in \R^{(m + n) \times (m+n)}$, we 
get the desired claim of the lemma for the asymmetric matrix $\bX$.

\subsection{Proof of Lemma~\ref{lemma:elementary-inequality}}
\label{sec:proof-lemma-elementary-inequality}
First, we have, 
\begin{equation*}
g(s_1, t_1) - g(s_2, t_2) = \frac{s_2 - s_1}{t_1 +\dezero}
	+ \frac{s_2(t_1 - t_2)}{(t_1 + \dezero) (t_2 + \dezero)}.
\end{equation*}
Since $t_1, t_2 \ge 0$, by triangle inequality, we get, 
\begin{equation*}
\left|g(s_1, t_1) - g(s_2, t_2)\right| \leq \dezero^{-1} 
	\left|s_1 - s_2\right| + \dezero^{-2} |s_2| |t_1 - t_2|. 
\end{equation*}
Similarly,  
\begin{equation*}
\left|g(s_1, t_1) - g(s_2, t_2)\right| \leq \dezero^{-1} 
	\left|s_1 - s_2\right| + \dezero^{-2} |s_1| |t_1 - t_2|. 
\end{equation*}
The last two inequalities together give the desired claim.

\subsection{Proof of Lemma~\ref{lemma:f-abs-bar-bound}}
\label{sec:proof-lemma-f-abs-bar-bound}
By Assumption~{\sf A2}, we know that for all $i \in [m], j\in [n]$,
\begin{equation}
\left|g^2\left(p_W^\prime(W_{i, j}), p_W(W_{i, j})\right)\right| 
	\le \eps^{-2} \norm{p_W^\prime(\cdot)}_\infty^2
	\le \eps^{-2} M_2^2. 
\end{equation}
Thus, since $\wtilde{\Info}_{W,\eps}$ is the average of the 
random variables $g^2\left(p_W^\prime(W_{i, j}), p_W(W_{i, j})\right)$, 
we know by Hoeffding's inequality \cite[Theorem 2.8]{BoucheronLuMa13}
that, for all $t > 0$, 
\begin{equation}
\label{eqn:Hoeffiding-tilde-I-W-eps}
\P \left(| \wtilde{\Info}_{W,\eps} - \E \wtilde{\Info}_{W,\eps} | > t\right) 
	\leq 2 \exp \left(- \frac{mn t^2 \eps^{4}}{2M_2^4}\right).
\end{equation}
Now, we show the below crucial estimate, 
\begin{equation}
\label{eqn:crucial-estimate-diff-bound}
\left|\E \wtilde{\Info}_{W,\eps} - \Info_{W, \eps}\right| \le \delta_{W, \eps}.
\end{equation}
Indeed, Eq~\eqref{eqn:crucial-estimate-diff-bound} follows by the direction computations below
\begin{align}
\left|\E \wtilde{\Info}_{W,\eps} - \Info_{W, \eps}\right|
&= \left|\int_{\R} \frac{(p_W^\prime(w))^2}{(p_W(w)+ \eps)^2} p_W(w)\de w
	- \int_{\R} \frac{(p_W^\prime(w))^2}{p_W(w)+ \eps} \de w\right|
= \eps \left|\int_{\R} \frac{(p_W^\prime(w))^2}{(p_W(w) + \eps)^2} \de w\right| \nonumber \\
& \le \eps \left|\int_{\R} \frac{(p_W^\prime(w))^2}{p_W(w)(p_W(w) + \eps)} \de w\right|
= \left|\int_{\R} \frac{(p_W^\prime(w))^2}{p_W(w)+ \eps} \de w - 
\int_{\R} \frac{(p_W^\prime(w))^2}{p_W(w)} \de w\right|
=  \delta_{W, \eps}. \nonumber
\end{align}
Hence, Eq~\eqref{eqn:Hoeffiding-tilde-I-W-eps} and Eq~\eqref{eqn:crucial-estimate-diff-bound}
together show that, for all $t > 0$, 
\begin{equation}
\P \left(| \wtilde{\Info}_{W,\eps} - I_{W, \eps}| > t + \delta_{W, \eps}\right)
\le  
\P \left(| \wtilde{\Info}_{W,\eps} - \E \wtilde{\Info}_{W,\eps} | > t\right) 
	\leq 2 \exp \left(- \frac{mn t^2 \eps^{4}}{2M_2^4}\right),
\end{equation}
giving the desired claim of the lemma. 

\subsection{Proof of Lemma~\ref{lemma:upper-bound-error-one}}
\label{sec:proof-upper-bound-error-one}
To simplify the notations, we introduce the quantities $\{G_{1, i, j}\}_{i \in [m], j\in [n]}$ to be,
\begin{equation*}
G_{1, i, j} \defeq \left|g(\hat{p}_Y^\prime(\wtilde{Y}_{i, j}), \hat{p}_Y(\wtilde{Y}_{i, j}))
		- g(p_{W}^\prime(W_{i, j}), p_W(W_{i, j}))\right|
\end{equation*}
and the quantities 
$\{G_{2, i, j}\}_{i \in [m], j\in [n]}$ to be,
\begin{equation*}
G_{2, i, j} \defeq \left|g^2(\hat{p}^\prime_Y(\wtilde{Y}_{i, j}), \hat{p}_Y(\wtilde{Y}_{i, j}))
		- g^2(p_{W}^\prime(W_{i, j}), p_W(W_{i, j}))\right|.
\end{equation*}
Moreover, we denote $G_1^2$ to be the mean of $\{G^2_{1, i, j}\}_{i \in [m], j\in [n]}$, i.e, 
\begin{equation}
G_1^2 = \frac{1}{mn} \sum_{i \in [m], j \in [n]} G_{1, i, j}^2. 
\end{equation}
As our starting point, we see by triangle inequality that,
\begin{equation}
\label{eqn:start-bound-on-error-one}
\errortwo = \frac{1}{mn} \left|\sum_{i \in [m], j\in [n]}
	g^2(\hat{p}_Y^\prime(\wtilde{Y}_{i, j}), \hat{p}_Y(\wtilde{Y}_{i, j}))
		- g^2(p_{W}^\prime(W_{i, j}), p_W(W_{i, j}))\right| 
	\le  \frac{1}{mn} \sum_{i \in [m], j\in [n]}G_{2, i, j}.
\end{equation}
We next upper bound the RHS of Eq~\eqref{eqn:start-bound-on-error-one}. 
Note the elementary inequality below,
\begin{equation*}
|s^2 - t^2| \le (s-t)^2 + 2|s-t| (|s| \wedge |t|)~~\text{for $s, t \in \R$}. 
\end{equation*} 
If we apply it to  $s = g(\hat{p}_Y^\prime(\wtilde{Y}_{i, j}), \hat{p}_Y(\wtilde{Y}_{i, j}))$ 
and $t = g(p_{W}^\prime(W_{i, j}), p_W(W_{i, j}))$, we get that, 
\begin{align}
G_{2, i, j} &\le 2 \left|g(p_{W}^\prime(W_{i, j}), p_W(W_{i, j}))\right|G_{1, i, j} + G_{1, i, j}^2,  
\label{eqn:intermediate-bound-on-error-one}
\end{align}
Therefore, if we plug the individual estimate of 
Eq~\eqref{eqn:intermediate-bound-on-error-one} into RHS of 
Eq~\eqref{eqn:start-bound-on-error-one}, we get that,  
\begin{align}
\errortwo \le 
	\frac{2}{mn} \sum_{i, j} \left|g(p_{W}^\prime(W_{i, j}), p_W(W_{i, j}))\right|G_{1, i, j} 
+ 	\frac{1}{mn} \sum_{i, j}G_{1, i, j}^2 \le 2 \wtilde{I}_{W, \eps}^{1/2} \times 
		G_1 + G_1^2,
\end{align}
where the last inequality follows by Cauchy-Schwartz inequality. Now, to 
show the desired claim of the lemma, it suffices to show for some constant 
$C > 0$ the bound below on $G_1$: 
\begin{equation}
\label{eqn:upper-bound-on-G_1}
G_1 \le C \left(\dezero^{-2}\resT_{n, 1} + \dezero^{-1}\resT_{n, 2}\right).
\end{equation}
Indeed, to show Eq~\eqref{eqn:upper-bound-on-G_1}, we first upper 
each term $G_{1, i, j}$. By Lemma~\ref{lemma:elementary-inequality}, 
\begin{align}
G_{1, i, j} &\le \dezero^{-1} \left|\hat{p}_Y^\prime(\wtilde{Y}_{i, j}) - p_W^\prime(W_{i, j})\right|
	+  \dezero^{-2}\norm{p_W^\prime(\cdot)}_\infty
		\left|\hat{p}_Y(\wtilde{Y}_{i, j}) - p_W(W_{i, j})\right|
	\label{eqn:intermediate-step-g-entry-bound-info} 
\end{align}
Note that $\norm{p_W^\prime(\cdot)}_\infty \le M_2$ by Assumption~{\sf A2}. Hence, 
by Jensen's inequality, 
\begin{equation*}
G_{1, i, j}^2 \le C\left[\dezero^{-2} \left(\hat{p}^\prime_Y(\wtilde{Y}_{i, j}) - p^\prime_W(W_{i, j})\right)^2
	+  \dezero^{-4} \left(\hat{p}_Y(\wtilde{Y}_{i, j}) - p_W(W_{i, j})\right)^2\right]
\end{equation*}
for some constant $C > 0$. Summing over all $i \in [m], j\in [n]$, we get that, 
\begin{equation}
\label{eqn:G-1-intermediate-estimate}
G_1^2 \le C \left[\frac{1}{\dezero^{2}mn} \sum_{i \in [m], j\in [n]} 
	\left(\hat{p}_Y^\prime(\wtilde{Y}_{i, j}) - p_W^\prime(W_{i, j})\right)^2
	+  \frac{1}{\dezero^{4}mn} \sum_{i \in [m], j\in [n]} 
		\left(\hat{p}_Y(\wtilde{Y}_{i, j}) - p_W(W_{i, j})\right)^2  \right].
\end{equation}
Now, we bound each of the two individual summation terms in the above parentheses. In fact,
we show that, for some constant $C > 0$,
\begin{equation}
\label{eqn:py-diff-pw}
\frac{1}{mn} \sum_{i \in [m], j \in [n]} \left(\hat{p}_Y(\wtilde{Y}_{i, j}) - p_W(W_{i, j})\right)^2
	\le C \left(\normmax{\hat{p}_Y(\wtilde{\bY}) - p_W(\wtilde{\bY})}^2 + 
		\frac{1}{mn} \norm{\bX}_F^2 + \wbar{W}^2\right), 
\end{equation}
and 
\begin{equation}
\label{eqn:py-diff-pw-prime}
\frac{1}{mn} \sum_{i \in [m], j \in [n]} \left(\hat{p}^\prime_Y(\wtilde{Y}_{i, j}) - p^\prime_W(W_{i, j})\right)^2
	\le C \left(\normmax{\hat{p}_Y^\prime(\wtilde{\bY}) - p_W^\prime(\wtilde{\bY})}^2 + 
		\frac{1}{mn} \norm{\bX}_F^2 + \wbar{W}^2\right).
\end{equation}
The proof of Eq~\eqref{eqn:py-diff-pw} and Eq~\eqref{eqn:py-diff-pw-prime} is similar, 
and we exemplify the proof by showing Eq~\eqref{eqn:py-diff-pw}. Indeed, by triangle 
inequality, we have, for $i \in [m], j\in [n]$, 
\begin{align*}
\left|\hat{p}_Y(\wtilde{Y}_{i, j}) - p_W(W_{i, j})\right| &\le 
	\left|\hat{p}_Y(\wtilde{Y}_{i, j}) - p_W(\wtilde{Y}_{i, j})\right| + \left|p_W(\wtilde{Y}_{i, j}) - p_W(W_{i, j})\right|.
\end{align*}
Now, since $\norm{p_W^\prime(\cdot)}_\infty \le M_2$ by Assumption~{\sf A2}, we have, 
\begin{align*}
\left|\hat{p}_Y(\wtilde{Y}_{i, j}) - p_W(W_{i, j})\right| \le 
	\normmax{\hat{p}_Y(\wtilde{\bY}) - p_W(\wtilde{\bY})} + M_2\left|\wtilde{Y}_{i, j} - W_{i, j}\right|.
\end{align*}
Now, we sum up the above inequality for $i \in [m], j\in [n]$. Jensen's inequality implies, 
\begin{equation*}
\frac{1}{mn} \sum_{i \in [m], j\in [n]} \left(\hat{p}_Y(\wtilde{Y}_{i, j}) - p_W(W_{i, j})\right)^2
	\le 2\left(\normmax{\hat{p}_Y(\wtilde{\bY}) - p_W(\wtilde{\bY})}^2 + \frac{M_2^2}{mn}
		\norm{\wtilde{\bY} - \bW}_F^2\right). 
\end{equation*}
The desired claim of Eq~\eqref{eqn:py-diff-pw} now follows by the above estimate and 
\begin{equation}
\frac{1}{mn}\norm{\wtilde{\bY} - \bW}_F^2 = \frac{1}{mn} \norm{\wtilde{\bX} - 
	\wbar{W} \one_m \one_n^{\sT}}_F^2 \le 2 \left(\wbar{W}^2 + \frac{1}{mn} \norm{\wtilde{\bX}}_F^2\right)
	\le 2 \left(\wbar{W}^2 + \frac{1}{mn} \norm{\bX}_F^2\right).
\end{equation}
Now, the bound of $G_1$ in Eq~\eqref{eqn:upper-bound-on-G_1} follows by plugging 
Eq~\eqref{eqn:py-diff-pw} and Eq~\eqref{eqn:py-diff-pw-prime} into 
Eq~\eqref{eqn:G-1-intermediate-estimate} (and note that the RHS of both 
Eq~\eqref{eqn:py-diff-pw} and Eq~\eqref{eqn:py-diff-pw-prime} are simply $C\resT_{n, 1}^2$
and $C\resT_{n, 2}^2$).

\subsection{Proof of Lemma \ref{lemma:W-bar-bound}}
\label{sec:lemma-W-bar-bound}
By definition, $\wbar{W}$ is the average of $mn$ independent mean $0$
random variables $\{W_{i, j}\}_{i \in [m], j \in [n]}$. 
By Assumption~{\sf A1}, we know that $\E |W_{i, j}|^2 \le M_1$, and thus
\begin{equation*}
\E \wbar{W}^2 \leq (mn)^{-1} M_1.
\end{equation*}
Hence, Markov's inequality implies for all $t > 0$
\begin{equation}
\label{eqn:upper-bound-bar-W}
\P \left(|\wbar{W}| \geq t (mn)^{-1/2}\right)\leq t^{-2}M_1.
\end{equation}

\subsection{Proof of Lemma \ref{lemma:W-max-bound}}
\label{sec:lemma-W-max-bound}
Note that, Assumption~{\sf A1} gets $\E |W_{i, j}|^2 \le M_1$. Thus, for all $t > 0$, 
\begin{equation}
\label{eqn:upper-bound-max-W}
\P \left(\normmax{\bW} \ge t \right) \le \sum_{i \in [m], j\in [n]}
	\P (|W_{i, j}| \ge t) \le t^{-2}\sum_{i \in [m], j\in [n]} \E |W_{i, j}|^2 
		\le t^{-2} mn M_1. 
\end{equation}

\subsection{Proof of Lemma~\ref{lemma:kernel-pointwise-high-prob}}  
\label{sec:proof-of-prop-kernel-ptwise}
The proof is a simple application of Bernstein's inequality 
(see Lemma~\ref{lemma:bernstein-inequality}). Introduce the centered 
random variables below for $i \in [n]$ (for notational simplicity): 
\begin{equation*}
\wtilde{K}_{h, X_i}(x) \defeq K_{h, X_i}(x) - \E K_{h, X_i}(x).
\end{equation*}
By triangle inequality, the random variables $\wtilde{K}_{h, X_i}(x)$ are bounded
random variables as 
\begin{equation*} 
\left\vert \wtilde{K}_{h, X_i}(x) \right\vert \leq \left\vert K_{h, X_i}(x) \right\vert 
	+ \left\vert \E  K_{h, X_i}(x) \right\vert \leq 2 \norm{K(\cdot)}_\infty.
\end{equation*}
In addition, we have the following upper bound on the second moment of 
$|\wtilde{K}_{h, X_i}(x)|$
\begin{align*}  \label{eqn:second-moment-bound-kernel}
\E |\wtilde{K}_{h, X_i}(x)|^2 &\le \E K_{h, X_i}^2(x) 
	= \int_R K^2\left(\frac{x-w}{h}\right) p_{X_i}(w) \rmd w \\
&= h\int_{\R} K^2(z) p_{X_i}(x-hz) \rmd z \le h\sigma^2 \norm{p_{X_i}(\cdot)}_\infty,
\end{align*}
where the first inequality follows by the fact that $\wtilde{K}_{h, X_i}(x)$ centers 
$K_{h, X_i}(x)$ and the second inequality follows by the definition of $\sigma^2$
and $\norm{p_{X_i}(\cdot)}_\infty$. Now, we apply Bernstein's inequality (i.e,  
Lemma~\ref{lemma:bernstein-inequality}) to the independent random variables 
$\wtilde{K}_{h, X_i}(x)$ and get that, for all $t > 0$, 
\begin{equation*}
\prob \left(|Z_n(x) - \E Z_n(x)| \ge t\right) \le 2 \exp\left(-\frac{nh t^2}{\sigma^2 p_\infty + 
	M_{\infty}t}\right)
\end{equation*}
The desired claim of the lemma now thus follows. 

%


\section{Some useful tools}
\label{sec:preliminary-lemma}

\noindent\noindent 
Lemma~\ref{lemma:MZR-type-inequality} follows from Marcinkiewicz-Zygmund inequality 
and Rosenthal's inequality (see~\cite[Theorem 15.11]{BoucheronLuMa13}). 

\begin{lemma}
\label{lemma:MZR-type-inequality} Let $Z_1, Z_2, \ldots, Z_n$ be independent
random variables such that $\E Z_i = 0$ for $i \in [n]$. Then, for any fix $q \geq 2$, 
there exists some constant $C_q > 0$ depending only on $q$ such that,  
\begin{equation*}
\E \left|\sum_{i=1}^n Z_i\right|^q \leq C_q \left[
	\left(\E \sum_{i=1}^n Z_i^2\right)^{q/2} +  \sum_{i=1}^n \E \left|Z_i\right|^q \right]
\end{equation*}
\end{lemma}
\noindent\noindent
The next lemma is a restatement of Bernstein's inequality.
\begin{lemma}[Bernstein's inequality]
\label{lemma:bernstein-inequality}
Let $Z_1, Z_2, \ldots, Z_n$ be $n$ independent random variables with mean 
$\E [Z_i] = 0$ for all $i \in [n]$. Suppose that, almost surely $|Z_i| \le M$ for 
some (non-random) $M > 0$. Then, for some numerical constant $c > 0$, 
the following inequality holds for all $t > 0$, 
\begin{equation*} 
\P \left( \frac{1}{n}\left\vert \sum_{i=1}^n Z_i \right\vert \geq t \right)
	\leq	2\exp \left(-\frac{cnt^2}{\frac{1}{n}\sum_{i=1}^n \E [Z_i^2] + Mt}  \right).
\end{equation*}
\end{lemma}
\noindent\noindent
The next Lemma~\ref{lemma:operator-to-one-infty} is a standard result in 
matrix analysis~\cite[Thm 2.11]{StewartSun90}. 
\begin{lemma}
\label{lemma:operator-to-one-infty}
For any matrix $A \in \R^{n \times m}$, we have, 
\begin{equation*}
\opnorm{A} \leq \norm{A}_{\ell_1 \to \ell_1}^{1/2} \norm{A}_{\ell_\infty \to \ell_\infty}^{1/2}
	\le \sqrt{mn} \normmax{A}. 
\end{equation*}
\end{lemma}
\noindent\noindent
The next Lemma~\ref{lemma:opnorm-expectation-chen} gives a moment bound
on the operator norm of independent sums of random matrices. 
See~\cite[Theorem A.1(2)]{ChenGiTr12} for a proof. 
\begin{lemma}
\label{lemma:opnorm-expectation-chen}
Let $X_i \in \R^{d \times d}$ be independent and symmetrically distributed Hermitian 
matrices. Then, for $k\ge 2$ and $r \ge \max\{k, 2\log d\}$, 
\begin{equation}
\E\left[ \opnorm{\sum_{i=1}^n X_i}^k\right]^{1/k} \le \sqrt{2e r} 
	\opnorm{\left(\sum_{i=1}^n \E [X_i^2]\right)^{1/2}} + 
	2er \left(\E \max_{i \in [n]} \opnorm{X_i}^k\right)^{1/k}.
\end{equation}
\end{lemma}
\noindent\noindent

\newcommand{\gap}{\gamma}
Recall the following version of Davis-Kahan $\sin \Theta$ theorem~\cite[Theorem 4.4]{StewartSun90}
\begin{lemma}[Davis-Kahan]
\label{lemma:Davis-Kahan}
Let $\hat{\bA} = \bA + \bE \in \R^{p \times q}$. For $s \leq \min(p,q)$, let $\bU$ and $\hat{\bU}$ be the first $s$ columns of the left singular matrices of $A$ and $\hat{A}$ respectively. Denote by $\gap = \sigma_s(\bA) - \sigma_{s+1}(\bA)$. If $\gap \ge 2 \norm{\bE}_{op}$, we have the following inequality
\begin{equation*}
\sqrt{1-\sigma_{\min}^2(\hat{\bU}^{\sT} \bU)} \le \frac{\norm{\bE}_{op}}{\gap -\norm{\bE}_{op}}.
\end{equation*}
\end{lemma}

\end{document}

%% file: andrea_def.tex
\newcommand{\bea}{\begin{eqnarray}}
\newcommand{\eea}{\end{eqnarray}}

\newcommand{\targetpert}{f_{W, \eps}}
\newcommand{\normmax}[1]{\norm{#1}_{\rm max}}
\newcommand{\wtilde}{\widetilde}
\newcommand{\bT}{\boldsymbol{T}}
\newcommand{\res}{{\sf R}}
\newcommand{\resT}{{\sf T}}
\newcommand{\resQ}{{\sf Q}}
\newcommand{\errorone}{\mathcal{E}_1}
\newcommand{\errortwo}{\mathcal{E}_2} 
\newcommand{\org}{^{(0)}}
\newcommand{\bGamma}{\mathbf{\Gamma}}
\newcommand{\subs}{^{\rm sub}}
\newcommand{\pert}{\iota}

\newcommand{\neb}{\mathcal{N}}

\newcommand{\aux}{^{\rm aux}}

\newcommand{\wbar}{\overline}

\newcommand{\law}{\mathcal{L}}
\def\loss{L}

\def\Info{{\rm I}}

\def\eps{{\varepsilon}}

\def\id{{\boldsymbol{I}}}
\def\bh{\boldsymbol{h}}

\def\hsigma{\hat{\sigma}}
\def\hbv{\hat{\boldsymbol v}}

\def\hx{\hat{x}}

\def\bU{{\boldsymbol{U}}}
\def\bV{{\boldsymbol{V}}}
\def\bZ{{\boldsymbol{Z}}}
\def\bSigma{{\boldsymbol{\Sigma}}}
\def\hbSigma{{\boldsymbol{\widehat{\Sigma}}}}
\def\hSigma{{\widehat{\Sigma}}}
\def\bP{{\boldsymbol{P}}}

\newcommand{\set}{\mathcal{S}}

\def\bQ{{\boldsymbol{Q}}}

\def\hbX{{\boldsymbol{\widehat{X}}}}
\def\hbU{{\boldsymbol{\hat{U}}}}
\def\hbV{{\boldsymbol{\hat{V}}}}
\def\hbu{{\boldsymbol{\widehat{u}}}}

\def\bzero{{\mathbf 0}}

\def\cF{{\mathcal F}}

\def\Z{{\mathbb Z}}

\def\opt{\mbox{\tiny\rm opt}}

\def\sub{\mbox{\tiny\rm ub}}

\def\op{\mbox{\tiny\rm op}}

\def\naturals{{\mathbb N}}
\def\reals{{\mathbb R}}

\def\Z{{\mathbb Z}}

\def\normal{{\sf N}}

\def\sT{{\sf T}}

\def\bv{{\boldsymbol{v}}}

\def\bx{{\boldsymbol{x}}}
\def\ba{{\boldsymbol{a}}}

\def\bA{\boldsymbol{A}}
\def\bB{\boldsymbol{B}}

\def\hbx{\boldsymbol{\hat{x}}}

\def\de{{\rm d}}
\def\bX{\boldsymbol{X}}
\def\bY{\boldsymbol{Y}}
\def\bW{\boldsymbol{W}}
\def\bDelta{\boldsymbol{\Delta}}
\def\prob{{\mathbb P}}
\def\E{{\mathbb E}}

\def\<{\langle}
\def\>{\rangle}

\def\rank{{\rm rank}}
\def\diag{{\rm diag}}

\def\hbU{\widehat{\boldsymbol U}}
\def\hbV{\widehat{\boldsymbol V}}

\def\obX{\overline{\boldsymbol X}}
\def\obU{\overline{\boldsymbol U}}
\def\obV{\overline{\boldsymbol V}}

\def\by{{\boldsymbol{y}}}

\def\P{{\sf P}}

\def\be{{\boldsymbol{e}}}

\def\bu{{\boldsymbol{u}}}

\def\Var{{\rm Var}}

\def\obX{\overline{\boldsymbol X}}
\def\obU{\overline{\boldsymbol U}}
\def\obV{\overline{\boldsymbol V}}

\def\hf{\hat{f}}
\def\hp{\hat{p}}
\def\hh{h^{\prime}}
\def\hInfo{\hat{{\rm I}}}

\def\oY{\overline{Y}}
\def\oX{\overline{X}}

%% file: paper_am_5.bbl
\newcommand{\etalchar}[1]{$^{#1}$}
\providecommand{\bysame}{\leavevmode\hbox to3em{\hrulefill}\thinspace}
\providecommand{\MR}{\relax\ifhmode\unskip\space\fi MR }
\providecommand{\MRhref}[2]{%
  \href{http://www.ams.org/mathscinet-getitem?mr=#1}{#2}
}
\providecommand{\href}[2]{#2}
\begin{thebibliography}{PWBM16}

\bibitem[Abb18]{abbe2018community}
Emmanuel Abbe, \emph{Community detection and stochastic block models},
  Foundations and Trends in Communications and Information Theory \textbf{14}
  (2018), no.~1-2, 1--162.

\bibitem[AGZ09]{Guionnet}
Greg~W. Anderson, Alice Guionnet, and Ofer Zeitouni, \emph{{An introduction to
  random matrices}}, Cambridge University Press, 2009.

\bibitem[And55a]{anderson1955integral}
Theodore~W Anderson, \emph{The integral of a symmetric unimodal function over a
  symmetric convex set and some probability inequalities}, Proceedings of the
  American Mathematical Society \textbf{6} (1955), no.~2, 170--176.

\bibitem[And55b]{Anderson55}
\bysame, \emph{The integral of a symmetric unimodal function over a symmetric
  convex set and some probability inequalities}, Proceedings of the American
  Mathematical Society \textbf{6} (1955), no.~2, 170--176.

\bibitem[BAP05]{BBAP05}
Jinho Baik, G{\'e}rard~Ben Arous, and Sandrine P{\'e}ch{\'e}, \emph{Phase
  transition of the largest eigenvalue for nonnull complex sample covariance
  matrices}, The Annals of Probability \textbf{33} (2005), no.~5, 1643--1697.

\bibitem[BDM{\etalchar{+}}16]{barbier2016mutual}
Jean Barbier, Mohamad Dia, Nicolas Macris, Florent Krzakala, Thibault Lesieur,
  and Lenka Zdeborov{\'a}, \emph{Mutual information for symmetric rank-one
  matrix estimation: A proof of the replica formula}, Advances in Neural
  Information Processing Systems, 2016, pp.~424--432.

\bibitem[BEK{\etalchar{+}}14]{BloemendalKnYaYi14}
Alex Bloemendal, L{\'a}szl{\'o} Erdos, Antti Knowles, Horng-Tzer Yau, and Jun
  Yin, \emph{Isotropic local laws for sample covariance and generalized wigner
  matrices}, Electron. J. Probab \textbf{19} (2014), no.~33, 1--53.

\bibitem[BGN12]{BGN12}
Florent Benaych-Georges and Raj~Rao Nadakuditi, \emph{The singular values and
  vectors of low rank perturbations of large rectangular random matrices},
  Journal of Multivariate Analysis (2012), 120--135.

\bibitem[BKYY16]{BloemendalKnYa16}
Alex Bloemendal, Antti Knowles, Horng-Tzer Yau, and Jun Yin, \emph{On the
  principal components of sample covariance matrices}, Probability theory and
  related fields \textbf{164} (2016), no.~1-2, 459--552.

\bibitem[BL08]{bickel2008covariance}
Peter~J Bickel and Elizaveta Levina, \emph{Covariance regularization by
  thresholding}, The Annals of Statistics \textbf{36} (2008), no.~6,
  2577--2604.

\bibitem[BLM13]{BoucheronLuMa13}
St{\'e}phane Boucheron, G{\'a}bor Lugosi, and Pascal Massart,
  \emph{Concentration inequalities: A nonasymptotic theory of independence},
  Oxford university press, 2013.

\bibitem[BS06]{BS06}
Jinho Baik and Jack~W Silverstein, \emph{Eigenvalues of large sample covariance
  matrices of spiked population models}, Journal of Multivariate Analysis
  \textbf{6} (2006), 1382--1408.

\bibitem[BS10a]{bai2010spectral}
Zhidong Bai and Jack~W Silverstein, \emph{Spectral analysis of large
  dimensional random matrices}, vol.~20, Springer, 2010.

\bibitem[BS10b]{BaiSi10}
\bysame, \emph{Spectral analysis of large dimensional random matrices},
  vol.~20, Springer, 2010.

\bibitem[CGT12]{ChenGiTr12}
Richard~Y Chen, Alex Gittens, and Joel~A Tropp, \emph{The masked sample
  covariance estimator: an analysis using matrix concentration inequalities},
  Information and Inference: A Journal of the IMA \textbf{1} (2012), no.~1,
  2--20.

\bibitem[Cha15]{chatterjee2015matrix}
Sourav Chatterjee, \emph{Matrix estimation by universal singular value
  thresholding}, The Annals of Statistics \textbf{43} (2015), no.~1, 177--214.

\bibitem[CMW13]{cai2013sparse}
T~Tony Cai, Zongming Ma, and Yihong Wu, \emph{{Sparse PCA: Optimal rates and
  adaptive estimation}}, The Annals of Statistics \textbf{41} (2013), no.~6,
  3074--3110.

\bibitem[CP10]{candes2010matrix}
Emmanuel~J Candes and Yaniv Plan, \emph{Matrix completion with noise},
  Proceedings of the IEEE \textbf{98} (2010), no.~6, 925--936.

\bibitem[CZZ10]{cai2010optimal}
T~Tony Cai, Cun-Hui Zhang, and Harrison~H Zhou, \emph{Optimal rates of
  convergence for covariance matrix estimation}, The Annals of Statistics
  \textbf{38} (2010), no.~4, 2118--2144.

\bibitem[DAM16]{deshpande2016asymptotic}
Yash Deshpande, Emmanuel Abbe, and Andrea Montanari, \emph{Asymptotic mutual
  information for the balanced binary stochastic block model}, Information and
  Inference: A Journal of the IMA \textbf{6} (2016), no.~2, 125--170.

\bibitem[DG14]{donoho2014minimax}
David Donoho and Matan Gavish, \emph{Minimax risk of matrix denoising by
  singular value thresholding}, The Annals of Statistics \textbf{42} (2014),
  no.~6, 2413--2440.

\bibitem[DGJ18]{donoho2018optimal}
David~L Donoho, Matan Gavish, and Iain~M Johnstone, \emph{Optimal shrinkage of
  eigenvalues in the spiked covariance model}, Annals of Statistics \textbf{46}
  (2018), no.~4, 1742.

\bibitem[dGJL05]{d2005direct}
Alexandre d'Aspremont, Laurent~E Ghaoui, Michael~I Jordan, and Gert~R
  Lanckriet, \emph{{A direct formulation for sparse PCA using semidefinite
  programming}}, Advances in neural information processing systems, 2005,
  pp.~41--48.

\bibitem[Din17]{Ding17}
Xiucai Ding, \emph{High dimensional deformed rectangular matrices with
  applications in matrix denoising}, {\sf arXiv:1702.06975} (2017).

\bibitem[DM14]{deshpande2014information}
Yash Deshpande and Andrea Montanari, \emph{{Information-theoretically optimal
  sparse PCA}}, Information Theory (ISIT), 2014 IEEE International Symposium
  on, IEEE, 2014, pp.~2197--2201.

\bibitem[DM15]{deshpande2015finding}
\bysame, \emph{{Finding Hidden Cliques of Size $\sqrt{N/e}$ in Nearly Linear
  Time}}, Foundations of Computational Mathematics \textbf{15} (2015), no.~4,
  1069--1128.

\bibitem[Duc18]{Duchi18}
John Duchi, \emph{{Theoretical Statistics. Lecture Notes on Contiguity and
  Asymptotics}}, 2018, available at
  {https://stanford.edu/class/stats300b/Notes/contiguity-and-asymptotics.pdf}.

\bibitem[DW18]{DobribanWa18}
Edgar Dobriban and Stefan Wager, \emph{High-dimensional asymptotics of
  prediction: Ridge regression and classification}, The Annals of Statistics
  \textbf{46} (2018), no.~1, 247--279.

\bibitem[EK08]{el2008operator}
Noureddine El~Karoui, \emph{Operator norm consistent estimation of
  large-dimensional sparse covariance matrices}, The Annals of Statistics
  \textbf{36} (2008), no.~6, 2717--2756.

\bibitem[GD14]{gavish2014optimal}
Matan Gavish and David~L Donoho, \emph{{The optimal hard threshold for singular
  values is $4/\sqrt{3}$}}, IEEE Transactions on Information Theory \textbf{60}
  (2014), no.~8, 5040--5053.

\bibitem[JL09]{johnstone2009consistency}
Iain~M Johnstone and Arthur~Yu Lu, \emph{On consistency and sparsity for
  principal components analysis in high dimensions}, Journal of the American
  Statistical Association \textbf{104} (2009), no.~486, 682--693.

\bibitem[JM13]{javanmard2013localization}
Adel Javanmard and Andrea Montanari, \emph{Localization from incomplete noisy
  distance measurements}, Foundations of Computational Mathematics \textbf{13}
  (2013), no.~3, 297--345.

\bibitem[JMRT16]{javanmard2016phase}
Adel Javanmard, Andrea Montanari, and Federico Ricci-Tersenghi, \emph{Phase
  transitions in semidefinite relaxations}, Proceedings of the National Academy
  of Sciences \textbf{113} (2016), no.~16, E2218--E2223.

\bibitem[KMO10]{keshavan2010matrix}
Raghunandan~H Keshavan, Andrea Montanari, and Sewoong Oh, \emph{Matrix
  completion from noisy entries}, Journal of Machine Learning Research
  \textbf{11} (2010), no.~Jul, 2057--2078.

\bibitem[KXZ16]{krzakala2016mutual}
Florent Krzakala, Jiaming Xu, and Lenka Zdeborov{\'a}, \emph{Mutual information
  in rank-one matrix estimation}, Information Theory Workshop (ITW), 2016 IEEE,
  IEEE, 2016, pp.~71--75.

\bibitem[KY17]{knowles2017anisotropic}
Antti Knowles and Jun Yin, \emph{Anisotropic local laws for random matrices},
  Probability Theory and Related Fields \textbf{169} (2017), no.~1-2, 257--352.

\bibitem[LCY12]{LeYa12}
Lucien Le~Cam and Grace~Lo Yang, \emph{Asymptotics in statistics: some basic
  concepts}, Springer Science \& Business Media, 2012.

\bibitem[LKZ15]{lesieur2015mmse}
Thibault Lesieur, Florent Krzakala, and Lenka Zdeborov{\'a}, \emph{Mmse of
  probabilistic low-rank matrix estimation: Universality with respect to the
  output channel}, Communication, Control, and Computing (Allerton), 2015 53rd
  Annual Allerton Conference on, IEEE, 2015, pp.~680--687.

\bibitem[LKZ17]{lesieur2017constrained}
\bysame, \emph{Constrained low-rank matrix estimation: Phase transitions,
  approximate message passing and applications}, Journal of Statistical
  Mechanics: Theory and Experiment \textbf{2017} (2017), no.~7, 073403.

\bibitem[LM16]{lelarge2016fundamental}
Marc Lelarge and L{\'e}o Miolane, \emph{Fundamental limits of symmetric
  low-rank matrix estimation}, {\sf arXiv:1611.03888} (2016).

\bibitem[Mio17]{miolane2017fundamental}
L{\'e}o Miolane, \emph{Fundamental limits of low-rank matrix estimation}, {\sf
  arXiv:1702.00473} (2017).

\bibitem[Moo17]{moore2017computer}
Cristopher Moore, \emph{The computer science and physics of community
  detection: Landscapes, phase transitions, and hardness}, Bulletin of EATCS
  \textbf{1} (2017), no.~121.

\bibitem[MV17]{montanari2017estimation}
Andrea Montanari and Ramji Venkataramanan, \emph{Estimation of low-rank
  matrices via approximate message passing}, {\sf arXiv:1711.01682} (2017).

\bibitem[OMH13]{onatski2013asymptotic}
Alexei Onatski, Marcelo~J Moreira, and Marc Hallin, \emph{Asymptotic power of
  sphericity tests for high-dimensional data}, The Annals of Statistics
  \textbf{41} (2013), no.~3, 1204--1231.

\bibitem[Pau07]{paul2007asymptotics}
Debashis Paul, \emph{Asymptotics of sample eigenstructure for a large
  dimensional spiked covariance model}, Statistica Sinica (2007), 1617--1642.

\bibitem[PWBM16]{perry2016optimality}
Amelia Perry, Alexander~S Wein, Afonso~S Bandeira, and Ankur Moitra,
  \emph{Optimality and sub-optimality of pca for spiked random matrices and
  synchronization}, {\sf arXiv:1609.05573} (2016).

\bibitem[Rem13]{Remmert13}
Reinhold Remmert, \emph{Classical topics in complex function theory}, vol. 172,
  Springer Science \& Business Media, 2013.

\bibitem[Rin08]{ringner2008principal}
Markus Ringn{\'e}r, \emph{What is principal component analysis?}, Nature
  biotechnology \textbf{26} (2008), no.~3, 303.

\bibitem[Sin08]{singer2008remark}
Amit Singer, \emph{A remark on global positioning from local distances},
  Proceedings of the National Academy of Sciences \textbf{105} (2008), no.~28,
  9507--9511.

\bibitem[Sin11]{singer2011angular}
\bysame, \emph{Angular synchronization by eigenvectors and semidefinite
  programming}, Applied and computational harmonic analysis \textbf{30} (2011),
  no.~1, 20--36.

\bibitem[SN13]{shabalin2013reconstruction}
Andrey~A Shabalin and Andrew~B Nobel, \emph{Reconstruction of a low-rank matrix
  in the presence of gaussian noise}, Journal of Multivariate Analysis
  \textbf{118} (2013), 67--76.

\bibitem[SS90]{StewartSun90}
Gilbert~W Stewart and Ji-Guang Sun, \emph{Matrix perturbation theory (computer
  science and scientific computing)}, Academic Press Boston, 1990.

\bibitem[Ver12]{vershynin2012close}
Roman Vershynin, \emph{How close is the sample covariance matrix to the actual
  covariance matrix?}, Journal of Theoretical Probability \textbf{25} (2012),
  no.~3, 655--686.

\bibitem[WS13]{wang2013exact}
Lanhui Wang and Amit Singer, \emph{Exact and stable recovery of rotations for
  robust synchronization}, Information and Inference: A Journal of the IMA
  \textbf{2} (2013), no.~2, 145--193.

\bibitem[ZHT06]{zou2006sparse}
Hui Zou, Trevor Hastie, and Robert Tibshirani, \emph{Sparse principal component
  analysis}, Journal of computational and graphical statistics \textbf{15}
  (2006), no.~2, 265--286.

\end{thebibliography}
